\documentclass[a4paper,11pt,nocut]{article}
\usepackage[top=2.5cm, bottom=2.5cm, left=2cm, right=2cm]{geometry}
\usepackage{settings_custom}
\usepackage[sort&compress,numbers]{natbib}
\allowdisplaybreaks

\begin{document}

\title{Exact threshold for approximate ellipsoid fitting of random points}
\date{\today}
\author{Afonso S.\ Bandeira$^\star$, Antoine Maillard$^{\circ, \star, \diamond}$}
\maketitle

{\let\thefootnote\relax\footnote{
    \noindent
$\circ$ Inria Paris, DI ENS, PSL University, Paris, France. \\
$\star$ Department of Mathematics, ETH Z\"urich, Switzerland.\\
$\diamond$ To whom correspondence shall be sent: \href{mailto:antoine.maillard@inria.fr}{antoine.maillard@inria.fr}.
}}
\setcounter{footnote}{0}

\begin{abstract}
    We consider the problem $\mathrm{(P)}$ of exactly fitting an ellipsoid (centered at $0$) to $n$ standard Gaussian random vectors in $\bbR^d$, as $n, d \to \infty$ with $n / d^2 \to \alpha > 0$.
    This problem is conjectured to undergo a sharp transition: with high probability, $\mathrm{(P)}$ has a solution if $\alpha < 1/4$, while $\mathrm{(P)}$ has no solutions
    if $\alpha > 1/4$.
    So far, only a trivial bound $\alpha > 1/2$ is known to imply the absence of solutions, while the sharpest results on the positive side assume $\alpha \leq \eta$ (for $\eta > 0$ a 
    small constant) to prove that $\mathrm{(P)}$ is solvable.
    In this work we show a universality property for the minimal fitting error achievable by ellipsoids: 
    we show that, to leading order, it coincides with the minimal error in a so-called ``Gaussian equivalent'' problem, 
    for which the satisfiability transition can be rigorously analyzed.
    Our main results follow from this finding, and they are twofold.
    On the positive side, we prove that if $\alpha < 1/4$, there exists an ellipsoid fitting all the points
    up to a small error, and that the lengths of its principal axes are bounded above and below.
    On the other hand, for $\alpha > 1/4$, we show that achieving small fitting error is not possible 
    if the length of the ellipsoid's shortest axis does not approach $0$ as $d \to \infty$ (and in particular there does not exist any 
    ellipsoid fit whose shortest axis length is bounded away from $0$ as $d \to \infty$).
    To the best of our knowledge, our work is the first rigorous result characterizing the expected phase transition in ellipsoid fitting at $\alpha = 1/4$.
    In a companion non-rigorous work, the second author and D.\ Kunisky give a general analysis of ellipsoid fitting using the replica method 
    of statistical physics, which inspired the present work.  
\end{abstract}

\section{Introduction and main results}\label{sec:introduction}
\subsection{The ellipsoid fitting conjecture}

\noindent
We consider the \emph{random ellipsoid fitting problem}: given $n$ random standard Gaussian vectors in dimension $d$, 
when do they all lie on the boundary of a (centered) ellipsoid?
Formally, we define an ellipsoid fit using the set $\mcS_d$ of $d \times d$ real symmetric matrices, as follows.
\begin{definition}[Ellipsoid fit]\label{def:ellipsoid_fit}
    \noindent
    Let $x_1, \cdots, x_n \in \bbR^d$. We say that $S \in \mcS_d$ is an \emph{ellipsoid fit} for $(x_\mu)_{\mu=1}^n$ 
    if it satisfies:
    \begin{equation}\label{eq:def_P}
   ({\rm P}) \, : \, \begin{dcases} 
    &x_\mu^\T S x_\mu = d \, \textrm{ for all } \mu \in \{1, \cdots, n\}, \\ 
    &S \succeq 0.
   \end{dcases}  
\end{equation}
\end{definition}
\noindent
In Definition~\ref{def:ellipsoid_fit}, the matrix $S \succeq 0$ defines the ellipsoid $\Sigma \coloneqq \{x \in \bbR^d \, : \, x^\T S x = d\}$.
Geometrically speaking, the eigenvectors of $S$ give the directions of the principal axes of the ellipsoid, while 
its eigenvalues $(\lambda_i)_{i=1}^d$ are related to the lengths $(r_i)_{i=1}^d$ of its principal (semi-)axes by
$r_i = \sqrt{d} \lambda_i^{-1/2}$.

\myskip 
\textbf{Scaling --}
In what follows, we will rather refer to the rescaled quantities $r'_i = r_i / \sqrt{d}$ as the lengths of the ellipsoid axes, 
effectively rescaling distances so that the sphere of radius $\sqrt{d}$ (with $S = \Id_d$) has all (semi-)axes of length $1$.
In particular, the lengths of the ellipsoid's longest and shortest axis are then
respectively $\lambda_{\min}(S)^{-1/2}$ and $\lambda_{\max}(S)^{-1/2}$.

\myskip
We are interested in finding an ellipsoid fit to a set of \emph{random} points $x_1, \cdots, x_n \iid \mcN(0, \Id_d)$.
The question of the existence of such an ellipsoid arose first in \cite{saunderson2011subspace,saunderson2012diagonal,saunderson2013diagonal}, 
which conjectured the following (see e.g.\ Conjecture~2 in \cite{saunderson2013diagonal} or Conjecture~1.1 in \cite{potechin2023near}).
\begin{conjecture}[The ellipsoid fitting conjecture]\label{conj:ellipsoid_fitting}
    \noindent
    Let $n, d \geq 1$, and $x_1, \cdots, x_n$ be drawn i.i.d. from $\mcN(0, \Id_d)$.
    Let 
    \begin{equation*}
        p(n, d) \coloneqq \bbP[\exists S \in \mcS_d \textrm{ an ellipsoid fit for }(x_\mu)_{\mu=1}^n].
    \end{equation*}
    For any $\eps > 0$, the following holds:
    \begin{align}
        \label{eq:conj_ef_positive}
        \limsup_{d \to \infty} \frac{n}{d^2} \leq \frac{1 - \eps}{4} &\Rightarrow \lim_{d \to \infty} p(n, d) = 1, \\ 
        \label{eq:conj_ef_negative}
        \liminf_{d \to \infty} \frac{n}{d^2} \geq \frac{1 + \eps}{4} &\Rightarrow \lim_{d \to \infty} p(n, d) = 0. 
    \end{align}
\end{conjecture}
\noindent
Informally, Conjecture~\ref{conj:ellipsoid_fitting} predicts a sharp transition for the existence of an ellipsoid fit 
in the regime $n / d^2 \to \alpha > 0$ exactly at $\alpha = 1/4$.

\subsection{Related works}

\noindent
Conjecture~\ref{conj:ellipsoid_fitting} was first stated and studied in the series of works \cite{saunderson2011subspace,saunderson2012diagonal,saunderson2013diagonal}, where it arose as being connected to 
the decomposition of a (random) data matrix $M$ as $M = L + D$, with $L \succeq 0$ being low-rank, and $D$ a diagonal matrix.
Connections to other problems throughout theoretical computer science have since then been unveiled,
such as certifying a lower bound on the discrepancy of a random matrix using a canonical semidefinite relaxation \cite{saunderson2012diagonal,potechin2023near}, 
overcomplete independent component analysis \cite{podosinnikova2019overcomplete}, or Sum-of-Squares lower bound for the Sherrington-Kirkpatrick Hamiltonian \cite{ghosh2020sum}.
We refer the reader to the detailed expositions of \cite{potechin2023near,maillard2024fitting} 
on the connections of ellipsoid fitting to theoretical computer science and machine learning.

\myskip
Interestingly, Conjecture~\ref{conj:ellipsoid_fitting} arose both from numerical evidence%
\footnote{Given $(x_\mu)_{\mu=1}^n$, eq.~\eqref{eq:def_P} is a convex problem (it is an example of a semidefinite program, or SDP), which is efficiently solvable when solutions exist.}
and the remark that $d^2/4$ is known to be the statistical dimension (or squared Gaussian width) of $\mcS_d^+$, the set of positive semidefinite matrices \cite{chandrasekaran2012convex,amelunxen2014living}.
As such, if one replaces eq.~\eqref{eq:def_P} by 
\begin{equation}\label{eq:def_P_Gaussian}
    ({\rm P}_{\Gauss}) \, : \, \begin{dcases} 
    &\Tr[S G_\mu] = d \, \textrm{ for all } \mu \in \{1, \cdots, n\}, \\ 
    &S \succeq 0,
\end{dcases}
\end{equation}
in which $(G_\mu)_{\mu=1}^n$ are (independent) standard Gaussian matrices, 
Conjecture~\ref{conj:ellipsoid_fitting} provably holds for $(\rm P_\Gauss)$.
The crucial property of $(\rm P_\Gauss)$ behind this result is that the affine subspace 
$\{S \in \mcS_d \, : \, (\Tr[S G_\mu] = d)_{\mu=1}^n\}$ is randomly oriented, \emph{uniformly} in all directions.
Although this motivation for the conjecture was known,
our work is (to the best of our knowledge) 
the first mathematically rigorous approach to leverage the connection between $(\rm P)$ and 
$(\rm P_\Gauss)$.

\myskip
Indeed, previous progress on Conjecture~\ref{conj:ellipsoid_fitting} has mostly focused on proving the existence of a fitting ellipsoid
using an ansatz solution:
the first line of eq.~\eqref{eq:def_P} defines an affine subspace $V$ of symmetric matrices of codimension $n$, 
so one can study a well-chosen $S^\star \in V$,
and argue that for small enough $n$ it satisfies $S^\star \succeq 0$ with high probability as $d \to \infty$.
Various such constructions have been used, and we summarize in Fig.~\ref{fig:summary} the current rigorous progress on the ellipsoid fitting conjecture that arose from these approaches.
\begin{figure}[t]
  \centering
\begin{tikzpicture}

  \draw[->,thick,line width = 1.5pt] (-2.3,0) -- (11.25,0) node[below=5pt,font=\large] {$n$};

  \pgfmathsetmacro\tickPositionA{-1.9}
  \pgfmathsetmacro\tickPositionE{-0.25}
  \pgfmathsetmacro\tickPositionF{1.75}
  \pgfmathsetmacro\tickPositionB{4}
  \pgfmathsetmacro\tickPositionC{7}
  \pgfmathsetmacro\tickPositionD{10}
  \def\tickLabels{{{"$d^{6/5-\eps}$","\phantom{}\cite{saunderson2012diagonal}"},{"$d^{3/2-\eps}$","\phantom{}\cite{ghosh2020sum}"},{"$d^2/\plog(d)$","\phantom{}\cite{potechin2023near,kane2023nearly}"},{"$d^2/C$","\phantom{}\cite{bandeira2023fitting,hsieh2023ellipsoid,tulsiani2023ellipsoid}"}, {"$d^2/4$", "\phantom{}"}, {"$d^2/2$","(Trivial)"}}}

  \foreach \pos [count=\i] in {\tickPositionA, \tickPositionE, \tickPositionF, \tickPositionB, \tickPositionC, \tickPositionD} {
    \draw[line width=1.5pt] (\pos,0.1) -- (\pos,-0.1) node[below] {\pgfmathparse{\tickLabels[\i-1][0]}\pgfmathresult};
    \node at (\pos,-0.7) [below, anchor=north] {\pgfmathparse{\tickLabels[\i-1][1]}\pgfmathresult};
  }

  \fill[green, opacity=0.4] (-2.5,0.2) rectangle (\tickPositionB, -0.1);
  \fill[yellow, opacity=0.4] (\tickPositionB,0.2) rectangle (\tickPositionC, -0.1);
  \fill[orange, opacity=0.4] (\tickPositionC,0.2) rectangle (\tickPositionD, -0.1);
  \fill[red, opacity=0.4] (\tickPositionD,0.2) rectangle (11.1, -0.1);
%
  \node[green!75!black, anchor=south, font = \small] at ({(-2 + \tickPositionB)/2},0.65) {SAT};
  \node[green!75!black, anchor=south, font = \small] at ({(-2 + \tickPositionB)/2},0.25) {(rigorous)};
  \node[yellow!80!black, anchor=south, font = \small] at ({\tickPositionB+(\tickPositionC-\tickPositionB)/2},0.65) {SAT};
  \node[yellow!80!black, anchor=south, font = \small] at ({\tickPositionB+(\tickPositionC-\tickPositionB)/2},0.25) {(conjecture)};
  \node[orange!100!black, anchor=south,font = \small] at ({\tickPositionC+(\tickPositionD-\tickPositionC)/2},0.65) {UNSAT};
  \node[orange!100!black, anchor=south,font = \small] at ({\tickPositionC+(\tickPositionD-\tickPositionC)/2},0.25) {(conjecture)};
  \node[red!100!black, anchor=south,font = \small] at ({\tickPositionD+(11.25-\tickPositionD)/2},0.65) {UNSAT};
  \node[red!100!black, anchor=south,font = \small] at ({\tickPositionD+(11.25-\tickPositionD)/2},0.25) {(rigorous)};

  \end{tikzpicture}
\caption{
\label{fig:summary}    
A summary of the current state of the ellipsoid fitting conjecture. 
In red, we show regions for which ellipsoid fitting is rigorously known to be unsatisfiable (UNSAT), and in orange regions which are conjectured to be.
Similarly, we show in green regions rigorously known to be satisfiable (SAT), and in yellow regions which are conjectured to be so.
Figure is taken from~\cite{maillard2024fitting}.
}
\end{figure}
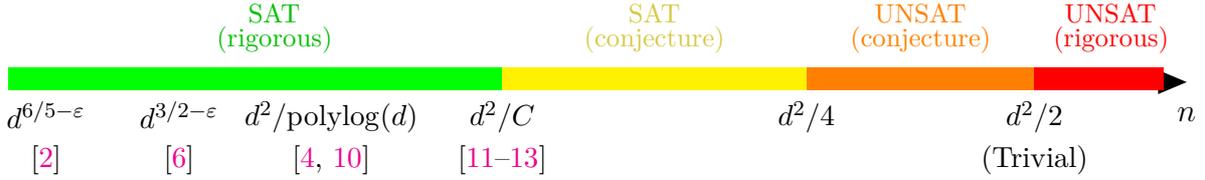
Presently, the best rigorous results on Conjecture~\ref{conj:ellipsoid_fitting} 
are due to the recent works \cite{bandeira2023fitting,hsieh2023ellipsoid,tulsiani2023ellipsoid} and can be summarized as follows:
\begin{theorem}[\cite{bandeira2023fitting,hsieh2023ellipsoid,tulsiani2023ellipsoid}]
    \label{thm:previous_results}
    \noindent
    Let $n, d \geq 1$, and $x_1, \cdots, x_n \iid \mcN(0, \Id_d)$.
    Let $p(n, d) \coloneqq \bbP[\exists S \in \mcS_d \textrm{ an ellipsoid fit for }(x_\mu)_{\mu=1}^n]$.
    There exists a (small) universal constant $\eta > 0$ such that:
    \begin{equation*}
        \limsup_{d \to \infty} \frac{n}{d^2} \leq \eta \Rightarrow \lim_{d \to \infty} p(n, d) = 1, \\ 
    \end{equation*}
    Moreover, if $n > d(d+1)/2$, then $p(n, d) = 0$.
\end{theorem}
\noindent
Note that the bound $n > d(d+1)/2$ in Theorem~\ref{thm:previous_results} arises from a simple dimension counting argument, 
as $d(d+1)/2$ is the dimension of the space of symmetric matrices: for such values of $n$, not only does there not exist 
a solution to eq.~\eqref{eq:def_P}, there does not exist any solution even without the constraint $S \succeq 0$! 

\myskip
\textbf{Statistical physics approaches: heuristic and rigorous --}
In this work, we tackle Conjecture~\ref{conj:ellipsoid_fitting} using techniques inspired by the statistical physics of disordered systems. 
While analytical methods developed in this field were originally designed to study models known as spin glasses \cite{anderson1989spin,mezard1987spin},
they have seen in the past decades a great number of applications in high dimensional statistics, theoretical computer science, and machine learning. 
Moreover, despite these techniques often being non-rigorous, a growing line of mathematics literature has emerged establishing many of their predictions.
Notably, ellipsoid fitting is an example of a \emph{semidefinite program (SDP)}\footnote{i.e.\ a combination of linear equations with a positivity constraint $S \succeq 0$.}
with random linear constraints, and some such SDPs have been previously analyzed with tools of statistical physics \cite{montanari2016semidefinite,javanmard2016phase},
although the methods of these works fall short for analyzing the satisfiability transition in random ellipsoid fitting~\cite{maillard2024fitting}. 
We refer the interested reader to the recent book \cite{charbonneau2023spin} that compiles many (sometimes surprising) applications of the
theory of disordered systems, as well as mathematically rigorous approaches to it.

\myskip
Notably, in the companion work to our manuscript \cite{maillard2024fitting},
non-rigorous methods of statistical physics are employed to provide a 
detailed picture of the satisfiability transition in random ellipsoid fitting.
Besides predicting a threshold for $n \sim d^2/4$, this work gives analytical formulas for the typical shape of ellipsoid fits in the satisfiable phase (i.e.\ the spectral density of $S$), 
generalizes these predictions for non-Gaussian but rotationally-invariant vectors $\{x_\mu\}_{\mu=1}^n$, 
and also studies the performance of different explicit solutions, notably ones used in the previous literature (see Fig.~\ref{fig:summary}).
We emphasize that the present paper is, in contrast, mathematically rigorous.

\myskip 
\textbf{Inspiration of our approach --}
Importantly, the non-rigorous analysis of \cite{maillard2024fitting} suggests that a quantity known as the free entropy (or free energy) in statistical physics, 
is universal: its value is (with high probability) the same for $(\rm P)$ and a variant of $(\rm P_\Gauss)$, as $d \to \infty$.
Such a universality property would have major consequences, as the free entropy carries deep information 
about the structure of the space of solutions to the problem.
Remarkably, similar phenomena have been studied numerically and theoretically in statistical learning models, in which one can effectively replace an arbitrary (and possibly complicated) data distribution 
by its ``Gaussian equivalent''.
Investigating this Gaussian equivalence phenomenon is the object of a recent and very active line of work, 
with consequences on the theory of empirical risk minimization and beyond \cite{goldt2022gaussian,hu2022universality,montanari2022universality,gerace2024gaussian,adamczak2011restricted,dandi2023universality,loureiro2021learning,dhifallah2020precise,schroder2023deterministic}. 
Inspired by these works (in particular \cite{montanari2022universality}) we provide a rigorous proof 
of the universality conjectured in \cite{maillard2024fitting}, using an interpolation argument.
We then leverage tools of the theory of random convex programs \cite{chandrasekaran2012convex,amelunxen2014living}, such as Gordon's min-max inequality \cite{gordon1988milman},
to sharply characterize the space of solutions to $(\rm P_\Gauss)$. Using the aforementioned universality allows to transfer 
many of these conclusions to the original problem $(\rm P)$, yielding our main results.

\subsection{Main results}

\noindent
We now state our main results, separating the conjecturally satisfiable ($\alpha = n/d^2 < 1/4$) and unsatisfiable ($\alpha > 1/4$) regimes.

\myskip
\textbf{Notation --}
$f = \smallO_d(g)$ (respectively $f = \mcO_d(g)$) means that $f/g \to 0$ as $d \to \infty$ (respectively $f/g$ is bounded as $d \to \infty$). 
We also use $g \gtrsim f$ to denote $f = \mcO_d(g)$.
We denote $\mcS_d$ the set of $d \times d$ real symmetric matrices, while 
$\bbS^{d-1}(r)$ refers to the Euclidean sphere of radius $r$ in $\bbR^d$.
For $S \in \mcS_d$, $\Sp(S) = \{\lambda_i\}_{i=1}^d$ is the set of eigenvalues of $S$.
For $\gamma \in [1, \infty]$, $\|S\|_{S_\gamma} \coloneqq (\sum_i |\lambda_i|^\gamma)^{1/\gamma}$ stands for the Schatten-$\gamma$ norm.
$B_\gamma(S, \delta)$ is the Schatten-$\gamma$ ball of radius $\delta$ centered in $S$, and $B_\gamma(\delta)$ the ball centered at $S = 0$.
We denote by $\|S\|_{\op} \coloneqq \|S\|_{S_\infty}$ the operator norm, and $\|S\|_F \coloneqq \|S\|_{S_2}$ the Frobenius norm.
For a function $\psi : \bbR \to \bbR$, we write $\|\psi\|_L$ to denote its Lipschitz constant.
Finally, we use a generic nomenclature $C, c > 0$ to denote positive constants (not depending on the dimension), that may vary from line to line. 
If necessary, we will make explicit the dependency of these constants on parameters of the problems.
Finally, we use in the text the abbreviation ``w.h.p.'', short for ``with high probability'', to refer to events that 
have probability $1 - \smallO_d(1)$ as $d \to \infty$.

\subsubsection{The satisfiable phase: \texorpdfstring{$\alpha < 1/4$}{}}

Our main result on the ``positive'' side of the ellipsoid fitting conjecture can be stated as follows.
\begin{theorem}[Satisfiable regime]\label{thm:main_positive_side}
    \noindent
    Assume $\alpha \coloneqq \limsup (n/d^2) < 1/4$ and let $r \in [1,4/3)$.
    There exist $0 < \lambda_- \leq \lambda_+$, depending only on $\alpha$, 
    such that the following holds.
    Let 
    \begin{equation*}
        \Gamma_r(\eps) \coloneqq \Bigg\{S \in \mcS_d : \Sp(S) \subseteq [\lambda_-, \lambda_+] \textrm{ and } \frac{1}{n} \sum_{\mu=1}^n \left|\sqrt{d} \left[\frac{x_\mu^\T S x_\mu}{d} - 1\right]\right|^r \leq \eps \Bigg\}.
    \end{equation*}
    Then for any $\eps > 0$, if $x_1, \cdots, x_n \iid \mcN(0, \Id_d)$, $\bbP[\Gamma_r(\eps) \neq \emptyset] \to 1$ as $n, d \to \infty$.
\end{theorem}
\noindent
Let us make a series of remarks on Theorem~\ref{thm:main_positive_side}.
First, one can alternatively formulate its conclusion as: 
\begin{equation}\label{eq:main_positive_side_alternate}
    \plim_{d \to \infty} \min_{\Sp(S) \subseteq [\lambda_-, \lambda_+]} \frac{1}{n} \sum_{\mu=1}^n \left|\sqrt{d} \left[\frac{x_\mu^\T S x_\mu}{d} - 1\right]\right|^r = 0,
\end{equation}
where $\plim$ denotes limit in probability.
Secondly, while our current proof limits the choice of $r \in [1, 4/3)$,  it might be possible to refine our arguments
to reach the same result for any $r \in [1,2]$, see our discussion in Section~\ref{subsec:generalizations}.
Moreover, note that by standard concentration arguments, we expect Gaussian points to be close to the sphere $\bbS^{d-1}(\sqrt{d})$, 
i.e.\ the ellipsoid defined by $S = \Id_d$.
A detailed analysis yields, for any $r \in [1,2]$:
\begin{equation}\label{eq:error_identity}
    \plim_{d \to \infty} \frac{1}{n} \sum_{\mu=1}^n \left|\sqrt{d} \left[\frac{\|x_\mu\|^2}{d} - 1\right]\right|^r = \EE[|Z|^r] > 0,
\end{equation}
where $Z \sim \mcN(0,2)$.
This follows from classical concentration arguments, see Appendix~\ref{sec_app:error_identity} for a detailed derivation.
Theorem~\ref{thm:main_positive_side} therefore shows that there exists an ellipsoid whose ``fitting error'' 
improves by an arbitrary factor over the one achieved by the unit sphere, as long as $\alpha < 1/4$.
On the other hand, we will see that this is not possible for $\alpha > 1/4$, strongly suggesting that our results capture the phenomenon responsible for the conjectured satisfiability transition of ellipsoid fitting.
In Section~\ref{subsec:generalizations} we will consider possible future directions that could allow to 
improve our results to the existence of fitting ellipsoids with \emph{exactly zero} error, i.e.\ the conjecture of eq.~\eqref{eq:conj_ef_positive}.

\myskip
Finally, we notice that Theorem~\ref{thm:main_positive_side} is coherent with the non-rigorous analysis of \cite{maillard2024fitting},
which predicts that for any $\alpha < 1/4$ typical solutions to ellipsoid fitting have spectral density contained in an interval of the type $[\lambda_-, \lambda_+]$ depending only on $\alpha$.

\subsubsection{The unsatisfiable phase: \texorpdfstring{$\alpha > 1/4$}{}}

Our main result towards proving the non-existence of fitting ellipsoids for $\alpha > 1/4$ is the following.
\begin{theorem}[Unsatisfiable regime]\label{thm:main_negative_side}
    \noindent
    Assume $\alpha \coloneqq \liminf (n/d^2) > 1/4$.
    Let $\phi: \bbR_+ \to \bbR_+$ be a non-decreasing differentiable function, with $\phi(0) = 0$, and such that $\phi$ has a unique global minimum in $0$.
    For any $\eps > 0$ and $M > 0$ we  let:
    \begin{equation*}
        \Gamma(\eps, M) \coloneqq \Bigg\{S \in \mcS_d \, : \Sp(S) \subseteq [0, M] \textrm{ and } \frac{1}{n} \sum_{\mu=1}^n \phi\left(\sqrt{d} \left|\frac{x_\mu^\T S x_\mu}{d} - 1\right|\right) \leq \eps \Bigg\}.
    \end{equation*}
    There exists $\eps = \eps(\alpha, \phi) > 0$ such that for all $M > 0$,
    if $x_1, \cdots, x_n \iid \mcN(0, \Id_d)$,
    \begin{equation*}
    \lim_{d \to \infty} \bbP\left[\Gamma(\eps, M) \neq \emptyset\right] = 0.
    \end{equation*}
\end{theorem}
\noindent
As stated below, a direct corollary of Theorem~\ref{thm:main_negative_side} is that, when $\alpha > 1/4$, ellipsoid fitting admits (w.h.p.) no solutions
with spectrum bounded above as $d \to \infty$ (i.e.\ an ellipsoid whose smallest axis has length bounded away from zero).
\begin{corollary}[Non-existence of fitting ellipsoids with bounded spectrum]
    \label{cor:no_fit_bounded}
    Let $\alpha > 1/4$.    
    Let $n, d \to \infty$ with $n/d^2 \to \alpha > 0$, and $x_1, \cdots, x_n \iid \mcN(0, \Id_d)$. 
    We denote $\Gamma$ the set of ellipsoid fits for $(x_\mu)_{\mu=1}^n$.
    Then, for all $c > 0$,
    \begin{equation*}
        \lim_{d \to \infty} \bbP[\exists S \in \Gamma \, : \, \|S\|_\op \leq c] = 0.
    \end{equation*}
\end{corollary}

\myskip 
In particular, our results imply that the negative side of the ellipsoid fitting conjecture, i.e.\ eq.~\eqref{eq:conj_ef_negative},  
would follow from either:
\begin{enumerate}[label=\textbf{H.\arabic*},ref=H.\arabic*]
  \item\label{hyp:1} a proof that (for $\alpha > 1/4$ and with high probability)
  the set of ellipsoid fits is bounded in spectral norm as $d \to \infty$,
  \item[] 
    \begin{center}
        \emph{or}
    \end{center}
  \item\label{hyp:2} a proof that, if $S \in \Gamma$ is an ellipsoid fit, there exists a (possibly different) ellipsoid fit $\hS \in \Gamma$ that has bounded spectral norm 
  as $d \to \infty$.
\end{enumerate}
As a consequence of Corollary~\ref{cor:no_fit_bounded}, proving either \ref{hyp:1} or \ref{hyp:2} would yield the negative side of the ellipsoid fitting conjecture.
Notice further that Theorem~\ref{thm:main_negative_side} implies the non-existence of bounded ellipsoid fits even allowing a small ``fitting error'',
so that it would be sufficient to prove~\ref{hyp:2} considering an ellipsoid $\hS$ that fits $(x_\mu)_{\mu=1}^n$ only up to a small enough error $\eps > 0$ (in the sense of Theorem~\ref{thm:main_negative_side}).

\subsection{Discussion and consequences}\label{subsec:generalizations}

\noindent
The combination of our two main results (Theorem~\ref{thm:main_positive_side} and Theorem~\ref{thm:main_negative_side}) provides strong evidence for 
the original ellipsoid fitting conjecture (Conjecture~\ref{conj:ellipsoid_fitting}). 
Our conclusions are attained through the study of a ``Gaussian equivalent'' problem, which partly motivated Conjecture~\ref{conj:ellipsoid_fitting}.
Informally, we show that an approximate version of the ellipsoid fitting (i.e.\ by allowing infinitesimally small error) undergoes a sharp satisfiability transition at
$\alpha = n/d^2 = 1/4$.
Moreover, we also show in our proof that in the Gaussian equivalent problem, the satisfiability transition for this ``approximate'' version 
corresponds to the one of the exact fitting problem (i.e.\ not allowing for any non-zero error). 
This strongly suggests that our method is indeed capturing the phenomenon responsible for the ellipsoid fitting transition.

\myskip 
Our results are an example of a universality phenomenon in high-dimensional stochastic geometry: we show that 
the statistical dimension (or the square of the Gaussian width) of the set of positive semidefinite matrices determines the satisfiability of -- a modified version of -- random ellipsoid fitting, 
even though the affine subset $\{S \in \mcS_d \, : \, (x_\mu^\T S x_\mu = d)_{\mu=1}^n\}$ is \emph{not} randomly oriented uniformly in all directions.
In general, understanding the conditions under which universality holds in such problems of high-dimensional random geometry is an important open question.
We mention \cite{oymak2018universality}, which proves universality between a model in which the random subspace is given by the kernel of a random i.i.d.\ Gaussian matrix, 
and a second model where the subspace is the kernel of a matrix with independent elements (not necessarily Gaussian).

\subsubsection{Towards Conjecture~\ref{conj:ellipsoid_fitting}}

Unfortunately, while our main theorems characterize a satisfiability transition for ellipsoid fitting at the expected threshold, 
they do not formally imply Conjecture~\ref{conj:ellipsoid_fitting}.
We discuss briefly some improvements of our results that could potentially allow to bridge this gap.

\myskip 
On the ``positive'' side of the conjecture (i.e.\ the regime $\alpha = n/d^2 < 1/4$), 
Theorem~\ref{thm:main_positive_side} shows the existence of bounded ellipsoids that can achieve arbitrarily small error $\eps$ (where the error is taken to $0$ \emph{after} $d \to \infty$).
On the other hand, a proof of eq.~\eqref{eq:conj_ef_positive} would require to invert these limits, and take $\eps \to 0$ \emph{before} $d \to \infty$.
In this regard, an important strengthening of Theorem~\ref{thm:main_positive_side} would be to obtain non-asymptotic
bounds on $\bbP[\Gamma_r(\eps) = \emptyset]$, that depend on $\eps$.

\myskip 
Another potential for improvement stems from geometrical considerations: denoting $V \coloneqq \{S \in \mcS_d \, : \, x_\mu^\T S x_\mu = d \textrm{ for all } \mu \in [n]\}$, 
one may use Theorem~\ref{thm:main_positive_side} to bound the distance of the set $\Gamma_r(\eps)$ to the affine subspace $V$.
Since $\lambda_{\min}(S) \geq \lambda_-(\alpha)$ for any $S \in \Gamma_r(\eps)$, it would suffice to show that $d_\op(\Gamma_r(\eps), V) \leq \lambda_-(\alpha)$ 
to deduce eq.~\eqref{eq:conj_ef_positive}, the first part of Conjecture~\ref{conj:ellipsoid_fitting}.
We perform in Appendix~\ref{sec_app:geometric_approx_exact} a naive analysis of necessary conditions for this conclusion to follow from Theorem~\ref{thm:main_positive_side}.
Unfortunately, we find that (among other considerations) these conditions would require a significantly stronger form of Theorem~\ref{thm:main_positive_side}, 
by proving the conclusion for larger values of $r \in [1,2]$ and/or a better scaling with $d$ of the minimal error achievable (i.e.\ proving the conclusion of Theorem~\ref{thm:main_positive_side} for $\Gamma_r(\eps_d)$ with $\eps_d \to 0$ as $d \to \infty$).

\myskip
Moreover, let us emphasize that a critical difficulty in improving our proof techniques would be to quantitatively sharpen the universality arguments we carry out, and in particular to strengthen Proposition~\ref{prop:universality_gs}, which shows the universality of the 
minimal fitting error, or ``ground state'' energy, for ellipsoid fitting and a simpler ``Gaussian equivalent'' problem. 
While the present form of Proposition~\ref{prop:universality_gs} shows universality of this error up to a $\smallO_d(1)$ difference, this estimate would likely have to be improved in order to carry out the aforementioned approaches.
This part of our proof is greatly inspired by a recent literature on similar universality phenomena
\cite{goldt2022gaussian,hu2022universality,montanari2022universality,gerace2024gaussian,adamczak2011restricted,dandi2023universality,loureiro2021learning,dhifallah2020precise,schroder2023deterministic}, 
and we are not aware of the existence of such universality results at a finer scale (or even predictions/conjectures of conditions under which they should hold).

\myskip
Finally, on the ``negative'' side of the conjecture (i.e.\ for $\alpha > 1/4$), 
as emphasized in Corollary~\ref{cor:no_fit_bounded} and the discussion thereafter, 
Theorem~\ref{thm:main_negative_side} reduces the second part of Conjecture~\ref{conj:ellipsoid_fitting} (eq.~\eqref{eq:conj_ef_negative}) 
to proving either \ref{hyp:1} or \ref{hyp:2}.
If such a proof were to become available, our results would imply the regime $\alpha > 1/4$ of Conjecture~\ref{conj:ellipsoid_fitting}.

\subsubsection{Further directions}

Our proof method that leverages universality of the minimal fitting error is quite versatile, and
we end our discussion by mentioning a few further directions and generalized results that stem from our analysis. 

\myskip 
\textbf{The dual program --}
First, as a semidefinite program, ellipsoid fitting admits a dual formulation, as written e.g.\ in \cite{bandeira2023fitting}.
While the limitations of Theorems~\ref{thm:main_positive_side} and \ref{thm:main_negative_side}, discussed above, prevent us from directly drawing conclusions on the dual, 
it might be possible to directly apply to it a similar universality approach.
Such an application might allow to overcome some current limitations of our results, and we leave this investigation for future work.

\myskip 
\textbf{Beyond Gaussian vectors --}
Secondly, while we perform our analysis for $x_1, \cdots, x_n \sim \mcN(0, \Id_d)$, 
it is clear from our proof that our results (both Theorems~\ref{thm:main_positive_side} and \ref{thm:main_negative_side}) hold for any i.i.d.\ $(x_\mu)_{\mu=1}^n$ such that the matrices $W_\mu \coloneqq (x_\mu x_\mu^\T - \Id_d)/\sqrt{d}$ satisfy 
a uniform pointwise normality (or uniform one-dimensional CLT) assumption, as defined in Definition~\ref{def:one_dimensional_CLT}, 
and proven for the case of Gaussian vectors in Lemma~\ref{lemma:1d_clt_ellipse}.
An interesting example of a non-Gaussian distribution is given by the case of rotationally-invariant vectors with fluctuating norm, of the form 
\begin{equation*}
    x_\mu \deq \sqrt{r_\mu} \omega_\mu,
\end{equation*}
with $r_\mu$ and $\omega_\mu$ independent, and $\omega_\mu \sim \Unif(\bbS^{d-1}(\sqrt{d}))$. Letting $\tau \coloneqq \lim_{d \to \infty} [\sqrt{d} \Var(r_1)]$, 
\cite{maillard2024fitting} conjectures that the ellipsoid fitting transition point for this model is located at $n/d^2 = \alpha_c(\tau) \in (0,1/2)$,
and gives an exact expression of $\alpha_c(\tau)$ (see Fig.~5 of \cite{maillard2024fitting}), showing that ellipsoid fitting becomes harder as the fluctuations of the norm increase.
While pointwise normality may not hold in this setting, it is conceivable that our proof techniques can be adapted to handle these distributions, by following the calculations of \cite{maillard2024fitting}, to obtain results akin to Theorems~\ref{thm:main_positive_side} and \ref{thm:main_negative_side}.
More generally, while it is clear that some distributions can not satisfy uniform pointwise normality (see the discussion below Lemma~\ref{lemma:1d_clt_ellipse} for examples), a more thorough investigation of the class of distributions of $x_\mu$'s for which pointwise normality holds 
(and thus our proof applies) is, in our opinion, an interesting direction to explore.

\myskip 
\textbf{A minimal nuclear norm estimator --}
Let us conclude by mentioning a different approach to a possible solution of the first part of Conjecture~\ref{conj:ellipsoid_fitting}. 
It is conjectured in \cite{maillard2024fitting} (through non-rigorous methods) that the minimal nuclear norm solution, i.e.\ 
\begin{equation*}
    \hS_\NN \coloneqq \argmin_{\substack{S \in \mcS_d \\ \{x_\mu^\T S x_\mu = d\}_{\mu=1}^n}} \|S\|_{S_1} = \argmin_{\substack{S \in \mcS_d \\ \{x_\mu^\T S x_\mu = d\}_{\mu=1}^n}} \sum_{i=1}^d |\lambda_i(S)| ,
\end{equation*}
satisfies $\hS_\NN \succeq 0$ with high probability for any $\alpha < 1/4$. 
Analyzing $\hS_\NN$, whether through the techniques of the present paper or with different methods, 
is another promising approach to prove eq.~\eqref{eq:conj_ef_positive}, the ``positive'' part of Conjecture~\ref{conj:ellipsoid_fitting}.

\subsection{Structure of the paper}

\noindent
In Section~\ref{sec:outline_proof} we present the proof of our main results.
The proof of some intermediate results is postponed to later sections: in Section~\ref{sec:gaussian_equivalent} we study in detail 
the ``Gaussian equivalent'' problem to random ellipsoid fitting, and in Section~\ref{sec:proof_universality_gs} we prove a crucial universality property for the minimal fitting error between ellipsoid fitting and its ``Gaussian equivalent''.

\section{Proof of the main results}\label{sec:outline_proof}
\noindent
We prove here Theorems~\ref{thm:main_positive_side} and \ref{thm:main_negative_side}. 
The core idea of our proof can be sketched as follows: 
\begin{itemize}
    \item[$(i)$] Using rigorous methods inspired by statistical physics, we prove that a quantity known as the \emph{asymptotic free entropy} 
    is universal for the ellipsoid fitting problem of eq.~\eqref{eq:def_P} and a variant of its Gaussian counterpart of eq.~\eqref{eq:def_P_Gaussian}, for any value of $\alpha = n/d^2$.
    The main technique we use is an interpolation method.
    In a suitable limit (known as the low-temperature limit in statistical physics), this implies the universality of the minimal ``fitting error''.
    \item[$(ii)$] We study the Gaussian equivalent problem using methods of random convex geometry \cite{chandrasekaran2012convex,amelunxen2014living}, leveraging in particular
    Gordon's min-max inequality \cite{gordon1988milman}.
    When $\alpha < 1/4$ we show that not only a zero error is achievable, but that one can achieve it by a matrix whose spectrum is 
    contained in an interval of the type $[\lambda_-, \lambda_+]$, i.e.\ the axis of the corresponding ellipsoid have lengths bounded above and below.
    On the other hand, for $\alpha > 1/4$, we prove that not only is the Gaussian equivalent problem not satisfiable, 
    but one can lower bound the minimal fitting error as long as the set of candidate matrices is contained in an operator norm ball
    $B_\op(M)$ (for any constant $M > 0$).
    \item[$(iii)$] We prove that the conclusions of $(ii)$ transfer 
    to the original ellipsoid fitting problem, using the universality shown in $(i)$.
\end{itemize}
Our proof of $(i)$ leverages an important line of work on free entropy universality \cite{hu2022universality,montanari2022universality,gerace2024gaussian}, 
and a part of it closely follows the proof of \cite{montanari2022universality}, which we will point out in relevant places. Nevertheless, as our setting does not satisfy all the hypotheses of this work, 
and for completeness of our exposition, we chose to write the whole proof in a self-contained manner.

\subsection{Reduction of the problem and Gaussian equivalent}

\noindent
For any $0 \leq \lambda_- \leq \lambda_+$, any function $\phi : \bbR \to \bbR$
and any $\eps > 0$ we define the set
\begin{equation}\label{eq:def_Gamma}
    \Gamma(\phi, \lambda_-, \lambda_+, \eps) \coloneqq \Bigg\{S \in \mcS_d : \Sp(S) \subseteq [\lambda_-, \lambda_+] \textrm{ and } \frac{1}{n} \sum_{\mu=1}^n \phi\left(\sqrt{d} \left[\frac{x_\mu^\T S x_\mu}{d} - 1\right]\right) \leq \eps \Bigg\}.
\end{equation}
If one thinks of $\phi$ as an error (or loss) function, then $\Gamma(\phi, \lambda_-, \lambda_+, \eps)$
represents the set of matrices with spectrum in $[\lambda_-, \lambda_+]$ that solve $(\rm P)$ up to an approximation error $\eps$.
Notice that for any $S \in \mcS_d$:
\begin{align}\label{eq:correspondance_X_W}
   \sqrt{d} \left[\frac{x_\mu^\T S x_\mu}{d} - 1\right] = \Tr[W_\mu S] - \frac{d - \Tr[S]}{\sqrt{d}},
\end{align}
with $W_\mu \coloneqq (x_\mu x_\mu^\T - \Id_d)/\sqrt{d}$. Moreover, $W_\mu$ has the same first two moments as a Gaussian matrix. 
Formally, we define: 
\begin{definition}[Matrix ensembles]\label{def:matrix_ensembles}
    \noindent
    Let $d \geq 1$. We say that a random symmetric $W \in \mcS_d$ is generated according to\footnotemark:
    \begin{itemize}
        \item $W \sim \GOE(d)$ if $W_{ij} \iid \mcN(0, [1+\delta_{ij}]/d)$ for $i \leq j$.
        \item $W \sim \Ell(d)$ if $W \deq (x x^\T - \Id_d)/\sqrt{d}$ for $x \sim \mcN(0, \Id_d)$.
    \end{itemize}
\end{definition}
\footnotetext{$\GOE(d)$ stands for \emph{Gaussian Orthogonal Ensemble}.}
\noindent
One checks easily that $\EE_{\Ell(d)}[W_{ij} W_{kl}] = \EE_{\GOE(d)}[W_{ij} W_{kl}]$ for any $i \leq j$ and $k \leq l$.
This remark and eq.~\eqref{eq:correspondance_X_W} lead to consider the following modified problem,
with $W_\mu \coloneqq (x_\mu x_\mu^\T - \Id_d)/\sqrt{d}$ and $b \in \bbR$:
\begin{equation}\label{eq:def_Gamma_b}
    \Gamma_b(\phi, \lambda_-, \lambda_+, \eps) \coloneqq \left\{S \in \mcS_d : \Sp(S) \subseteq [\lambda_-, \lambda_+] \textrm{ and } \frac{1}{n} \sum_{\mu=1}^n \phi\left(\Tr[W_\mu S] - b\right) \leq \eps \right\}.
\end{equation}
In the rest of the proof we will focus on studying the set $\Gamma_b$ of eq.~\eqref{eq:def_Gamma_b} with\footnote{Furthermore, by rescaling $S$ (and up to a change in $\lambda_-, \lambda_+, \phi$) we will reduce to the case $b \in \{-1,0,1\}$.} $b \in \bbR$, 
for both $W_\mu \sim \Ell(d)$ and $W_\mu \sim \GOE(d)$ (which we call the ``Gaussian equivalent'' problem).
At the end of our proof, we will transfer our conclusions on $\Gamma_b$ back to the original solution set $\Gamma$ of eq.~\eqref{eq:def_Gamma}.

\subsection{Universality of the minimal error}\label{subsec:fe_universality_main_results}

\noindent
We can now state the main result concerning on the universality of the minimal error (or ``ground state energy'' in statistical physics jargon).
This result is inspired by a rich line of work on universality of empirical risk minimization \cite{hu2022universality,montanari2022universality,gerace2024gaussian,dandi2023universality}.
\begin{proposition}[Ground state universality]
    \label{prop:universality_gs}
    \noindent
    Let $\phi : \bbR \to \bbR_+$ and $\psi : \bbR \to \bbR$ two bounded differentiable functions with bounded derivatives, and assume furthermore $\|\psi'\|_L < \infty$.
    Let $n,d \geq 1$ and $n,d \to \infty$ with $\alpha_1 d^2 \leq n \leq \alpha_2 d^2$ for some $0 < \alpha_1 \leq \alpha_2$, and $B \subseteq \mcS_d$ a closed set such that $B \subseteq B_\op(C_0)$ for some $C_0 > 0$ (not depending on $d$).
    For $X_1, \cdots, X_n \in \mcS_d$ we define the \emph{ground state energy}:
    \begin{equation}\label{eq:def_gs}
        \GS_d(\{X_\mu\}) \coloneqq \inf_{S \in B} \frac{1}{d^2}\sum_{\mu=1}^n\phi(\Tr[X_\mu S]).
    \end{equation}
    Then we have:
    \begin{equation}\label{eq:universality_gs}
        \lim_{d \to \infty} \left| \EE_{\{W_\mu\} \iid \Ell(d)} \psi[\GS_d(\{W_\mu\})] - \EE_{\{G_\mu\} \iid \mathrm{GOE}(d)} \psi[\GS_d(\{G_\mu\})] \right| = 0.
    \end{equation}
    Therefore, for any $\rho \geq 0$ and $\delta > 0$: 
    \begin{equation}\label{eq:universality_gs_probabilities}
        \begin{dcases}
           \lim_{d \to \infty} \bbP[\GS_d(\{W_\mu\}) \geq \rho + \delta] \leq \lim_{d \to \infty} \bbP[\GS_d(\{G_\mu\}) \geq \rho], \\
           \lim_{d \to \infty} \bbP[\GS_d(\{W_\mu\}) \leq \rho - \delta] \leq \lim_{d \to \infty} \bbP[\GS_d(\{G_\mu\}) \leq \rho].
        \end{dcases}
    \end{equation}
\end{proposition}

\myskip
\textbf{A word on the proof --}
The main proof technique we use is Gaussian interpolation: namely we define an interpolating $U_\mu(t)$ such that 
$U_\mu(0) = G_\mu$ and $U_\mu(1) = W_\mu$, and show that $\GS_d(\{U_\mu(t)\})$ is constant (up to negligible terms) along the interpolation path. 
Note that while Proposition~\ref{prop:universality_gs} is very close to the results of \cite{montanari2022universality},
there is a technical difference with the setup of this work: for any fixed $S$, 
the random variable $\Tr[W S]$ for $W \sim \Ell(d)$ is not sub-Gaussian but only sub-exponential. 
As a consequence, we can not achieve a good control of the Lipschitz constant of the error (or ``energy'' function) of eq.~\eqref{eq:def_gs} with respect to the Frobenius norm of $S$,
as is required in \cite{montanari2022universality}.
We bypass this difficulty by controlling instead the Lipschitz constant with respect to the operator norm (see Lemma~\ref{lemma:energy_change_small_ball}), using important empirical process bounds over the operator norm ball (see Lemma~\ref{lemma:emp_proc_ellipse}): this leads to the limitation $B \subseteq B_\op(C_0)$. 
Interestingly, improving these bounds would also allow to relax the limitation $r \in [1,4/3)$ in Theorem~\ref{thm:main_positive_side}, as we discuss after.
Having dealt with this difficulty, the rest of the interpolation argument is very similar to \cite{montanari2022universality}. 
We show Proposition~\ref{prop:universality_gs} in Section~\ref{sec:proof_universality_gs}, deferring some arguments to Appendix~\ref{sec_app:technical_universality}.

\subsection{The Gaussian equivalent problem}\label{subsec:proof_gaussian_equivalent}

\noindent
We now study the Gaussian equivalent problem.
We will later transfer our analysis to the original ellipsoid fitting case 
using Proposition~\ref{prop:universality_gs}.
Our results are stated separately for the satisfiable and unsatisfiable regimes.
\begin{proposition}[Regular solutions in the satisfiable regime]
    \label{prop:regular_sol_gaussian}
    \noindent
    Let $n, d \to \infty$ with $n / d^2 \to \alpha < 1/4$, 
    and let $\{G_\mu\}_{\mu=1}^n \iid \mathrm{GOE}(d)$.
    Let
    \begin{equation*}
        V \coloneqq \{S \in \mcS_d \, : \forall \mu \in [n], \, \Tr[G_\mu S] = 1 \}.
    \end{equation*}
    There exist $0 < \lambda_- \leq \lambda_+$ (depending only on $\alpha$) such that:
    \begin{equation*}
        \lim_{d \to \infty} \bbP \{\exists S \in V \, \textrm{ s.t. } \Sp(S) \subseteq [\lambda_-, \lambda_+] \} = 1.
    \end{equation*}
\end{proposition}
\noindent
Proposition~\ref{prop:regular_sol_gaussian} shows that for $\alpha < 1/4$, there exist with high probability ellipsoids satisfying the ``Gaussian equivalent'' to random ellipsoid fitting, 
and that such solutions might also be assumed to have their axes' lengths bounded above and below as $d \to \infty$.
In the unsatisfiable regime $\alpha > 1/4$, we show on the other hand that with high probability there are no solutions to the Gaussian equivalent problem, even allowing for some error when fitting the random points.
\begin{proposition}[No approximate solution in the unsatisfiable regime]
    \label{prop:no_approx_unsat_gaussian}
    \noindent
    Let $n, d \to \infty$ with $n / d^2 \to \alpha > 1/4$, 
    and let $\{G_\mu\}_{\mu=1}^n \iid \mathrm{GOE}(d)$.
    Let $b \in \bbR$ and denote $\mcC_\mu^{(b)}(S) \coloneqq |\Tr(G_\mu S) - b|$. 
    Define the affine subspace:
    \begin{equation*}
        V_b \coloneqq \{S \in \mcS_d \, : \, \forall \mu \in [n], \, \mcC^{(b)}_\mu(S) = 0 \}.
    \end{equation*}
    \begin{itemize}
        \item[$(i)$] Assume that $b \neq 0$. 
        Then
        there exist $c = c(\alpha,b) > 0$ and $\eta = \eta(\alpha, b) \in (0,1)$ such that 
        \begin{equation*}
            \lim_{d \to \infty} \bbP\{\forall S \succeq 0 \, : \, \# \{\mu \in [n] \, : \, \mcC^{(b)}_\mu(S) > c\} \geq \eta n\} = 1.
        \end{equation*}
        \item[$(ii)$] Assume that $b = 0$. 
        Then
        there exist $c = c(\alpha) > 0$ and $\eta = \eta(\alpha) \in (0,1)$ such that, with probability $1 - \smallO_d(1)$, 
        the following holds for all $\tau \geq 0$:
        \begin{equation*}
           \sup_{\substack{S \succeq 0 \\ \# \{\mu \in [n] \, : \, \mcC^{(0)}_\mu(S) > c \tau\} < \eta n}} \|S\|_F \leq \tau \sqrt{d}.
        \end{equation*}
    \end{itemize}
\end{proposition}
\noindent
Propositions~\ref{prop:regular_sol_gaussian} and \ref{prop:no_approx_unsat_gaussian} are proven in Section~\ref{sec:gaussian_equivalent}.
Our proof follows a standard approach in random geometry problems involving Gaussian distributions, by leveraging Gordon's min-max inequality~\cite{gordon1988milman}
and its sharpness in convex settings \cite{thrampoulidis2015regularized,thrampoulidis2018precise}.
It strengthens for our setting the results obtained for general random convex programs in \cite{chandrasekaran2012convex,amelunxen2014living} (using either Gordon's inequality or tools of integral geometry).

\subsection{The satisfiable regime: proof of Theorem~\ref{thm:main_positive_side}}
\label{subsec:proof_thm_main_positive_side}

\noindent
Propositions~\ref{prop:universality_gs} and \ref{prop:regular_sol_gaussian}
have the following consequence, taking $B \coloneqq \{S \, : \, \lambda_-\Id_d \preceq S \preceq \lambda_+ \Id_d\}$, with 
$(\lambda_-, \lambda_+)$ given by Proposition~\ref{prop:regular_sol_gaussian}.
\begin{corollary}\label{cor:positive_side_weak_phi}
    \noindent
    Let $n, d \to \infty$ with $n / d^2 \to \alpha < 1/4$, and $W_1, \cdots, W_n \iid \Ell(d)$.
    There exist $\lambda_-, \lambda_+ > 0$ depending only on $\alpha$ such that the following holds. 
    If we have $\phi : \bbR \to \bbR_+$ with $\|\phi\|_\infty, \|\phi'\|_\infty < \infty$ 
    and such that $\phi(0) = 0$, then
    \begin{equation*}
        \plim_{d\to \infty} \min_{\Sp(S) \subseteq [\lambda_-, \lambda_+]} \frac{1}{n} \sum_{\mu=1}^n \phi(\Tr[W_\mu S] - 1) = 0.
    \end{equation*}
\end{corollary}
\noindent
The proof of Corollary~\ref{cor:positive_side_weak_phi} is immediate by combining Propositions~\ref{prop:universality_gs} and \ref{prop:regular_sol_gaussian}.
We can furthermore relax some of the assumptions on $\phi$ in Corollary~\ref{cor:positive_side_weak_phi}, as we now show. 
\begin{lemma}\label{lemma:positive_side_strong_phi}
    \noindent
    Corollary~\ref{cor:positive_side_weak_phi} holds for $\phi(x) = |x|^r$, for any $1 \leq r < 4/3$.
\end{lemma}
\noindent
Note that the limitation $r < 4/3$ is a consequence of a limitation on the control 
of an empirical process that is done in Lemma~\ref{lemma:emp_proc_ellipse} (see also the discussion in Section~\ref{subsec:lipschitz_energy}).

\begin{proof}[Proof of Lemma~\ref{lemma:positive_side_strong_phi}]
    Let us first assume that $r >1$, so that $\phi(x) = |x|^r$ is continuously differentiable in $x = 0$.
    Let $\eps > 0$ and $A > 0$, and let us denote 
    $u_A : \bbR_+ \to [0,1]$ a $\mcC^\infty$ function such that $u_A(x) = 1$ if $x \leq A$ and $u_A(x) = 0$ if $x \geq A+1$. 
    We denote $\phi_A(z) \coloneqq |z|^r u_A(|z|)$. Then $\phi_A$ is bounded, with bounded derivative.
    Moreover, we have for any $x \in \bbR$:
    \begin{equation*}
        |x|^r = \phi_A(x) + |x|^r(1-u_A(|x|)) \leq \phi_A(x) + |x|^r \indi\{|x| \geq A\}.
    \end{equation*}
    By Corollary~\ref{cor:positive_side_weak_phi}, under an event of probability $1 - \smallO_d(1)$ we can 
    fix $S$ with $\Sp(S) \in [\lambda_-, \lambda_+]$ and such that $\sum_{\mu=1}^n \phi_A(\Tr[W_\mu S] - 1) \leq n \eps/2$.
    We pick $\gamma > 1$ such that $\gamma r \leq 4/3$, and condition on the $1-\smallO_d(1)$ probability event, thanks to Lemma~\ref{lemma:emp_proc_ellipse}: 
    \begin{equation}\label{eq:bound_emp_process}
        \max_{\|R\|_\op = 1} \sum_{\mu=1}^n |\Tr (W_\mu R) |^{\gamma r} \leq C n
    \end{equation}
    We have, with probability $1 - \smallO_d(1)$:
    \begin{align*}
        \frac{1}{n} \sum_{\mu=1}^n |\Tr[W_\mu S] - 1|^r &\leq \frac{\eps}{2} + \frac{1}{n} \sum_{\mu=1}^n |\Tr[W_\mu S] - 1|^r \indi\{|\Tr[W_\mu S] - 1| \geq A\}, \\
        &\aleq \frac{\eps}{2}+ \frac{A^{r(1-\gamma)}}{n} \sum_{\mu=1}^n |\Tr[W_\mu S] - 1|^{\gamma r}, \\ 
        &\bleq \frac{\eps}{2}+ \frac{A^{r(1-\gamma) }2^{\gamma r - 1}}{n} [n + C \lambda_+^{\gamma r} n], \\
        &\leq \frac{\eps}{2}+ C(\gamma, r, \alpha) A^{r(1-\gamma)}.
    \end{align*}
    We used in $(\rm a)$ the following inequality, for a positive random variable $X$, $t \geq 0$, and any $\gamma > 1$:
    \begin{equation*}
        \EE[X \indi\{X \geq t\}] = t \EE\left[\frac{X}{t} \indi\left\{\frac{X}{t} \geq 1\right\}\right] \leq t^{1-\gamma} \EE[X^\gamma].
    \end{equation*}
    In $(\rm b)$ we used eq.~\eqref{eq:bound_emp_process} and $|a+b|^{r} \leq 2^{r-1}(|a|^r + |b|^r)$.
    We pick 
    \begin{equation*}
    A = \left(\frac{\eps}{2 C(\gamma, r, \alpha)}\right)^{1/[r(1-\gamma)]}.
    \end{equation*}
    We have then, with probability $1 - \smallO_d(1)$: 
    \begin{equation*}
        \min_{\Sp(S) \subseteq [\lambda_-, \lambda_+]}\frac{1}{n} \sum_{\mu=1}^n |\Tr[W_\mu S] - 1|^r \leq \eps,
    \end{equation*}
    which ends the proof.

    \myskip
    We now tackle the case $r = 1$. 
    For $\eta > 0$, we let $v_\eta : \bbR_+ \to [0,1]$ a $\mcC^\infty$ function such that 
    $v_\eta(x) = 1$ for $x \geq \eta$ and $v_\eta(x) = 0$ for $x \leq \eta/2$.
    A transposition of the argument above shows that
    Corollary~\ref{cor:positive_side_weak_phi} 
    applies to $\phi_\eta(x) \coloneqq |x| v_\eta(|x|)$, which is continuously differentiable everywhere.
    Notice that for any $x \in \bbR$:
    \begin{equation*}
        |x| = \phi_\eta(x) + |x|(1-v_\eta(|x|)) \leq \phi_\eta(x) + |x| \indi\{|x| \leq \eta\} 
        \leq \phi_\eta(x) + \eta. 
    \end{equation*}
    So, for any $S \in \mcS_d$:
    \begin{equation}\label{eq:bound_abs_phi_eta}
       \frac{1}{n} \sum_{\mu=1}^n |\Tr[W_\mu S] - 1| \leq 
       \frac{1}{n} \sum_{\mu=1}^n \phi_\eta(\Tr[W_\mu S] - 1) + \eta.
    \end{equation}
    Fixing now any $\eps > 0$, and letting $\eta \coloneqq \eps / 2$, we get from Corollary~\ref{cor:positive_side_weak_phi} applied to $\phi_\eta$ that with probability $1 - \smallO_d(1)$:
    \begin{equation*}
        \min_{\Sp(S) \subseteq [\lambda_-, \lambda_+]} \frac{1}{n} \sum_{\mu=1}^n \phi_\eta(\Tr[W_\mu S] - 1) \leq \frac{\eps}{2}.
    \end{equation*}
    Combining this result with eq.~\eqref{eq:bound_abs_phi_eta}, we get that with probability $1 - \smallO_d(1)$:
    \begin{equation*}
        \min_{\Sp(S) \subseteq [\lambda_-, \lambda_+]}\frac{1}{n} \sum_{\mu=1}^n |\Tr[W_\mu S] - 1| \leq \eps,
    \end{equation*}
    which ends the proof.
\end{proof}

\myskip 
\textbf{Proof of Theorem~\ref{thm:main_positive_side} --}
Notice that Lemma~\ref{lemma:positive_side_strong_phi} precisely shows that, for $\phi(x) = |x|^r$ and $\eps > 0$, the set $\Gamma_1$ of eq.~\eqref{eq:def_Gamma_b} is non-empty with high probability.
We now use the following remark (recall the definition of $\Gamma$ in eq.~\eqref{eq:def_Gamma}):
\begin{lemma}\label{lemma:Gamma_1_incl_Gamma}
    \noindent
    For any $(x_\mu)_{\mu=1}^n$ and $\lambda_-, \lambda_+, \eps > 0, r \geq 1$, if $S \in \Gamma_1(|\cdot|^r, \lambda_-, \lambda_+, \eps)$, then
    $\hS \in \Gamma(|\cdot|^r, \lambda_-', \lambda_+', \eps')$, with 
    \begin{equation*}
            \hS = \frac{dS}{\sqrt{d} + \Tr[S]}, \quad
            \lambda_-' = \frac{\lambda_-}{\lambda_+ + d^{-1/2}}, \quad
            \lambda_+' = \frac{\lambda_+}{\lambda_- +  d^{-1/2}}, \quad
            \eps' = \frac{\eps}{\left(\lambda_- +  d^{-1/2}\right)^r}.
    \end{equation*}
\end{lemma}
\noindent
Since $\lambda_+' \leq \lambda_+ / \lambda_-$, $\eps' \leq \eps/(\lambda_-)^r$, and $\lambda_-' \geq \lambda_- / (2 \lambda_+)$ for $d$ large enough, combining 
Lemmas~\ref{lemma:positive_side_strong_phi} and \ref{lemma:Gamma_1_incl_Gamma} imply that 
for any $r \in [1, 4/3)$ and $\eps > 0$:
\begin{equation*}
    \bbP[\Gamma(|\cdot|^r, a, b, \eps) \neq \emptyset] \to_{d \to \infty} 1,
\end{equation*}
for some $0 < a \leq b$ depending only on $\alpha$,
which ends the proof of Theorem~\ref{thm:main_positive_side}. $\qed$

\myskip
\begin{proof}[Proof of Lemma~\ref{lemma:Gamma_1_incl_Gamma}]
    Let $S \in \Gamma_1(|\cdot|^r, \lambda_-, \lambda_+, \eps)$. Defining $\hS =  dS / (\sqrt{d} + \Tr[S])$, 
    we have by eq.~\eqref{eq:correspondance_X_W}: 
    \begin{align*}
        \left|\sqrt{d} \left[\frac{x_\mu^\T \hS x_\mu}{d} - 1\right]\right|^r &= \left(\frac{d}{\sqrt{d} + \Tr[S]}\right)^r \left|\Tr[S W_\mu] - 1\right|^r, \\
        &\leq \frac{1}{\left(\lambda_- + d^{-1/2}\right)^r} \left|\Tr[S W_\mu] - 1\right|^r.
    \end{align*}
\end{proof}

\subsection{The unsatisfiable regime: proof of Theorem~\ref{thm:main_negative_side}}

\noindent
Propositions~\ref{prop:universality_gs} and \ref{prop:no_approx_unsat_gaussian} have the following corollary.
\begin{corollary}\label{cor:negative_side_weak_phi}
    \noindent
    Let $n, d \to \infty$ with $n / d^2 \to \alpha > 1/4$, and $W_1, \cdots, W_n \iid \Ell(d)$.
    Let $\phi: \bbR_+ \to \bbR_+$ be a non-decreasing differentiable function, with $\phi(0) = 0$, and such that $\phi$ has a unique global minimum in $0$.
    Then: 
    \begin{itemize}
        \item[$(i)$] Let $b \in \{-1, 1\}$. 
        There exists $\eps = \eps(\alpha, \phi) > 0$ such that for all $M > 0$:
        \begin{equation*}
            \lim_{d\to \infty} \bbP\left[\min_{\Sp(S) \subseteq [0, M]} \frac{1}{n} \sum_{\mu=1}^n \phi(|\Tr[W_\mu S] - b|) \geq \eps\right] = 1.
        \end{equation*}
        \item[$(ii)$] Let $b = 0$. 
        For all $\tau > 0$, there exists $\eps = \eps(\tau, \alpha, \phi) > 0$ such that 
        for all $M > 0$:
        \begin{equation*}
            \lim_{d\to \infty} \bbP\left[\min_{\substack{\Sp(S) \subseteq [0, M] \\ 
            \|S\|_F \geq \tau \sqrt{d}}} \frac{1}{n} \sum_{\mu=1}^n \phi(|\Tr[W_\mu S]|) \geq \eps\right] = 1.
        \end{equation*}
    \end{itemize}
    
\end{corollary}
\begin{proof}[Proof of Corollary~\ref{cor:negative_side_weak_phi}]
    Note that we can assume that $\phi$ is bounded with bounded derivative and $\phi'(0) = 0$: if not it is always possible to lower bound $\phi$ by such a function. 
    $x \mapsto \phi(|x|)$ is then a bounded function on $\bbR$ with bounded derivative.
    We start with $(i)$. By Proposition~\ref{prop:no_approx_unsat_gaussian}, there exist $c_\alpha, \eta_\alpha > 0$ such that 
        \begin{equation*}
            \lim_{d \to \infty} \bbP\{\forall S \succeq 0 \, : \, \# \{\mu \in [n] \, : \, |\Tr(G_\mu S) - b| \leq c_\alpha\} \leq (1-\eta_\alpha) n\} = 1.
        \end{equation*}
        Conditioning on this event, we have 
        \begin{equation*}
            \inf_{S \succeq 0} \sum_{\mu=1}^n \phi(|\Tr[G_\mu S] - b|) \geq n \eta_\alpha \phi(c_\alpha).
        \end{equation*}
        Using Proposition~\ref{prop:universality_gs} with $B = \{S \, : \, \Sp(S) \subseteq [0, M]\}$ we reach that, for all $M > 0$, with probability $1 - \smallO_d(1)$:
        \begin{equation*}
            \inf_{\Sp(S) \subseteq [0, M]} \frac{1}{n}\sum_{\mu=1}^n \phi(|\Tr[W_\mu S] - b|) \geq \frac{1}{2}\eta_\alpha \phi(c_\alpha).
        \end{equation*}
        We now turn to $(ii)$. Again by Proposition~\ref{prop:no_approx_unsat_gaussian}, we fix $c_\alpha, \eta_\alpha > 0$ 
        such that for all $\tau \geq 0$: 
        \begin{equation*}
           \lim_{d \to \infty} \bbP\left\{\sup_{\substack{S \succeq 0 \\ \# \{\mu \in [n] \, : \, |\Tr(G_\mu S)| > c_\alpha \tau\} < \eta_\alpha n}} \|S\|_F \leq \frac{\tau}{2} \sqrt{d}\right\} = 1.
        \end{equation*}
        Stated differently:
        \begin{equation*}
            \lim_{d \to \infty} \bbP\{\forall S  \succeq 0 : \, \|S\|_F \leq \frac{\tau}{2} \sqrt{d} \, \textrm{ or } \, \# \{\mu \in [n] \, : \, |\Tr(G_\mu S)| > c_\alpha \tau\} \geq \eta_\alpha n\} = 1.
        \end{equation*}
        Conditioning on this event and since $\phi$ is non-decreasing on $\bbR_+$:
        \begin{equation*}
            \inf_{\substack{S \succeq 0 \\ \|S\|_F \geq \tau \sqrt{d}}} \sum_{\mu=1}^n \phi(|\Tr[G_\mu S]|) \geq n \eta_\alpha \phi(c_\alpha \tau).
        \end{equation*}
        Using Proposition~\ref{prop:universality_gs} with $B = \{S \, : \, 0 \preceq S \preceq M \Id_d \, \textrm{ and } \|S\|_F \geq \tau\sqrt{d}\}$
        we reach that, for all $\tau, M$, with probability $1 - \smallO_d(1)$:
        \begin{equation*}
            \inf_{\substack{\Sp(S) \subseteq [0, M] \\ 
            \|S\|_F \geq \tau \sqrt{d}}} \frac{1}{n}\sum_{\mu=1}^n \phi(|\Tr[W_\mu S]|) \geq \frac{1}{2}\eta_\alpha \phi(c_\alpha \tau),
        \end{equation*}
        which ends the proof.
\end{proof}

\myskip 
We now turn to the proof of Theorem~\ref{thm:main_negative_side}. 

\begin{proof}[Proof of Theorem~\ref{thm:main_negative_side} --]
    Let $M > 0$. 
    As in the proof of Corollary~\ref{cor:negative_side_weak_phi},
    we can assume without loss of generality that $\phi$ has bounded derivative: if it does not, it is always possible to lower bound $\phi$ by such a function.

    \myskip 
    Let $\delta \in (0,1)$, and $S$ with $\Sp(S) \subseteq [0, M]$ and $|\Tr[S] - d| \geq \delta \sqrt{d}$.
    Notice that, defining $S' \coloneqq \sqrt{d}S / |d - \Tr[S]| \succeq 0$, we have 
    with $b \coloneqq \sign(d - \Tr[S]) \in \{\pm 1\}$:
    \begin{equation*}
        \Tr[S' W_\mu] - b = \frac{x_\mu^\T S x_\mu - d}{|\Tr S - d|},
    \end{equation*}
    and so since $\phi$ is non-decreasing:
    \begin{equation*}
        \phi\left(\sqrt{d} \left|\frac{x_\mu^\T S x_\mu}{d} - 1\right|\right) \geq \phi(\delta |\Tr (W_\mu S') - b|).
    \end{equation*}
    Moreover, $\Sp(S') \subseteq [0, M / \delta]$.
    Since this argument is valid for any $S$ with $\Sp(S) \subseteq [0, M]$ we get:
    \begin{equation*}
        \min_{\substack{\Sp(S) \subseteq [0, M]\\ |\Tr[S] - d| \geq \delta \sqrt{d}}} \frac{1}{n} \sum_{\mu=1}^n \phi\left(\sqrt{d} \left|\frac{x_\mu^\T S x_\mu}{d} - 1\right|\right) 
        \geq
        \min_{b \in \{\pm1\}}\min_{\Sp(S) \subseteq [0, M / \delta]} \frac{1}{n} \sum_{\mu=1}^n \phi(\delta |\Tr[W_\mu S] - b|).
    \end{equation*}
    Using Corollary~\ref{cor:negative_side_weak_phi} applied to $x \mapsto \phi(\delta x)$ there exists therefore $c = c(\alpha, \delta, \phi) > 0$ such that with probability $1 - \smallO_d(1)$:
    \begin{equation}\label{eq:lb_negative_side_trace_far}
        \min_{\substack{\Sp(S) \subseteq [0, M]\\ |\Tr[S] - d| \geq \delta \sqrt{d}}} \frac{1}{n} \sum_{\mu=1}^n \phi\left(\sqrt{d} \left|\frac{x_\mu^\T S x_\mu}{d} - 1\right|\right) 
        \geq c(\alpha, \delta, \phi).
    \end{equation}
    Let now $S \in \mcS_d$ with $\Sp(S) \subseteq [0, M]$ and $|\Tr[S] - d| \leq \delta \sqrt{d}$.
    Then:
    \begin{equation*}
        \Tr[W_\mu S] = \sqrt{d} \left(\frac{x_\mu^\T S x_\mu}{d} - 1\right) + \underbrace{\frac{d - \Tr[S]}{\sqrt{d}}}_{|\cdot| \leq \delta}, 
    \end{equation*}
    so that 
    \begin{equation}\label{eq:lb_negative_side_trace_close_1}
        \min_{\substack{\Sp(S) \subseteq [0, M]\\ |\Tr[S] - d| \leq \delta \sqrt{d}}} \frac{1}{n} \sum_{\mu=1}^n \phi\left(\sqrt{d} \left|\frac{x_\mu^\T S x_\mu}{d} - 1\right|\right) \geq 
        \min_{\substack{\Sp(S) \subseteq [0, M]\\ |\Tr[S] - d| \leq \delta \sqrt{d}}} \frac{1}{n} \sum_{\mu=1}^n \phi(|\Tr[W_\mu S]|) - \|\phi'\|_\infty \delta.
    \end{equation}
    Notice that since $\delta < 1$, for large enough $d$ we have $|\Tr[S] - d|\leq \delta \sqrt{d} \Rightarrow \Tr[S] \geq d/2$.
    If moreover $S \succeq 0$, by Cauchy-Schwarz we have $\|S\|_F \geq \Tr[S]/\sqrt{d} \geq \sqrt{d}/2$.
    This implies:
    \begin{equation}\label{eq:lb_negative_side_trace_close_2}
    \min_{\substack{\Sp(S) \subseteq [0, M]\\ |\Tr[S] - d| \leq \delta \sqrt{d}}} \frac{1}{n} \sum_{\mu=1}^n \phi(|\Tr[W_\mu S]|)
    \geq 
    \min_{\substack{\Sp(S) \subseteq [0, M]\\ \|S\|_F \geq \sqrt{d}/2}} \frac{1}{n} \sum_{\mu=1}^n \phi(|\Tr[W_\mu S]|).
    \end{equation}
    Using Corollary~\ref{cor:negative_side_weak_phi} we can obtain $\eps = \eps(\alpha, \phi) > 0$ such that,
    with probability $1 - \smallO_d(1)$, we have:
    \begin{equation}\label{eq:lb_negative_side_trace_close_3}
        \min_{\substack{\Sp(S) \subseteq [0, M]\\ \|S\|_F \geq \sqrt{d}/2}} \frac{1}{n} \sum_{\mu=1}^n \phi(|\Tr[W_\mu S]|)
        \geq \eps.
    \end{equation}
    Combining eq.~\eqref{eq:lb_negative_side_trace_far} with all three equations~\eqref{eq:lb_negative_side_trace_close_1},\eqref{eq:lb_negative_side_trace_close_2},\eqref{eq:lb_negative_side_trace_close_3}, 
     we get that for any $\delta > 0$, with probability $1 - \smallO_d(1)$:
    \begin{equation*}
        \min_{\Sp(S) \subseteq [0, M]} \frac{1}{n} \sum_{\mu=1}^n \phi\left(\sqrt{d} \left[\frac{x_\mu^\T S x_\mu}{d} - 1\right]\right) \geq
        \min[c(\alpha, \delta, \phi), \eps(\alpha, \phi) - \delta \|\phi'\|_\infty].
    \end{equation*}
    Taking $\delta \coloneqq \min\left(1,\eps(\alpha, \phi)/(2 \|\phi'\|_\infty)\right) > 0$ ends the proof.
\end{proof}

\section{The Gaussian equivalent problem}\label{sec:gaussian_equivalent}
\subsection{The satisfiable regime: proof of Proposition~\ref{prop:regular_sol_gaussian}}\label{subsec:gaussian_sat}

\subsubsection{Gordon's min-max theorem}

    We will use the Gaussian min-max theorem of Gordon \cite{gordon1988milman}, as stated in \cite{thrampoulidis2015regularized,thrampoulidis2018precise}:
    \begin{proposition}[Gaussian min-max theorem \cite{gordon1988milman,thrampoulidis2015regularized,thrampoulidis2018precise}]\label{prop:gaussian_minmax}
    \noindent
    Let $n, p \geq 1$,
    $W \in \bbR^{n \times p}$ an i.i.d.\ standard normal matrix, and $g \in \bbR^n,h \in \bbR^p$ two independent vectors 
    with i.i.d.\ $\mcN(0,1)$ coordinates.
    Let $\mcS_v, \mcS_u$ be two compact subsets respectively of $\bbR^p$ and $\bbR^n$, and let 
    $\psi : \mcS_v \times \mcS_u \to \bbR$ a continuous function. 
    We define the two optimization problems: 
    \begin{equation*}
        \begin{dcases}
            C(W)  &\coloneqq \min_{v \in \mcS_v} \max_{u \in \mcS_{u}} \left\{u^\intercal W v+ \psi(v, u)\right\}, \\ 
            \mcC(g, h) &\coloneqq \min_{v \in \mcS_v} \max_{u \in \mcS_{u}} \left\{\|u\|_2 h^\intercal v + \|v\|_2 g^\intercal u+ \psi(v, u)\right\}.
        \end{dcases}
    \end{equation*}
    Then:
    \begin{itemize}
        \item[$(i)$] For all $t \in \bbR$, one has 
    \begin{equation*}
        \bbP[C(W) < t] \leq 2\bbP[\mcC(g,h) \leq t].
    \end{equation*}
    \item[$(ii)$] Assume that $\mcS_v, \mcS_u$ are convex and that $\psi$ is convex-concave on 
    $\mcS_v \times \mcS_u$. Then for all $t \in \bbR$:
    \begin{equation*}
        \bbP[C(W) > t] \leq 2\bbP[\mcC(g,h) \geq t].
    \end{equation*}
    \end{itemize}
    \end{proposition}

\subsubsection{Gordon's min-max inequality and random geometry}

\noindent
We first introduce the notion of Gaussian width of a convex cone.
\begin{definition}[Gaussian width]\label{def:gaussian_width}
    \noindent
    For $p \geq 1$, and a closed convex cone $K \subseteq \bbR^p$, we define its \emph{Gaussian width} as 
    \begin{equation*}
        \omega(K) \coloneqq \EE \max_{x \in K \cap \bbS^{p-1}} \langle g, x\rangle,
    \end{equation*}
    for $g \sim \mcN(0, \Id_p)$.
\end{definition}
\noindent
We show now a general result leveraging Gordon's min-max inequality to prove the existence of a solution 
to a general type of random geometry problem. 
Such applications are classical, and we show here that one can assume furthermore that 
the solution is bounded.
\begin{lemma}\label{lemma:bounded_sols_general}
    \noindent
    Let $n, p \geq 1$, and $(g_\mu)_{\mu=1}^n \iid \mcN(0, \Id_p)$.
    We define $V \coloneqq \{x \in \bbR^p \, : \forall \mu \in [n], \, \langle g_\mu, x \rangle = 1 \}$.
    Let $K \subseteq \bbR^p$ be a closed convex cone, with Gaussian width $\omega(K)$.
    Assume that there exists $\varepsilon \in (0,1)$ such that $n \leq (1-\varepsilon) \, \omega(K)^2$ as $n \to \infty$. 
    Then:
    \begin{equation}\label{eq:bounded_sols_general}
        \lim_{n \to \infty} \bbP \left\{\exists x \in K \cap V \, \textrm{ s.t. } \, \|x\|_2 \leq  \frac{2}{\sqrt{\eps}} \right\} = 1.
    \end{equation}
\end{lemma}
\begin{proof}[Proof of Lemma~\ref{lemma:bounded_sols_general}]
    Let us denote, for $A > 0$:
    \begin{equation*}
        P_n(A) \coloneqq \bbP \left\{\exists x \in K \cap V \, \textrm{ s.t. } \, \|x\|_2 \leq  A \right\},
    \end{equation*}
    and define $G \in \bbR^{n \times p}$ as the Gaussian matrix with $g_\mu$ as its $\mu$-th row.
    By elementary compactness and duality arguments, we have:
    \begin{equation*}
        1 - P_n(A) = \bbP \Big[\min_{\substack{x \in K \\ \|x \|_2 \leq A}} \|G x - \ones_n \|_2 > 0\Big] = \bbP \Big[\min_{\substack{x \in K \\ \|x \|_2 \leq A}} \max_{\| \lambda\|_2 \leq 1} \{-\lambda^\intercal \ones_n + \lambda^\intercal G x\} > 0\Big],
    \end{equation*}
    for $G \in \bbR^{n \times p}$ with i.i.d.\ $\mcN(0,1)$ elements.
    By dominated convergence we have then 
    \begin{equation}\label{eq:ub_PnA_1}
        1 - P_n(A) 
        =  \lim_{\eta \to 0} \bbP \Big[\min_{\substack{x \in K \\ \|x \|_2 \leq A}} \max_{\| \lambda\|_2 \leq 1} \{-\lambda^\intercal \ones_n + \lambda^\intercal G x\} > \eta\Big].
    \end{equation}
    Note that in eq.~\eqref{eq:ub_PnA_1}, both $x$ and $\lambda$ belong to a convex and compact set (since $K$ is closed and convex), 
    and $\psi(x, \lambda) = -\lambda^\T \ones_n$ is clearly convex-concave.
    We can thus apply item $(ii)$ of Proposition~\ref{prop:gaussian_minmax}:
    \begin{equation}\label{eq:ub_PnA_2}
        1 - P_n(A) \leq 2 \lim_{\eta \to 0} \bbP\Big[\min_{\substack{x \in K \\ \|x \|_2 \leq A}} \max_{\| \lambda\|_2 \leq 1} \{-\lambda^\intercal \ones_n + \| \lambda\|_2 g^\intercal x + \|x \|_2 \lambda^\intercal h\} \geq \eta\Big],
    \end{equation}
    with $g \sim \mcN(0, \Id_p)$ and $h \sim \mcN(0, \Id_n)$.
    We then control the right-hand-side of the last equation, using that $K$ is a cone:
    \begin{align}\label{eq:ub_PnA_3}
        \nonumber
        &\min_{\substack{x \in K \\ \|x \|_2 \leq A}} \max_{\| \lambda\|_2 \leq 1} \{-\lambda^\intercal \ones_n + \| \lambda\|_2 g^\intercal x + \|x \|_2 \lambda^\intercal h\} 
        = \min_{\substack{x \in K \\ \|x \|_2 \leq A}} \max \{0, \|\|x\|_2 h - \ones_n \|_2 + g^\intercal x\}, \\ 
        \nonumber
        &= \max \Big\{0,\min_{\substack{x \in K \\ \|x \|_2 \leq A}}  \big[\|\|x\|_2 h - \ones_n \|_2 + g^\intercal x \big]\Big\}, \\ 
        &= \max \Big\{0,\min_{v \in [0, A]}  \big[\|v h - \ones_n \|_2 + v \min_{x \in K \cap \bbS^{p-1}}g^\intercal x \big]\Big\}.
    \end{align}
    Note that $g \to \max_{x \in K \cap \bbS^{p-1}}[g^\intercal x]$ is $1$-Lipschitz, 
    and in particular concentrates on its average, which by definition is the Gaussian width $\omega(K)$.
    We use the classical result (see e.g.\ Theorem~3.25 of \cite{van2014probability}): 
    \begin{theorem}
        \label{thm:gaussian_conc_lipschitz}
        \noindent
        Let $X_1, \cdots, X_n \iid \mcN(0,1)$. 
        Let $f : \bbR^n \to \bbR$ a Lipschitz function. 
        Then for all $t \geq 0$:
        \begin{equation*}
            \bbP[f(X_1, \cdots, X_n) - \EE f(X_1, \cdots, X_n) \geq t] \leq \exp\Big\{-\frac{t^2}{2 \| f \|_\rL^2}\Big\}.
        \end{equation*}
    \end{theorem}
    \noindent
    Therefore, for $\delta \in (0, \omega(K))$, we have\footnote{Since $g \deq -g$, $\min_{x \in K \cap \bbS^{p-1}}[g^\intercal x] \deq - \max_{x \in K \cap \bbS^{p-1}}[g^\intercal x]$.}:
    \begin{equation}\label{eq:concentration_gwidth}
       \bbP\Big\{\min_{x \in K \cap \bbS^{p-1}}[g^\intercal x] \geq - \omega(K) + \delta\Big\} \leq e^{-\delta^2/2}.
    \end{equation}
    From eqs.~\eqref{eq:ub_PnA_2},\eqref{eq:ub_PnA_3} and \eqref{eq:concentration_gwidth} we have: 
    \begin{equation}\label{eq:ub_PnA_4}
        1 - P_n(A) \leq 2 \lim_{\eta \to 0} \bbP\Big[
        \max \Big\{0,\min_{v \in [0, A]}  \big[\|v h - \ones_n \|_2 + v (- \omega(K) + \delta) \big]\Big\}\geq \eta\Big] + 2 e^{-\delta^2/2}.
    \end{equation}
    Recall that we assumed $n \leq (1-\varepsilon) \, \omega(K)^2$ and $n \to \infty$.
    Therefore, for $n \geq n_0(\varepsilon,\delta)$ large enough we can assume $n \leq (1-\varepsilon/2) \, [\omega(K) - \delta]^2$.
    Let $v^\star = v^\star(\varepsilon) = \sqrt{(4-\eps)/\eps}$ such that:
    \begin{equation*}
        \Big(1 - \frac{\varepsilon}{4}\Big) \EE_{X \sim \mcN(0,1)} [(v^\star X - 1)^2] = (v^\star)^2. 
    \end{equation*}
    Thus, for $A = v^\star(\eps)$, we have from eq.~\eqref{eq:ub_PnA_4}:
    \begin{equation}\label{eq:ub_PnA_5}
        1 - P_n(v^\star) \leq 2 \lim_{\eta \to 0} \bbP\Big[
        \max \Big\{0, \big[\|v^\star h - \ones_n \|_2 + v^\star (- \omega(K) + \delta) \big]\Big\}\geq \eta\Big] + 2 e^{-\delta^2/2}.
    \end{equation}
    By the law of large numbers we have ($\pto$ denotes convergence in probability):
    \begin{equation*}
        \frac{1}{\sqrt{n}} \|v^\star h - \ones_n \|_2 \pto \sqrt{\EE[(v^\star X - 1)^2]} = \frac{v^\star}{\sqrt{1 - \frac{\varepsilon}{4}}} < \frac{v^\star}{\sqrt{1 - \frac{\varepsilon}{3}}}.
    \end{equation*}
    In particular, with probability $1 - \smallO_n(1)$, we have
    \begin{equation*}
        \|v^\star h - \ones_n \|_2 + v^\star (- \omega(K) + \delta) \leq  v^\star \sqrt{n} \left[\left(1 - \frac{\varepsilon}{3}\right)^{-1/2} - \left(1 - \frac{\varepsilon}{2}\right)^{-1/2} \right] < 0.
    \end{equation*}
    Therefore, we have by eq.~\eqref{eq:ub_PnA_5}:
    \begin{equation*}
        1 - P_n(v^\star) \leq \smallO_n(1) + 2 e^{-\delta^2/2}.
    \end{equation*}
    Taking the limit $n \to \infty$ and then $\delta \to \infty$ finishes the proof\footnote{Recall that $\delta < \omega(K)$ but $\omega(K) \to \infty$ as $n \to \infty$ by hypothesis.} (notice that $v^\star \leq 2/\sqrt{\eps}$).
\end{proof}

\subsubsection{The cone of positive matrices with bounded condition number}

We study here the Gaussian width of the convex cone of positive semidefinite matrices with bounded condition number. 
We start with a classical result on the Gaussian width of $\mcS_d^+$ \cite{chandrasekaran2012convex,amelunxen2014living}, we refer the reader to Proposition~10.2 of \cite{amelunxen2014living} for a proof.
\begin{proposition}[Gaussian width of $\mcS_d^+$]
    \label{prop:gwidth_pds}
    \noindent 
    The Gaussian width of $\mcS_d^+$ satisfies: 
    \begin{equation*}
        \sqrt{\frac{d(d+1)}{4} - 1} \leq \omega(\mcS_d^+) \leq \sqrt{\frac{d(d+1)}{4}}.
    \end{equation*}
\end{proposition}
\noindent
In particular, notice that $\omega(\mcS_d^+) \sim d/2$ as $d \to \infty$.
We generalize this result by asking that the matrices have bounded condition number.
\begin{lemma}[Gaussian width of PSD matrices with bounded condition number]
    \label{lemma:gwidth_positive_condition_nb}
    \noindent
    For any $\kappa \geq 1$, define
    $K_\kappa \coloneqq \{S \in \mcS_d^+ \, : \, \lambda_\mathrm{max}(S) \leq \kappa \lambda_\mathrm{min}(S) \}$.
    Then $K_\kappa$ is a closed convex cone, and its Gaussian width satisfies 
    \begin{equation*}
        \liminf_{d \to \infty} \frac{2 \omega(K_\kappa)}{d} = f(\kappa),
    \end{equation*}
    where $\kappa \to f(\kappa) \in [0,1]$ is non-decreasing, with
    $\lim_{\kappa \to \infty} f(\kappa) = 1$.
\end{lemma}
\begin{proof}[Proof of Lemma~\ref{lemma:gwidth_positive_condition_nb}]
The fact that $K_\kappa$ is a closed convex cone is easy to verify.
Moreover, for any $\kappa \leq \kappa'$ we have $K_\kappa \subseteq K_{\kappa'} \subseteq \mcS_d^+$,
and $\omega(\mcS_d^+) \sim d/2$ by Proposition~\ref{prop:gwidth_pds}, so $\kappa \mapsto f(\kappa) \in [0,1]$, and $f$ is non-decreasing. 
To finish the proof, we show that $f(\kappa) \to 1$ as $\kappa \to \infty$.
We let $\kappa > 1$.
To identify $\mcS_d$ with $\bbR^{d(d+1)/2}$, we use the matrix flattening function, for $M \in \mcS_d$:
\begin{align}\label{eq:def_flattening}
    \flatt(M) &\coloneqq ((\sqrt{2}M_{ab})_{1 \leq a < b \leq d}, (M_{aa})_{a=1}^d) \in \bbR^{d(d+1)/2}, \\ 
    \nonumber
    &= ((2 - \delta_{ab})^{1/2} M_{ab})_{a\leq b}.
\end{align}
It is an isometry ($\langle \flatt(M), \flatt(N) \rangle = \Tr[MN]$), and if $Z \sim \GOE(d)$ then the entries of $\flatt(Z)$ satisfy $\flatt(Z)_k \iid \mcN(0, 2 / d)$.
We thus reach:
\begin{equation*}
    \omega(K_\kappa) \coloneqq \EE \max_{\substack{S \in K_\kappa \\ \|S \|_\rF^2 = 2 d}} \left\{\frac{1}{2} \Tr[Z S]\right\},
\end{equation*}
in which $Z \sim \GOE(d)$, cf.\ Definition~\ref{def:matrix_ensembles}.
Letting $z_1 \geq \cdots \geq z_d$ be the eigenvalues of $Z$, by Wigner's theorem \cite{anderson2010introduction}
$(1/d) \sum_i \delta_{z_i}$ weakly converges as $d \to \infty$ (a.s.) to $\sigma_{\sci}(\rd x) = (2\pi)^{-1} \sqrt{4-x^2} \indi\{ |x|\leq 2\} \rd x$.
Moreover, we have (taking $S$ having same eigenvectors as $Z$)\footnote{Notice that we replaced $\|S\|_F^2 = 2d$ by $\|S\|_F^2 \leq 2d$ since w.h.p.\ one can always 
find $S \in K_\kappa$ such that $\Tr[S Z] > 0$.}:
\begin{equation}\label{eq:lb_omega_Kkappa}
    \omega(K_\kappa) \geq \EE \max_{\substack{\lambda_1 \geq \cdots \lambda_d \geq 0 \\ \sum \lambda_i^2 \leq 2 d \\ \lambda_1 \leq \kappa \lambda_d}} \left\{\frac{1}{2} \sum_{i=1}^d \lambda_i z_i\right\}.
\end{equation}
Let us now sketch the end of the proof. 
We define
\begin{equation}\label{eq:def_lambdastar_Kkappa}
    \begin{dcases}
    \gamma(\kappa) &\coloneqq \sqrt{\frac{2}{\int_{2 \kappa^{-1}}^2 x^2 \, \sigma_\sci(\rd x) + \frac{4}{\kappa^2} \int_{-2}^{2 \kappa^{-1}} \sigma_\sci(\rd x)}}, \\    
    \lambda_i^\star &\coloneqq \gamma(\kappa) \Big[z_i \indi\{z_i \geq 2 \kappa^{-1}\} + \frac{2}{\kappa} \indi\{z_i < 2 \kappa^{-1}\}\Big].
    \end{dcases}
\end{equation}
Let $\eps > 0$.
Since one can show that $z_1 \to 2$ a.s.\ as $d \to \infty$ \cite{anderson2010introduction}, with high probability we have $\lambda_1^\star \leq \kappa(1+\eps) \lambda_d^\star$.
Moreover, one checks easily that $d^{-1} \sum_{i=1}^d [\lambda_i^\star]^2 \pto 2$ as $d \to \infty$.
Letting $\mu_i \coloneqq \lambda_i^\star / \sqrt{1+\eps}$, we can therefore use $\{\mu_i\}$ to lower bound $\omega(K_\kappa)$ as $d \to \infty$, 
with $\kappa_\eps \coloneqq \kappa(1+\eps)$. 
This yields that, for any $\eps > 0$: 
\begin{equation}\label{eq:lb_omega_Kkappa_2}
    f(\kappa_\eps) = \liminf_{d \to \infty} \frac{2 \omega(K_{\kappa_\eps})}{d} \geq \frac{1}{\sqrt{1+\eps}}
    \sqrt{\frac{2 \left[\int_{2 \kappa^{-1}}^2 x^2 \, \sigma_\sci(\rd x) + \frac{2}{\kappa} \int_{-2}^{2 \kappa^{-1}} x \, \sigma_\sci(\rd x)\right]^2}{\int_{2 \kappa^{-1}}^2 x^2 \, \sigma_\sci(\rd x) + \frac{4}{\kappa^2} \int_{-2}^{2 \kappa^{-1}} \sigma_\sci(\rd x)}}
    .
\end{equation}
Taking the limits $\kappa \to \infty$ and $\eps \to 0$
in eq.~\eqref{eq:lb_omega_Kkappa_2} yields $\lim_{\kappa \to \infty} f(\kappa) \geq 1$, which ends the proof.
\end{proof}

\subsubsection{Proof of Proposition~\ref{prop:regular_sol_gaussian}}

    We can now complete the proof of Proposition~\ref{prop:regular_sol_gaussian}.
    Recall that 
    \begin{equation*}
        V = \{S \in \mcS_d \, : \forall \mu \in [n], \, \Tr[G_\mu S] = 1 \}.
    \end{equation*}
    Since $\alpha = n/d^2 < 1/4$,
    by Lemma~\ref{lemma:gwidth_positive_condition_nb}, we can find $\kappa = \kappa(\alpha)$ and $\eps = \eps(\alpha)$ such that 
    $n \leq (1 - \varepsilon) \omega(K_\kappa)^2$ for $n$ large enough.
    Therefore, by Lemma~\ref{lemma:bounded_sols_general} there exists $A = A(\alpha) > 0$ such that:
    \begin{equation}
        \label{eq:regular_sol_gauss_1}
        \lim_{d \to \infty} \bbP \left\{\exists S \in V \, \textrm{ s.t. } S \succeq 0 \, \textrm{ and } \,\Tr[S^2] \leq A d \,\textrm{ and } \, \lambda_\mathrm{max}(S) \leq \kappa \lambda_\mathrm{min}(S) \right\} = 1.
    \end{equation}
    Notice that $H \coloneqq n^{-1/2} \sum_{\mu=1}^n G_\mu \sim \GOE(d)$,
    and thus $\bbP[\|H\|_\op \leq 3] = 1 - \smallO_d(1)$ \cite{vershynin2018high}. 
    Conditioning on this event, let $S \in V$. 
    Then $\Tr[HS] = \sqrt{n} = \sqrt{\alpha} d$, and
    thus by duality of $\|\cdot\|_\op$ and $\|\cdot\|_{S_1}$:
    \begin{equation*}
        \Tr|S| \geq \frac{1}{\|H\|_\op} |\Tr[H S]| \geq \frac{\sqrt{\alpha}}{3} d.
    \end{equation*}
    The proof of Proposition~\ref{prop:regular_sol_gaussian} is then ended by 
    noticing that:
    \begin{equation*}
        \Sp(S) \subseteq [\lambda_-, \lambda_+] \Leftarrow
        \begin{cases}
        &S \succeq 0, \\ 
        &\Tr[S^2] \leq A(\alpha) d,\\ 
        &\lambda_{\max}(S) \leq \kappa(\alpha) \lambda_{\min}(S) , \\ 
        &\Tr[S] \geq B(\alpha) d,
        \end{cases}
    \end{equation*}
    for some $0 < \lambda_- \leq \lambda_+$ depending only on $\alpha$.
    $\qed$

\subsection{The unsatisfiable regime: proof of Proposition~\ref{prop:no_approx_unsat_gaussian}}\label{subsec:gaussian_unsat}

\noindent
We will show the following general result on the unsatisfiability of approximate versions of a general class of random geometry problems.
\begin{proposition}
    \label{prop:thick_mesh}
    \noindent
    Let $b \in \bbR$, $p, n \to \infty$, $K \subseteq \bbR^p$ a closed convex cone, with Gaussian width $\omega(K)$,
    $(a_\mu)_{\mu=1}^n \iid \mcN(0, \Id_p)$, 
    and denote $\mcC_\mu^{(b)}(x) \coloneqq |a_\mu^\T x - b|$ the $\mu$-th ``constraint''.
    We denote
    \begin{equation*}
        V_b \coloneqq \{x \in \bbR^p \, : \forall \mu \in [n], \, \, \mcC^{(b)}_\mu(x) = 0 \}
    \end{equation*}
    a randomly-oriented affine subspace.
    Then we have the following: 
    \begin{itemize}
        \item[$(i)$] Assume that $b \neq 0$. Then
        \begin{itemize}
            \item[$(a)$] If there exists $\varepsilon > 0$ such that $n \leq (1-\varepsilon) \, \omega(K)^2$ as $n \to \infty$,
            then $\bbP[V_{b} \cap K \neq \emptyset] \to_{n\to\infty} 1$.
            \item[$(b)$] If there exists $\varepsilon > 0$ such that $n \geq (1+\varepsilon) \, \omega(K)^2$ as $n \to \infty$, 
            then there exist $c = c(\eps,b) > 0$ and $\eta = \eta(\eps,b) \in (0,1)$ such that
            \begin{equation*}
                \lim_{n \to \infty} \bbP\{\forall x \in K \, : \, \# \{\mu \in [n] \, : \, \mcC^{(b)}_\mu(x) > c\} \geq \eta n\} = 1.
            \end{equation*}
        \end{itemize}
        \item[$(ii)$] Assume that $b = 0$. Note that $V_0$ is a linear subspace, and $V_{0} \cap K$ is a closed convex cone.
        \begin{itemize}
            \item[$(a)$] Assume that $n \leq (1-\varepsilon) \, \omega(K)^2$ for some $\eps > 0$. 
            Then as $n \to \infty$, $\bbP[V_{0} \cap K \cap S^{p-1} \neq \emptyset] \to 1$.
            \item[$(b)$] Assume that $n \geq (1+\varepsilon) \, \omega(K)^2$ for some $\eps > 0$. 
            Then there exist $c = c(\eps,b) > 0$ and $\eta = \eta(\eps,b) \in (0,1)$ such that, 
            with probability $1 - \smallO_n(1)$, the following holds for all $\tau \geq 0$:
            \begin{equation*}
                \max_{\substack{x \in K \\ 
                \# \{\mu \in [n] \, : \, \mcC^{(0)}_\mu(x) > c \tau\} < \eta n}} \|x\|_2 \leq \tau.
            \end{equation*}
        \end{itemize}
    \end{itemize}
\end{proposition}
\noindent
It is clear that Proposition~\ref{prop:thick_mesh} ends the proof of Proposition~\ref{prop:no_approx_unsat_gaussian}.
Indeed, seen as an element of $\bbR^{d(d+1)/2}$, $\sqrt{d/2} G_\mu \iid \mcN(0, \Id_{d(d+1)/2})$, see the canonical embedding in eq.~\eqref{eq:def_flattening}.  
By Proposition~\ref{prop:gwidth_pds}, we have
$\omega(\mcS_d^+)^2 = d^2/4 + \smallO(d^2)$, and thus we can apply Proposition~\ref{prop:thick_mesh} with $p = d(d+1)/2$.
The difference $\sqrt{d}$ in normalization in point $(ii)$ of Proposition~\ref{prop:no_approx_unsat_gaussian} comes from the additional $\sqrt{d}$ factor that arises when relating $G_\mu$ to a standard Gaussian.

\myskip 
We now focus on proving Proposition~\ref{prop:thick_mesh}.

\subsubsection{Proof of Proposition~\ref{prop:thick_mesh}}
    We first show $(i)$, assuming $b \neq 0$.
    Results of \cite{chandrasekaran2012convex,amelunxen2014living}
    show that $V_b \cap K \neq \emptyset$ with high probability if $n \leq (1-\epsilon) \omega(K)^2$ (see e.g.\ Theorem~8.1 of \cite{amelunxen2014living}), 
    i.e.\ point $(a)$.
    While the converse is also shown in these works for $n \geq (1+\varepsilon) \omega(K)^2$, here we wish to prove the stronger statement $(b)$.
    Assume therefore that $n \geq (1+\varepsilon) \omega(K)^2$.
    By the union bound, for all $c \geq 0$ and $\eta \in (0,1)$: 
    \begin{align}
        \label{eq:ub_ceta}
        \nonumber
        &\bbP(\exists x \in K : \, \# \{\mu \in [n] \, : \, \mcC^{(b)}_\mu(x) > c\} < \eta n) \\ 
        \nonumber
        &= \bbP(\exists x \in K, \exists S \subseteq [n] \, : \, |S| > (1-\eta) n \, \textrm{ and } \,  \forall \mu \in S, \, \mcC^{(b)}_\mu(x) \leq c) \\ 
        \nonumber
        &\leq 
        \nonumber
        \sum_{\substack{S \subseteq [n] \\ |S| > (1-\eta) n}}\bbP(\exists x \in K : \, \forall \mu \in S, \, \mcC^{(b)}_\mu(x) \leq c), \\ 
        \nonumber
        &\aleq \left[\sum_{k < \eta \cdot n} \binom{n}{k}\right] \bbP(\exists x \in K \, : \, \|\{a_\mu^\T x - b\}_{\mu=1}^{(1-\eta)n} \|_\infty \leq c), \\ 
        &\bleq \exp\left\{\eta n \log \frac{e}{\eta}\right\} \bbP(\exists x \in K \, : \, \|\{a_\mu^\T x - b\}_{\mu=1}^{(1-\eta)n} \|_\infty \leq c).
    \end{align}
    We used that $a_\mu$ are i.i.d.\ in $(\rm a)$, 
    and the bound $\sum_{i=0}^k \binom{n}{i} \leq (en/k)^k$ in $(\rm b)$.
    We now make use of the following lemma (proven later on).
    \begin{lemma}
        \label{lemma:ub_ceta}
        \noindent
        Recall that $n \geq (1+\eps) \omega(K)^2$.
        There exist $c_1, c_2 > 0$ and $\eta_0 \in (0,1)$ depending only on $\eps$ such that
        for any $\eta \in (0,\eta_0)$:
        \begin{equation*}
            \bbP(\exists x \in K \, : \, \|\{a_\mu^\T x - b\}_{\mu=1}^{(1-\eta)n} \|_\infty \leq c_1 \cdot b) \leq 2 \exp\{-n c_2\}.
        \end{equation*}
    \end{lemma}
    \noindent
    Applying Lemma~\ref{lemma:ub_ceta} in eq.~\eqref{eq:ub_ceta}, we can 
    consider $\eta = \eta(\eps,b) \in (0,1/2)$ small enough, such that $\eta < \eta_0$ and $\eta \log(e/\eta) \leq c_2 /2$. 
    This yields then that 
    \begin{equation*}
        \bbP(\exists x \in K : \, \# \{\mu \in [n] \, : \, \mcC^{(b)}_\mu(x) > c_1 \cdot b\} < \eta n) \leq 2 \exp\{-nc_2/2\} \to 0,
    \end{equation*}
    and ends the proof.
    We now prove $(ii)$ of Proposition~\ref{prop:thick_mesh}, assuming $b = 0$.
    $(a)$ is here a simple consequence of the usual Gordon's ``escape through a mesh'' theorem \cite{gordon1988milman}, 
    so we focus on $(b)$, assuming $n \geq (1+\eps) \omega(K)^2$.
    Note that since $K$ is a cone,  we have 
    for all $c \geq 0$, $\eta \in (0,1)$: 
    \begin{align}
        \label{eq:b0}
        \nonumber
        &\sup_{\substack{x \in K \\ \# \{\mu \in [n] \, : \, \mcC^{(0)}_\mu(x) > c\} < \eta n}} \|x\|_2 \\
        &= \sup\{v \geq 0 \, : \, \exists x \in K \cap \bbS^{p-1} \, \textrm{ s.t. } \# \{\mu \in [n] \, : \, \mcC^{(0)}_\mu(x) > c / v\} < \eta n\}.
    \end{align}
    We now use the following counterpart to Lemma~\ref{lemma:ub_ceta} in the case $b = 0$, also proven later: 
    \begin{lemma}
        \label{lemma:ub_ceta_b0}
        \noindent
        Recall that $n \geq (1+\eps) \omega(K)^2$.
        There exist $c_1, c_2 > 0$ and $\eta_0 \in (0,1)$ depending only on $\eps$ such that
        for any $\eta \in (0,\eta_0)$:
        \begin{equation*}
            \bbP(\exists x \in K \cap \bbS^{p-1} \, : \, \|\{a_\mu^\T x\}_{\mu=1}^{(1-\eta)n} \|_\infty \leq c_1) \leq 2 \exp\{-n c_2\}.
        \end{equation*}
    \end{lemma}
    \noindent
    Repeating the same reasoning as in the $b \neq 0$ case, we can then find $c = c(\eps) > 0$ and $\eta = \eta(\eps) \in (0,1/2)$ such that 
    with probability $1 - \smallO_n(1)$:
    \begin{align*}
        \forall x \in K \cap \bbS^{p-1} \, : \, 
        \# \{\mu \in [n] \, : \, \mcC^{(0)}_\mu(x) > c\} \geq \eta n.
    \end{align*}
    Plugging this in eq.~\eqref{eq:b0} yields that with probability $1 - \smallO_n(1)$, for all $\tau \geq 0$:
    \begin{equation*}
    \max_{\substack{x \in K \\ \# \{\mu \in [n] \, : \, \mcC^{(0)}_\mu(x) > c \tau\} < \eta n}} \|x\|_2 \leq \tau,
    \end{equation*}
    which ends the proof. $\qed$

    \myskip
    We now prove the two Lemmas~\ref{lemma:ub_ceta} and \ref{lemma:ub_ceta_b0}.

\subsubsection{Proofs of Lemma~\ref{lemma:ub_ceta} and \ref{lemma:ub_ceta_b0}}\label{subsubsec:proof_lemma_ub_ceta}

\begin{proof}[Proof of Lemma~\ref{lemma:ub_ceta}]
    Note that for all $t \geq 0$ and $\eta \in (0,1)$:
    \begin{equation}\label{eq:rewriting_prob_intersection}
        \bbP\left(\exists x \in K \, : \, \|\{a_\mu^\T x - b\}_{\mu=1}^{(1-\eta)n} \|_\infty \leq t\right) 
        = \bbP\left[\exists x \in  K \, : \, \|G x - b \ones_m\|_\infty \leq t\right],
    \end{equation}
    with $m \coloneqq n(1-\eta)$,
    $G \in \bbR^{m \times p}$ an i.i.d.\ $\mcN(0,1)$ matrix, and $\ones_m$ the all-ones vector.
    We thus have:
    \begin{align}
        \label{eq:def_GammaAG}
        \nonumber
        &\bbP(\exists x \in K \, : \, \|\{a_\mu^\T x - b\}_{\mu=1}^{(1-\eta)n} \|_\infty \leq t)
        \\ 
        \nonumber
        &\aeq \lim_{A \to \infty} \bbP\left[\exists x \in  K \, : \, \| x \|_2 \leq A \, \textrm{ and }  \, \|G x - b \ones_m\|_\infty \leq t\right], \\
         &\beq \lim_{A \to \infty} \bbP\Big[\min_{\substack{x \in K \\ \| x\|_2 \leq A}} \|G x - b \ones_m\|_\infty \leq t\Big],
    \end{align}
    where $(\rm a)$ follows from dominated convergence, and $(\rm b)$ uses that the minimum is now over a compact set since $K$ is closed.
    Since $\| \cdot\|_\infty$ and $\| \cdot \|_1$ are dual norms, 
    we have for all $x \in K$:
    \begin{equation*}
        \|G x - b \ones_m\|_\infty = \max_{\substack{\lambda \in \bbR^m \\ \|\lambda\|_1 \leq 1}} \left[-b\lambda^\intercal \ones_m + \lambda^\intercal G x\right].
    \end{equation*}
    We can now use the Gaussian min-max inequality (Proposition~\ref{prop:gaussian_minmax}), which, 
    together with eq.~\eqref{eq:def_GammaAG}, implies:
    \begin{equation*}
        \bbP(\exists x \in K \, : \, \|\{a_\mu^\T x - b\}_{\mu=1}^{(1-\eta)n} \|_\infty \leq t) \leq \lim_{A \to \infty} 2\bbP[\gamma_A(g, h) \leq t]
        \aleq 2\bbP[\gamma(g, h)\leq t].
    \end{equation*}
    Here we defined: 
    \begin{equation}\label{eq:def_gamma_gh}
        \gamma(g, h) \coloneqq \inf_{x \in K} \max_{\substack{\lambda \in \bbR^m \\ \|\lambda\|_1 \leq 1}} \left[- b \lambda^\intercal \ones_m + \|\lambda\|_2 g^\intercal x + \|x \|_2 h^\intercal \lambda\right],
    \end{equation}
    and $\gamma_A$ is defined by restricting the infimum to $\|x\|_2 \leq A$. 
    Moreover, $g \sim \mcN(0, \Id_p), h \sim \mcN(0, \Id_m)$.
    The inequality $(\rm a)$ holds since $\gamma(g, h) \leq \gamma_A(g, h)$ for all $A > 0$.
    To conclude the proof, it therefore suffices to show: 
    \begin{equation}\label{eq:to_show_gammag}
        \bbP[\gamma(g, h) \leq c_1 b] \leq 2 \exp\{-nc_2\},
    \end{equation}
    for $c_1, c_2$ small enough (depending on $\eps, b$).
    We use again that $g \to \max_{x \in K \cap \bbS^{p-1}}[g^\intercal x]$ is $1$-Lipschitz, 
    and in particular concentrates on the Gaussian width by Theorem~\ref{thm:gaussian_conc_lipschitz}.
    Using that $g$ is distributed as $-g$, this implies that for any $u > 0$:
    \begin{equation*}
       \bbP\left\{\min_{x \in K \cap \bbS^{p-1}}[g^\intercal x] \leq - \omega(K) - u\right\} \leq e^{-u^2/2}.
    \end{equation*}
    Since $\omega(K) \leq \sqrt{n / (1+\varepsilon)}$ by hypothesis, 
    we can fix $\delta = \delta(\eps) > 0$ and $\eta_0 = \eta_0(\eps) > 0$ such that 
    for $n$ large enough we have for $\eta < \eta_0$: $\omega(K) + \delta \sqrt{m} \leq \sqrt{m / (1+\varepsilon/2)}$ (recall that $m = (1-\eta) n$).
    Thus, since $K$ is a cone, and using the max-min inequality, we have with probability at least $1 - e^{- n (1-\eta_0) \delta^2/2}$:
    \begin{equation}\label{eq:lb_gamma_gh}
        \gamma(g, h) \geq \inf_{v \geq 0} \max_{\substack{\lambda \in \bbR^m \\ \|\lambda\|_1 \leq 1}} \left[- b \lambda^\intercal \ones_m + v \left( h^\intercal \lambda- \|\lambda\|_2 (\omega(K) + \delta \sqrt{m}) \right)  \right].
    \end{equation}
    Let us now assume that $b > 0$, and let $X \sim \mcN(0,1)$, and $D \coloneqq \EE [|X|] = \sqrt{2/\pi}$. 
    We pick $\sigma = \sigma(\eps) \in (0,1)$ (its choice will be constrained later on), 
    and define $A_\eps$ by: 
    \begin{equation}
        \label{eq:def_Aeps}
            \EE[X^2 \indi\{X \leq A_\varepsilon\}] = \sigma(\eps).
    \end{equation}
    Finally, we define $\lambda^\star = \lambda^\star(h) \in \bbR^m$ by
    \begin{equation}\label{eq:def_lambdastar}
        \lambda^\star_\mu \coloneqq \frac{1}{D m} h_\mu \indi\{h_\mu \leq A_\varepsilon\}. 
    \end{equation}
    It is a simple exercise based on Hoeffding's and Bernstein's inequalities \cite{vershynin2018high}
    to check that for any $u > 0$ (recall that $m = (1-\eta) n \geq (1-\eta_0(\eps)) n$):
    \begin{equation}\label{eq:plims_lambdastar}
        \begin{dcases}
            \bbP[\|\lambda^\star\|_1 \leq 1] &\geq 1 - \exp\{-C_1 n\}  , \\
            \bbP\left[h^\T \lambda^\star - \frac{\sigma(\eps)^2}{D} \leq - u\right] &\leq \exp\{-C_2 n \min(u^2, u)\}, \\
            \bbP\left[\|\lambda^\star\|_2^2 - \frac{\sigma(\eps)^2}{m D^2} \geq \frac{u}{n}\right] &\leq \exp\{-C_3 n \min(u^2, u)\}, \\
            \bbP\left[\ones_m^\T \lambda^\star \geq - C_4\right] &\leq \exp\{-C_5 n\},
        \end{dcases}
    \end{equation}
    for some positive constants $(C_a)_{a=1}^5$, all depending on $\eps$.
    In particular, the first three lines of eq.~\eqref{eq:plims_lambdastar} imply 
    that for all $u \in (0,1)$, with probability at least $1 - 3 \exp\{-C(\eps) n u^2\}$:
    \begin{equation}\label{eq:def_sigma_gamma_eps}
        \|\lambda^\star\|_1 \leq 1 \, \textrm{ and } \frac{h^\T \lambda^\star}{\|\lambda^\star\|_2} \frac{1}{\omega(K) + \delta \sqrt{m}} 
        \geq \frac{\sigma(\eps)^2/D - u}{\sqrt{\sigma(\eps)^2 / D^2 + u}} \sqrt{1 + \frac{\eps}{2}} ,
    \end{equation}
    in which we used that $\omega(K) + \delta \sqrt{m} \leq \sqrt{m} (1+\eps/2)^{-1/2}$, 
    and $m / n \leq 1$.
    We can choose $\sigma(\eps) \in (0,1)$ sufficiently close to $1$, and $u(\eps) \in (0,1)$ sufficiently close to $0$
    such that the right-hand side of eq.~\eqref{eq:def_sigma_gamma_eps} is greater than $1$.
    Combining it with the last equation of eq.~\eqref{eq:plims_lambdastar} and 
    the lower bound of eq.~\eqref{eq:lb_gamma_gh}, 
    we get that (with new constants $c_1, c_2$ depending on $\eps$), with probability at least $1 - 2 \exp\{-c_2(\eps) n\}$:
    \begin{equation*}
        \gamma(g, h) \geq b c_1(\eps),
    \end{equation*}
    which implies eq.~\eqref{eq:to_show_gammag} and ends the proof.
    The case $b < 0$ is treated similarly, constraining $h_\mu \geq - A_\eps$ rather than $h_\mu \leq A_\eps$ in eq.~\eqref{eq:def_lambdastar}.
\end{proof}

\begin{proof}[Proof of Lemma~\ref{lemma:ub_ceta_b0}]
    Let $t \geq 0$. Repeating the same arguments as in the proof of Lemma~\ref{lemma:ub_ceta} 
    one obtains that 
    \begin{equation*}
        \bbP(\exists x \in K \cap \bbS^{p-1} \, : \, \|\{a_\mu^\T x\}_{\mu=1}^{(1-\eta)n} \|_\infty \leq t) 
        \leq 2 \bbP\Bigg[\underbrace{\min_{x \in K \cap \bbS^{p-1}} \max_{\substack{\lambda \in \bbR^m \\ \|\lambda\|_1 \leq 1}} \left\{\|\lambda\|_2 g^\T x + h^\T \lambda\right\} \leq t}_{\eqqcolon \gamma(g,h)}\Bigg].
    \end{equation*}
    Again, we
    can fix $\eta_0(\eps) > 0$ and $\delta(\eps) > 0$ such that for $\eta < \eta_0$
    and $n$ large enough we have $\omega(K) + \delta\sqrt{m} \leq \sqrt{m} [1+\eps/2]^{-1/2}$. 
    By the max-min inequality and the concentration of the Gaussian width, this implies that with probability at least $1-e^{-n (1-\eta_0) \delta^2}$:
    \begin{equation}\label{eq:lb_gamma_gh_b0}
        \gamma(g, h) \geq \max_{\substack{\lambda \in \bbR^m \\ \|\lambda\|_1 \leq 1}} \left[h^\intercal \lambda- \|\lambda\|_2 (\omega(K) + \delta \sqrt{m}) \right].
    \end{equation}
    Defining again $D \coloneqq \sqrt{2/\pi}$ so that $D = \EE[|X|]$ for $X \sim \mcN(0,1)$, we now define $\lambda^\star$ as:
    \begin{equation*}
        \lambda_\mu^\star \coloneqq \frac{1}{D m} h_\mu.
    \end{equation*}
    We have the counterpart to eq.~\eqref{eq:plims_lambdastar} for this case:
    for any $u > 0$ and $\tau > 0$,
    \begin{equation}\label{eq:plims_lambdastar_b0}
        \begin{dcases}
            \bbP[\|\lambda^\star\|_1 \leq 1 + \tau] &\geq 1 - \exp\{-C_1 n \tau^2\}  , \\
            \bbP\left[h^\T \lambda^\star - \frac{1}{D} \leq - u\right] &\leq \exp\{-C_2 n \min(u^2, u)\}, \\
            \bbP\left[\|\lambda^\star\|_2^2 - \frac{1}{m D^2} \geq \frac{u}{n}\right] &\leq \exp\{-C_3 n \min(u^2, u)\},
        \end{dcases}
    \end{equation}
    for some $(C_a)_{a=1}^3$ depending on $\eps$.
    We can fix $u = u(\eps) > 0$ such that 
    \begin{equation*}
        D^{-1}-u - \sqrt{\frac{u + D^{-2}}{1  + \eps/2}} \eqqcolon 2 c_1(\eps) > 0,
    \end{equation*}
    since the limit as $u \to 0$ of the left-hand side is strictly positive.
    Letting $\tau = 1$, we finally get that with probability at least $1 - 3 \exp(-c_2(\eps) n)$ we can lower bound (using $\lambda = \lambda^\star/2$ such that $\|\lambda\|_1 \leq 1$)
    \begin{align*}
        \gamma(g, h) &\geq \frac{D^{-1} - u}{2} - \frac{1}{2} \sqrt{\frac{u}{n} + \frac{1}{mD^2}} (\omega(K) + \delta \sqrt{m}), \\ 
        &\geq \frac{D^{-1} - u}{2} - \frac{1}{2} \sqrt{\frac{u + D^{-2}}{1 + \eps/2}}, \\ 
        &\geq c_1(\eps).
    \end{align*}
    This ends the proof.
\end{proof}

\section{Universality: proof of Proposition~\ref{prop:universality_gs}}\label{sec:proof_universality_gs}
\noindent
This section is devoted to the proof of Proposition~\ref{prop:universality_gs}.
We first show in Section~\ref{subsec:lipschitz_energy} a critical result on the Lipschitz constant of the ``error'' function appearing in 
eq.~\eqref{eq:def_gs}. This requires controlling a random process on the operator norm sphere, 
which is also useful in the proof of Theorem~\ref{thm:main_positive_side}, see Section~\ref{subsec:proof_thm_main_positive_side}.
We leverage this control to show in Section~\ref{subsec:universality_matrix} a general result on the universality of a quantity known as the asymptotic free entropy of the model, 
both for matrices arising from ellipsoid fitting and its Gaussian equivalent.
This result follows from an interpolation argument.
Finally, we apply these results in the so-called ``low-temperature'' limit in Section~\ref{subsec:proof_universality_gs} 
to deduce Proposition~\ref{prop:universality_gs}.
As we mentioned, while we can not directly apply the results of \cite{montanari2022universality},
parts of Sections~\ref{subsec:universality_matrix} and \ref{subsubsec:proof_universality_gs} closely follows their approach.
We defer to Appendix~\ref{sec_app:technical_universality} some technicalities, as well as some parts of the proof that 
more directly follow the arguments of \cite{montanari2022universality}.

\subsection{Lipschitz constant of the energy function, and bounding random processes}\label{subsec:lipschitz_energy}

\noindent
We show here the following result on the behavior of the error (or ``energy'') function, under 
both models $X_\mu \sim \Ell(d)$ and $X_\mu \sim \GOE(d)$.
\begin{lemma}[Lipschitz constant of the energy]\label{lemma:energy_change_small_ball}
    \noindent
    Let $n,d \geq 1$ and $n,d \to \infty$ with $\alpha_1 d^2 \leq n \leq \alpha_2 d^2$ for some $0 < \alpha_1 < \alpha_2$.
    Let $\phi : \bbR \to \bbR_+$ such that $\|\phi'\|_\infty < \infty$, and
    $X_1, \cdots, X_n \in \mcS_d$ be generated i.i.d.\ according to either $\GOE(d)$ or $\Ell(d)$.
    For $S \in \mcS_d$, we define the \emph{energy}:
    \begin{equation*}
        E_{\{X_\mu\}}(S) \coloneqq \ \frac{1}{d^2} \sum_{\mu=1}^n \phi[\Tr(X_\mu S)].
    \end{equation*}
    Then the following holds for some $C > 0$ (depending only on $\alpha$):
    \begin{equation*}
        \bbP\left[\sup_{S_1, S_2 \in \mcS_d} \frac{|E_{\{X_\mu\}}(S_1) - E_{\{X_\mu\}}(S_2)|}{\|S_1 - S_2\|_\op} \leq C \|\phi'\|_\infty \right] \geq 1 - 2 e^{-n}.
    \end{equation*}
\end{lemma}
\noindent
In other words, the energy function has a bounded Lipschitz constant (as $d \to \infty$) with respect to the operator norm.
Note that this is strictly stronger than what a naive use of the triangular inequality
and of the duality $\|\cdot\|_\op \leftrightarrow \|\cdot\|_{S_1}$ yields:
\begin{equation*}
    \frac{|E_{\{X_\mu\}}(S_1) - E_{\{X_\mu\}}(S_2)|}{\|S_1 - S_2\|_\op} \leq \frac{\|\phi'\|_\infty}{d^2} \sum_{\mu=1}^n \Tr|X_\mu|,
\end{equation*}
since $\Tr|X_\mu| \gtrsim d$ for $X_\mu \sim \GOE(d)$, and $\Tr|X_\mu| \gtrsim \sqrt{d}$ for $X_\mu \sim \Ell(d)$.
Instead, the proof of Lemma~\ref{lemma:energy_change_small_ball} is based on the following bounds for random processes, for which we separate 
the $\GOE(d)$ and $\Ell(d)$ setting. Lemma~\ref{lemma:emp_proc_gaussian} is a consequence of elementary concentration results, 
and is proven in Appendix~\ref{subsec_app:proof_lemma_emp_proc_gaussian}, while Lemma~\ref{lemma:emp_proc_ellipse} is proven in the following. 
\begin{lemma}[Bounding random processes, $\GOE(d)$ setting]\label{lemma:emp_proc_gaussian}
    \noindent 
    Let $(G_\mu)_{\mu=1}^n \iid \GOE(d)$.
    Let $r \in [1,2]$. 
    We assume that $n \geq \alpha_1 d^2$, for some $\alpha_1 > 0$.
    There exists $C = C(\alpha_1) > 0$ such that for all $t > 0$:
    \begin{equation*}
        \bbP\Bigg[\max_{\|S\|_F^2 = d} \left(\sum_{\mu=1}^n |\Tr[G_\mu S] | ^r\right)^{1/r} \geq (C + t) n^{1/r}\Bigg] \leq \exp\left(-nt^2/2\right).
    \end{equation*}
\end{lemma}
\begin{lemma}[Bounding random processes, $\Ell(d)$ setting]\label{lemma:emp_proc_ellipse}
    \noindent 
    Let $W_1, \cdots, W_n$ be drawn i.i.d.\ from $\Ell(d)$.
    Let $r \in [1,4/3]$. 
    We assume that $\alpha_1 d^2 \leq n \leq \alpha_2 d^2$, for some $0 < \alpha_1 < \alpha_2$.
    There are constants $C_1, C_2 > 0$ (that might depend on $\alpha_1, \alpha_2$) such that for all $t > 0$:
    \begin{equation*}
        \bbP\Big[\max_{\|S\|_\op = 1} \sum_{\mu=1}^n |\Tr (W_\mu S) |^r \geq n (C_1 + t)\Big] \leq 2 \exp\left\{-C_2 \min (n t^{\frac{2}{r}}, n^{\frac{1}{4} + \frac{1}{r}} t^{\frac{1}{r}} )\right\}.
    \end{equation*}
\end{lemma}
\begin{proof}[Proof of Lemma~\ref{lemma:energy_change_small_ball}]
    We finish here the proof, assuming Lemmas~\ref{lemma:emp_proc_gaussian} and \ref{lemma:emp_proc_ellipse}.
    By the mean value theorem and the duality $\|\cdot\|_\op \leftrightarrow \|\cdot\|_{S_1}$:
    \begin{align*}
        |E_{\{X_\mu\}}(S_1) - E_{\{X_\mu\}}(S_2)| &\leq \left|\sup_{S \in \mcS_d} \langle \nabla_S E_{\{X_\mu\}}(S), S_1 - S_2 \rangle\right|, \\
        &\leq \|S_1 - S_2\|_\op  \sup_{S \in \mcS_d} \| \nabla_S E_{\{X_\mu\}}(S) \|_{S_1}.
    \end{align*}
    Again using the duality $\|\cdot\|_\op \leftrightarrow \|\cdot\|_{S_1}$:
    \begin{align*}
        \| \nabla_S E_{\{X_\mu\}}(S) \|_{S_1} &= \frac{1}{d^2} \left\| \sum_{\mu=1}^n X_\mu \phi'[\Tr(X_\mu S)]\right\|_{S_1}, \\ 
        &= \frac{1}{d^2}\sup_{\|R\|_\op = 1} \sum_{\mu=1}^n \Tr[X_\mu R] \phi'[\Tr(X_\mu S)], \\
        &\leq \frac{\|\phi'\|_\infty}{d^2} \sup_{\|R\|_\op = 1} \sum_{\mu=1}^n |\Tr[X_\mu R]|.
    \end{align*}
    Using Lemmas~\ref{lemma:emp_proc_gaussian} and \ref{lemma:emp_proc_ellipse} in the case $r = 1$, we reach the sought statement.
\end{proof}

\myskip
\textbf{Remark I --} 
Note that a naive argument using that $\Tr[W_\mu S]$ is a sub-exponential random variable yields Lemma~\ref{lemma:emp_proc_ellipse} for $r = 1$, which is already enough to deduce Lemma~\ref{lemma:energy_change_small_ball}.
However, Lemma~\ref{lemma:emp_proc_ellipse} is also used later in the proof of Theorem~\ref{thm:main_positive_side}, see Section~\ref{subsec:proof_thm_main_positive_side}.
Since $\Tr[W_\mu S]$ is the sum of many independent random variables,
we can leverage its two-tailed behavior (by Bernstein's inequality) to prove Lemma~\ref{lemma:emp_proc_ellipse} for all $r \leq 4/3$, 
yielding the limitation $r < 4/3$ in Theorem~\ref{thm:main_positive_side}.
It is possible that a finer analysis could lead to a proof of Lemma~\ref{lemma:emp_proc_ellipse} for the case $4/3 \leq r \leq 2$, 
which would in turn imply Theorem~\ref{thm:main_positive_side} for $r < 2$.

\myskip
\textbf{Remark II --} 
We give an informal argument as to why we can not hope to extend Lemma~\ref{lemma:emp_proc_gaussian} nor \ref{lemma:emp_proc_ellipse} for $r > 2$.
Indeed, in the $\GOE(d)$ setting, the choice $S = \sqrt{d} G_\mu / \|G_\mu\|_F$ (for some $\mu \in [n]$) yields that the objective function 
is at least $\sqrt{d} \|G_\mu\|_F \gtrsim \sqrt{n}$.
In the $\Ell(d)$ setting, assume that $r > 2$.
If $1/q + 1/r = 1$, then by the dualities $\ell^p \leftrightarrow \ell^q$ and $\|\cdot\|_\op \leftrightarrow \|\cdot\|_{S_1}$:
\begin{equation}\label{eq:dual_empr_proc}
    \max_{\|S\|_\op = 1} \left(\sum_{\mu=1}^n |\Tr (W_\mu S) |^r\right)^{1/r} 
    = \max_{\|\lambda\|_q = 1} \left\|\sum_{\mu=1}^n \lambda_\mu W_\mu\right\|_{S_1}.
\end{equation}
Let us lower bound the right-hand side of eq.~\eqref{eq:dual_empr_proc}.
Let $\beta \in (0,1)$, $p = \beta n$, and $\lambda_1 = \cdots = \lambda_p = p^{-1/q} > \lambda_{p+1} = \cdots = \lambda_n = 0$.
Then $\|\lambda\|_q = 1$.
Moreover, 
\begin{equation*}
    \left\|\sum_{\mu=1}^n \lambda_\mu W_\mu\right\|_{S_1} = p^{1-1/q} d^{-1/2} \left\|\frac{1}{p}\sum_{\mu=1}^p x_\mu x_\mu^\T - \Id_d\right\|_{S_1}.
\end{equation*}
By classical results of concentration of Wishart matrices \cite{vershynin2018high}, we know that since $p \lesssim d^2$ then 
\begin{equation*}
    \left\|\frac{1}{p}\sum_{\mu=1}^p x_\mu x_\mu^\T - \Id_d\right\|_{S_1} \gtrsim \frac{d^{3/2}}{\sqrt{p}} \gtrsim \sqrt{\frac{d}{\beta}},
\end{equation*}
where $\gtrsim$ might hide constants that depend on $\alpha$.
We then reach:
\begin{equation*}
    \left\|\sum_{\mu=1}^n \lambda_\mu W_\mu\right\|_{S_1} \gtrsim \beta^{1/2 - 1/q} n^{1-1/q}.
\end{equation*}
Since $r > 2$, one has $1/2 - 1/q < 0$.
Letting $\beta$ going to $0$, this shows that any bound of the type
\begin{equation*}
    \max_{\|\lambda\|_q = 1} \left\|\sum_{\mu=1}^n \lambda_\mu W_\mu\right\|_{S_1} \leq C(\alpha) n^{1-1/q}
\end{equation*}
can not hold.

\myskip 
\begin{proof}[Proof of Lemma~\ref{lemma:emp_proc_ellipse}]
Throughout this proof, constants might depend on $\alpha_1, \alpha_2$.
Let us define, for any $S \in \mcS_d$:
\begin{equation}\label{eq:def_XY_processes}
   \begin{dcases}
        Y(S) &\coloneqq \sum_{\mu=1}^n |\Tr (W_\mu S) |^r, \\
        X(S) &\coloneqq Y(S) - \EE Y(S).
   \end{dcases} 
\end{equation}
For a set $E \subseteq \mcS_d$, a norm $\|\cdot\|$ on $\mcS_d$, and a value $\eps > 0$, 
an $\eps$-net of $E$ is a set $A \subseteq E$ such that every point of $E$ is at distance at most $\eps$ from $A$ (for the distance induced by $\|\cdot\|$).
We define the covering number
$\mcN(E, \|\cdot\|, \eps)$ as the smallest possible cardinality of an $\eps$-net of $E$.
We fix $\eps \in (0,1)$. 
It follows from classical covering number bounds~\cite{van2014probability} that if $T \coloneqq \{S \in \mcS_d \, : \, \|S \|_\op = 1\}$, 
then
\begin{equation}\label{eq:covering_number_op_norm_ball}
    \log \mcN(T, \| \cdot \|_\op, \eps) \leq \frac{d(d+1)}{2} \log \frac{3}{\eps}.
\end{equation}
Let us fix $N$ an $\eps$-net of $T$ for $\|\cdot \|_\op$, of minimal cardinality.
If we let $S^\star \coloneqq \argmax_{\|S\|_\op = 1} Y(S)$,
and $S_0 \in N$ with $\|S^\star - S_0\|_\op \leq \eps$, then we have by Minkowski's inequality (recall $r \geq 1$):
\begin{align*}
    \| \{\Tr[W_\mu S^\star]\}_{\mu=1}^n\|_r - \| \{\Tr[W_\mu S_0]\}_{\mu=1}^n\|_r
    &\leq \| \{\Tr[W_\mu (S_0 - S^\star)]\}_{\mu=1}^n\|_r \\ 
    &\leq \eps \max_{\|S\|_\op = 1} \| \{\Tr[W_\mu S]\}_{\mu=1}^n\|_r.
\end{align*}
This implies 
\begin{equation}\label{eq:reduction_net}
    \max_{\|S\|_\op = 1} \sum_{\mu=1}^n |\Tr (W_\mu S) |^r \leq \frac{1}{(1-\eps)^r} \max_{S \in N} \sum_{\mu=1}^n |\Tr (W_\mu S) |^r.
\end{equation}
We combine the covering number upper bound of eq.~\eqref{eq:covering_number_op_norm_ball} and the relation of eq.~\eqref{eq:reduction_net} with the following lemma, 
proven later on, which bounds the deviation probability of the process for a given $S$. 
\begin{lemma}[Tail bound at a fixed point]\label{lemma:ub_ellipse_process_op_sphere}
    \noindent 
    With the notations of eq.~\eqref{eq:def_XY_processes}, we have, for any 
    $S \in T$:
    \begin{itemize}
        \item[$(i)$] $\EE \,[ Y(S) ] \leq C_1 n$.
        \item[$(ii)$] For all $t \geq 0$: 
        \begin{equation*}
            \bbP[X(S) \geq n t] \leq 2 \exp\left\{-C_2 \min (n t^2, n t^{\frac{2}{r}}, n^{\frac{1}{4} + \frac{1}{r}} t^{\frac{1}{r}} )\right\}.
        \end{equation*}
    \end{itemize}
    Note that the above constants may depend on $r$. 
\end{lemma} 
\noindent
Picking $\eps = 1/2$ and performing a union bound over $N$, we reach 
using eqs.~\eqref{eq:covering_number_op_norm_ball},\eqref{eq:reduction_net} and Lemma~\ref{lemma:ub_ellipse_process_op_sphere}: 
\begin{equation*}
    \bbP[\sup_{S \in T} Y(S) \geq n (C_1 + t)] \leq 2\exp\left\{C_3 n - C_2 \min (n t^2, n t^{\frac{2}{r}}, n^{\frac{1}{4} + \frac{1}{r}} t^{\frac{1}{r}} )\right\}.
\end{equation*}
We thus have for any $t \geq 1$:
\begin{equation*}
    \bbP[\sup_{S \in T} Y(S) \geq n (C_1 + t)] \leq 2
    \begin{dcases}
        \exp\left\{C_3 n - C_2 n t^{\frac{2}{r}} \right\} \hspace{1cm} &\textrm{ if } t \leq n^{1 - \frac{3r}{4}},  \\
        \exp\left\{C_3 n - C_2 n^{\frac{1}{4} + \frac{1}{r}} t^{\frac{1}{r}} \right\} \hspace{1cm} &\textrm{ if } t \geq n^{1 - \frac{3r}{4}}.
    \end{dcases}
\end{equation*}
Note that $n^{1/4+1/r} \geq n$ since $r \leq 4/3$.
Therefore, for (new) constants $C_1, C_2$ we have for all $t > 0$:
\begin{equation*}
    \bbP[\sup_{S \in T} Y(S) \geq n (C_1 + t)] \leq 2 \exp\left\{ - C_2 \min (n t^{\frac{2}{r}}, n^{\frac{1}{4} + \frac{1}{r}} t^{\frac{1}{r}} )\right\},
\end{equation*}
which ends the proof.
\end{proof}

\myskip
We now tackle Lemma~\ref{lemma:ub_ellipse_process_op_sphere}.

\begin{proof}[Proof of Lemma~\ref{lemma:ub_ellipse_process_op_sphere}]
We start with $(i)$.
One has $\EE[Y(S)] = n \EE [|\Tr(W_1 S)|^r]$. 
Let $Z \coloneqq \Tr(W_1 S) \deq (x^\T S x - \Tr[S])/\sqrt{d}$, with $x \sim \mcN(0, \Id_d)$.
Since $\|S\|_\op = 1$, we have by Hanson-Wright's inequality \cite{vershynin2018high}:
\begin{align}\label{eq:ub_pz}
    \nonumber
    \bbP[|Z| \geq t] &\leq 2 \exp\Big\{-C \min \Big(\frac{d t^2}{\|S\|_F^2},\sqrt{d} t\Big)\Big\} , \\ 
    &\leq 2 \exp\left\{-C \min \left(t^2,\sqrt{d} t\right)\right\},
\end{align}
where $C > 0$ is an absolute constant, and we used that $\|S\|_F \leq \sqrt{d} \|S\|_\op$.
Separating the sub-Gaussian and the sub-exponential parts of the tail we have, we have for all $p \geq 1$: 
\begin{align*}
    \nonumber
    \EE[|Z|^p] &= p \int_{0}^\infty \rd u \, u^{p-1} \, \bbP[|Z| \geq u], \\ 
    \nonumber
    &\leq 2p \Bigg[\int_{0}^{\sqrt{d}} \rd u \, u^{p-1} \, e^{-Cu^2} + \int_{\sqrt{d}}^{\infty} \rd u \, u^{p-1} \, e^{- C \sqrt{d} u}\Bigg], \\ 
    \nonumber
    &\leq 2p \Bigg[\int_{0}^{\infty} \rd u \, u^{p-1} \, e^{- Cu^2} + e^{-Cd}\int_{0}^{\infty} \rd u \, (u+\sqrt{d})^{p-1} \, e^{- C \sqrt{d} u}\Bigg], \\ 
    \nonumber
    &\aleq 2p \Bigg[\frac{1}{2 C^{p/2}} \Gamma(p/2) + e^{-Cd} \max(1, 2^{p-2})\int_{0}^{\infty} \rd u \, [u^{p-1} + d^{(p-1)/2}] \, e^{- C \sqrt{d} u}\Bigg], \\ 
    \nonumber
    &\leq 2p \Bigg[\frac{\Gamma(p/2)}{2 C^{p/2}}  + e^{-Cd} \max(1, 2^{p-2})\Big(\frac{\Gamma(p)}{C^p d^{p/2}} + \frac{d^{(p-2)/2}}{C}\Big)\Bigg].
\end{align*}
We used in $(\rm a)$ that $(a+b)^x \leq \max(1, 2^{x-1})(a^x + b^x)$ for all $a,b, x \geq 0$.
Using Minkowski's inequality, we reach that for all $p \geq 1$:
\begin{equation}
    \label{eq:ub_Z_pth}
    \EE[|Z|^p]^{1/p} \leq C_1 \sqrt{p} + C_2 e^{-\frac{C_3 d}{p}} \Big(\frac{p}{\sqrt{d}} + d^{\frac{1}{2} - \frac{1}{p}}\Big),
\end{equation}
for some positive constants $(C_a)_{a=1}^3$ independent of $p$ and $d$.
Informally, the sub-Gaussian tail dominates the first moments of $Z$ since the sub-exponential tail only kicks in at the scale $\mcO(\sqrt{d})$.
Eq.~\eqref{eq:ub_Z_pth} implies claim $(i)$ of Lemma~\ref{lemma:ub_ellipse_process_op_sphere} by taking $p = r$ (since the second term goes to $0$ as $d \to \infty$ for any fixed $p$).

\myskip
We turn to $(ii)$.
We make use of classical tail bounds for sub-Weibull random variables, recalled in Lemma~\ref{lemma:tail_sum_sub_Weibull}.
Denoting $Z_\mu \coloneqq \Tr(W_\mu S)$, we have $X(S) = \sum_{\mu=1}^n \{|Z_\mu|^r - \EE[|Z_\mu|^r]\}$.
We decompose $X(S)$ in two parts, i.e.\ $X(S) = X_1(S) + X_2(S)$, with 
\begin{equation}\label{eq:def_Xa_bS}
    \begin{dcases}
        X_1(S) &\coloneqq 
        \sum_{\mu=1}^n \left[\min(|Z_\mu|, \sqrt{d})^r - \EE \{\min(|Z_\mu|, \sqrt{d})^r\}\right], \\
        X_2(S) &\coloneqq 
        \sum_{\mu=1}^n \left([|Z_\mu|^r - d^{r/2}] \indi\{|Z_\mu| > \sqrt{d}\} - \EE\left\{[|Z_\mu|^r - d^{r/2}] \indi\{|Z_\mu| > \sqrt{d}\}\right\}\right).
    \end{dcases}
\end{equation}
We will successively bound $X_1(S), X_2(S)$. To lighten the notations, we do not write their dependency on $S$ in what follows.
Observe that $(Z_\mu)_{\mu=1}^n$ are i.i.d.\ random variables, and that 
they satisfy the tail bound of eq.~\eqref{eq:ub_pz}.

\myskip
\textbf{Bounding $X_1$ -- }
Denoting $T_\mu \coloneqq \min(|Z_\mu|, \sqrt{d})$,
we have $\bbP[T_\mu \geq t] \leq 2 \exp\{-C_2 t^2\}$ by the tail bound of eq.~\eqref{eq:ub_pz}.
Moreover, $\EE[T_\mu^r] \leq \EE [|Z_\mu|^r] \leq C_1$ (depending only on $r$) by eq.~\eqref{eq:ub_Z_pth}.
Therefore, for every $t > C_1$, we have 
\begin{equation*}
    \bbP[|T_\mu^r - \EE [T_\mu^r]| \geq t] = \bbP[T_\mu^r \geq \EE [T_{\mu}^r] + t] \leq 2 e^{-C_2(t+\EE [T_\mu^r])^{2/r}} \leq 2 e^{- C_2t^{2/r}}. 
\end{equation*}
This implies that $\bbP[|T_\mu^r - \EE[T_\mu^r]| \geq t] \leq 2 e^{-C t^{2/r}}$ for all $t \geq 0$ and some (new) constant $C > 0$, depending only on $r$.
We can thus apply $(i)$ of Lemma~\ref{lemma:tail_sum_sub_Weibull} for $q = 2 / r \in [1, 2]$
and $a_i = 1/n$ (so $\|a\|_2^2 = \|a\|_{q^\star}^q = n^{-1}$),
which yields that for all $t \geq 0$
\begin{equation}\label{eq:ub_P_X1}
    \bbP[|X_1| \geq n t] = \bbP\Bigg[\frac{1}{n} \Bigg|\sum_{\mu=1}^n \{T_\mu^r - \EE[T_\mu^r]\} \Bigg| \geq t\Bigg] \leq 2 \exp \left\{-C n \min(t^2, t^{2/r})\right\}.
\end{equation}
\textbf{Bounding $X_2$ --}
We proceed similarly, using $(ii)$ of Lemma~\ref{lemma:tail_sum_sub_Weibull}.
Letting 
\begin{equation*}
    U_\mu \coloneqq d^{r/2}[|Z_\mu|^r - d^{r/2}] \indi\{|Z_\mu| > \sqrt{d}\},
\end{equation*}
then $U_\mu \geq 0$, and moreover, by the Cauchy-Schwarz inequality: 
\begin{align*}
    \EE[U_\mu] &\leq d^{r/2} \sqrt{\EE |Z_\mu|^{2r}} \sqrt{\bbP[|Z_\mu| > \sqrt{d}]}, \\ 
    &\leq C_1 e^{-C_2 d},
\end{align*}
for some $C_1, C_2 > 0$ depending only on $r$, using the moments and tail bound of eqs.~\eqref{eq:ub_pz} and \eqref{eq:ub_Z_pth}.
Repeating the argument used on $T_\mu$ above (using this time the second part of the tail of eq.~\eqref{eq:ub_pz}), we then reach that 
for all $t \geq 0$:
\begin{equation*}
    \bbP[|U_\mu - \EE[U_\mu]| \geq t] \leq 2 e^{-C t^{1/r}}. 
\end{equation*}
We can then apply $(ii)$ of Lemma~\ref{lemma:tail_sum_sub_Weibull} with $q = 1/r \in [1/2,1]$ to reach: 
\begin{align}\label{eq:ub_P_X2}
    \nonumber
    \bbP[|X_2| \geq n t] &= \bbP\Bigg[\frac{1}{n} \Bigg|\sum_{\mu=1}^n \{U_\mu - \EE[U_\mu]\} \Bigg| \geq t d^{r/2}\Bigg] 
    \leq 2 \exp \left\{-C \min(n d^r t^2, d^{1/2} (nt)^{1/r})\right\}, \\ 
    &\leq 2 \exp \left\{-C \min(n^{1 + r/2} t^2, n^{1/4 + 1/r} t^{1/r})\right\},
\end{align}
using that $\alpha_1 d^2 \leq n \leq \alpha_2 d^2$. 

\myskip
We conclude the proof of Lemma~\ref{lemma:ub_ellipse_process_op_sphere} by combining eqs.~\eqref{eq:ub_P_X1} and eq.~\eqref{eq:ub_P_X2}, 
along with the union bound $\bbP[|X| \geq nt] \leq \bbP[|X_1| \geq nt/2] + \bbP[|X_2| \geq nt/2]$.
\end{proof}

\subsection{Free entropy universality for matrix models}\label{subsec:universality_matrix}

\noindent
In this section we state and prove a general universality theorem for 
the asymptotic free entropy in a large class of matrix models, under a ``uniform 
one-dimensional central limit theorem'' assumption (or pointwise normality). We first need to define such an assumption.
\begin{definition}[Uniform pointwise normality]\label{def:one_dimensional_CLT}
    \noindent
    Let $d \geq 1$, and
    $\rho$ a probability distribution on $\mcS_d$.
    We say that $\rho$ satisfies a \emph{one-dimensional CLT with respect to the set $A_d \subseteq \mcS_d$} if:
    \begin{itemize}
        \item[$(i)$] The mean and covariance of $\rho$ are matching the $\mathrm{GOE}(d)$ distribution, i.e.\
        for $W \sim \rho$ and $G \sim \mathrm{GOE}(d)$, we have
        $\EE[W] = \EE[G] = 0$ and for all 
        $i\leq j$ and $k \leq l$: $\EE[W_{ij} W_{kl}] = \EE[G_{ij} G_{kl}] = \delta_{ik} \delta_{jl} (1+\delta_{ijkl})/d$.
        \item[$(ii)$] For any bounded Lipschitz function $\varphi$, we have:
    \begin{equation}\label{eq:1d_clt}
        \lim_{d \to \infty} \sup_{S \in A_d} \Big| \EE_{W \sim \rho}\big[\varphi\big(\Tr[W S]\big)\big] - \EE_{G \sim \mathrm{GOE}(d)}\big[\varphi\big(\Tr[G S]\big)\big]\Big| = 0.
    \end{equation}
    \end{itemize}
\end{definition}
\noindent
We can now state the universality theorem for the free entropy.
Its proof is in great part an adaptation of the proof arguments for Theorem~1 and Lemma~1 in \cite{montanari2022universality} (see also \cite{hu2022universality,gerace2024gaussian,dandi2023universality}).
We sketch the ideas of its proof in the following, deferring some technicalities and adaptations of the arguments of \cite{montanari2022universality} to appendices.
\begin{theorem}[Free entropy universality for matrix models]\label{thm:universality_matrix}
    \noindent  
    Let $n,d \geq 1$ and $n,d \to \infty$ with $\alpha_1 d^2 \leq n \leq \alpha_2 d^2$ for some $0 < \alpha_1 \leq \alpha_2$.
    We are given:
    \begin{enumerate}[label=(\roman*),ref=(\roman*)]
        \item\label{hyp:P0} $P_0$ a probability distribution on $\mcS_d$, such that $\supp(P_0) \subseteq B_2(C_0 \sqrt{d})$, for a constant $C_0 > 0$.
        \item\label{hyp:phi} $\phi : \bbR\to\bbR_+$ a bounded differentiable function with bounded derivative.
        \item\label{hyp:Ad} A series of symmetric convex sets $A_d$ such that $\supp(P_0) \subseteq A_d$.
        \item\label{hyp:rho} $\rho$ a probability distribution on $\mcS_d$, which satisfies a one-dimensional CLT with respect to $A_d$ as per Definition~\ref{def:one_dimensional_CLT}.
    \end{enumerate}
    For $W_1, \cdots, W_n \in \mcS_d$ we define the free entropy:
    \begin{equation}\label{eq:def_fe_matrix_universality}
        F_d(\{W_\mu\}) \coloneqq \frac{1}{d^2} \log \int P_0(\rd S) \exp\Big\{-\sum_{\mu=1}^n \phi\left(\Tr[W_\mu S]\right)\Big\}.
    \end{equation}
    Then for any bounded differentiable function $\psi$ with bounded Lipschitz derivative we have
    \begin{equation}\label{eq:equivalence_asymptotic_fe}
        \lim_{d \to \infty} \left| \EE_{\{W_\mu\} \iid \rho} \psi[F_d(\{W_\mu\})] - \EE_{\{G_\mu\} \iid \mathrm{GOE}(d)} \psi[F_d(\{G_\mu\})] \right| = 0.
    \end{equation}
\end{theorem}
\noindent
\textbf{Remark I --} One could straightforwardly weaken the hypothesis $\supp(P_0) \subseteq A$ in Theorem~\ref{thm:universality_matrix} 
to the weaker condition $d^{-2} \log P_0(A^c) \to -\infty$ as $d \to \infty$.

\myskip
\textbf{Remark II --} 
Note that our setup differs slightly from the one of \cite{montanari2022universality}, as 
we consider distributions $P_0$ with possibly continuous support, and (more importantly) for a fixed $S \in \mcS_d$,
the projections $\{\Tr[W_\mu S]\}_{\mu=1}^n$ are not sub-Gaussian when $W_\mu \sim \Ell(d)$, but only 
sub-exponential. Nevertheless, we will see that the approach of \cite{montanari2022universality} can 
in large part be adapted to prove Theorem~\ref{thm:universality_matrix}, thanks to the results we showed in Section~\ref{subsec:lipschitz_energy}.

\myskip 
\textbf{Sketch of proof of Theorem~\ref{thm:universality_matrix} --}
Since $\supp(P_0) \subseteq A_d$, the integral in eq.~\eqref{eq:def_fe_matrix_universality} can be restricted to $S \in A_d$.
We make use of an interpolation argument to show the universality of the free entropy.
We define, for $t \in [0, \pi/2]$ and $\mu \in [n]$:
\begin{equation}\label{eq:def_U_tU}
    \begin{dcases}
        U_\mu(t) &\coloneqq \cos(t) W_\mu + \sin(t) G_\mu, \\ 
        \tU_\mu(t) &\coloneqq \frac{\partial U_\mu(t)}{\partial t} = -\sin(t) W_\mu + \cos(t) G_\mu.
    \end{dcases}
\end{equation}
Note that $\{U_\mu\}_{\mu=1}^n$ are still i.i.d., and are smooth functions of $t$.
Moreover, if $W_\mu$ was also a $\GOE(d)$ matrix, then $U_\mu(t), \tU_\mu(t)$ would be independent
$\GOE(d)$ matrices.
By the fundamental theorem of calculus:
\begin{equation}\label{eq:ftc}
    |\EE \psi[F_d(W)] - \EE \psi[F_d(G)]| = \left|\int_{0}^{\pi/2}  \frac{\partial}{\partial t} \{\EE \, \psi[F_d(U(t))]\} \rd t\right| \aleq  \int_{0}^{\pi/2} \left|\EE \frac{\partial \psi[F_d(U(t))]}{\partial t}\right| \, \rd t,
\end{equation}
where $(\rm a)$ follows by dominated convergence since $\psi[F_d(U(t))]$ is continuously differentiable on $[0, \pi/2]$, and the triangular inequality.
We will deduce Theorem~\ref{thm:universality_matrix} if we can show the following two lemmas: 
\begin{lemma}[Domination]\label{lemma:domination}
    \noindent 
    Under the hypotheses of Theorem~\ref{thm:universality_matrix}: 
    \begin{equation*}
    \int_{0}^{\pi/2} \sup_{d \geq 1} \left|\EE \frac{\partial \psi[F_d(U(t))]}{\partial t}\right| \, \rd t < \infty.
    \end{equation*}
\end{lemma}
\begin{lemma}[Pointwise limit]\label{lemma:pointwise_limit}
    \noindent 
    Under the hypotheses of Theorem~\ref{thm:universality_matrix}, for any $t \in (0, \pi/2)$: 
    \begin{equation*}
        \lim_{d \to \infty} \EE \, \frac{\partial \psi[F_d(U(t))]}{\partial t} = 0. 
    \end{equation*}
\end{lemma}
\noindent
Indeed, plugging Lemmas~\ref{lemma:domination} and \ref{lemma:pointwise_limit} in eq.~\eqref{eq:ftc} and taking the $d \to \infty$ limit using dominated convergence ends the proof 
of Theorem~\ref{thm:universality_matrix}.
We therefore focus on proving these two lemmas in the following.

\myskip
As it will be useful, we state the result of the elementary computation of the derivative:
\begin{align}\label{eq:derivative_psi}
    &\frac{\partial \psi[F_d(U(t))]}{\partial t} \\ 
    \nonumber
    &= -\frac{\psi'[F_d(U(t))]}{d^2} \sum_{\mu=1}^n \frac{\int P_0(\rd S) \, e^{-\sum_\nu \phi(\Tr[U_\nu(t) S])} \left(\Tr[S \tU_\mu(t)] \, \phi'(\Tr[U_\mu(t) S])\right)}{\int P_0(\rd S) \, e^{-\sum_\nu \phi(\Tr[U_\nu(t) S])}}.
\end{align}
Because $\{G_\mu, W_\mu\}$ are i.i.d.\ we get further:
\begin{align}\label{eq:ee_derivative_psi}
    &\EE \, \frac{\partial \psi[F_d(U(t))]}{\partial t} \\ 
    \nonumber
    &= - \frac{n}{d^2} \EE\Bigg[\psi'[F_d(U(t))] \frac{\int P_0(\rd S) \, e^{-\sum_\nu \phi(\Tr[U_\nu(t) S])} \left(\Tr[S \tU_1(t)] \, \phi'(\Tr[U_1(t) S])\right)}{\int P_0(\rd S) \, e^{-\sum_\nu \phi(\Tr[U_\nu(t) S])}}\Bigg].
\end{align}
\noindent
Note that if $W_\mu$ was also a $\textrm{GOE}(d)$ matrix, for any $t$, $U_\mu(t)$ and $\tU_\mu(t)$ would be independent $\textrm{GOE}(d)$ matrices. 
The main idea behind the interpolation is that the matrix $W_\mu$ will only appear through some one-dimensional projection with a matrix $S$. 
We will then use Definition~\ref{def:one_dimensional_CLT} to argue that one can effectively replace $W_\mu$ by a $\GOE(d)$ matrix, which by the argument above would mean that 
we can consider the case of independent $\GOE(d)$ matrices $U_\mu(t)$ and $\tU_\mu(t)$. 
In this case, the RHS of eq.~\eqref{eq:ee_derivative_psi} would be $0$, since there is only a single term involving $\tU_1(t)$, 
which has zero mean: this crucial idea is the intuition behind Lemma~\ref{lemma:pointwise_limit}.

\myskip 
The details of the proofs of Lemmas~\ref{lemma:domination} and \ref{lemma:pointwise_limit} are fairly technical 
and substantially follow the ones of their counterparts in \cite{montanari2022universality}. 
For this reason, we defer them to Appendix~\ref{subsec_app:proof_universality_matrix}.

\subsection{Proof of Proposition~\ref{prop:universality_gs}}
\label{subsec:proof_universality_gs}

\subsubsection{Consequences of universality for ellipsoid fitting}\label{subsubsec:universality_ellipse}

We investigate here the consequences of Theorem~\ref{thm:universality_matrix} for the ellipsoid fitting problem.
It follows by the Berry-Esseen central limit theorem \cite{o2014analysis} that the distribution $\Ell(d)$ satisfies uniform pointwise normality 
on a large set of matrices (in the sense of Definition~\ref{def:one_dimensional_CLT}).
\begin{lemma}[One-dimensional CLT for the ellipse problem]\label{lemma:1d_clt_ellipse}
    \noindent
    Let $d \geq 1$ and $W \sim \Ell(d)$.
    Fix any $\eta \in (0,1/2)$ 
    Let $A_d \coloneqq \{S \in \mcS_d \, : \, \Tr[|S|^3] \leq d^{3/2-\eta} \}$.
    Then $A_d$ is convex and symmetric, and
    the law of $W$ satisfies a one-dimensional CLT with respect to $A_d$, 
    in the sense of Definition~\ref{def:one_dimensional_CLT}.
\end{lemma}
\noindent
\textbf{Remark --}
This lemma makes crucial use of the Gaussian nature of the vectors, and more specifically it relies on their rotation invariance and the first moments of their norm, 
as is clear from the proof.
On the other hand, for vectors sampled from other distributions, such as $x \sim \Unif(\{\pm 1\}^d)$ or $x \sim \Unif(\bbS^{d-1})$, 
it is easy to see that Lemma~\ref{lemma:1d_clt_ellipse} can not hold: indeed, $S = \Id_d$ is such that $\Tr[S W] = 0$ deterministically, while $\Tr[S G] = \Tr[G] \sim \mcN(0, 2)$ 
for $G \sim \GOE(d)$.
This is consistent, as in the example of these two distributions there always exists an ellipsoid fit, which is the sphere itself, 
and therefore Theorem~\ref{thm:main_positive_side} can not possibly hold.

\begin{proof}[Proof of Lemma~\ref{lemma:1d_clt_ellipse}]
    Note that $A_d$ is a centered ball for the $S_3$-norm, and is therefore convex and symmetric.
    The proof of the first and second moments condition of Definition~\ref{def:one_dimensional_CLT} is immediate via a simple calculation. We focus on proving condition $(ii)$ 
    of Definition~\ref{def:one_dimensional_CLT}.
    Fix $S \in A_d$, with eigenvalues $(\lambda_i)_{i=1}^d$.
    With $W \sim \Ell(d)$ and $G \sim \GOE(d)$, let
    \begin{equation*}
        \begin{dcases}
            X &\coloneqq \Tr[S W], \\
            Y &\coloneqq \Tr[S G] .
        \end{dcases}
    \end{equation*}
    It is trivial to see that $Y \sim \mcN(0, 2 \Tr[S^2]/d)$, 
    so that $Y \deq d^{-1/2} \sum_{i=1}^d \lambda_i z_i$ for $z_i \iid \mcN(0,2)$.
    Moreover,
    \begin{equation*}
        X \deq \frac{1}{\sqrt{d}} \sum_{i=1}^d \lambda_i (x_i^2 - 1),
    \end{equation*}
    with $x_i \iid \mcN(0,1)$. 
    We use the Berry-Esseen central limit theorem, and in particular the formulation of Chapter~11 of \cite{o2014analysis} -- itself a simple 
    consequence of the Lindeberg exchange method. 
    \begin{lemma}[Corollary~11.59 of \cite{o2014analysis}]\label{lemma:BE_lipschitz}
        \noindent
        There exists a universal constant $C > 0$ such that the following holds.
        Let $p \geq 1$ and $X_1, \cdots, X_p$ and $Y_1, \cdots, Y_p$ be independent random variables, such that 
        $\EE[X_i] = \EE[Y_i]$ and $\EE[X_i^2] = \EE[Y_i^2]$ for all $i \in [p]$.
        Let $\varphi : \bbR \to \bbR$ a Lipschitz function with Lipschitz constant $\|\varphi\|_L$.
        Then 
        \begin{equation*}
            \left|\EE \, \varphi\left(\sum_{i=1}^p X_i\right) - \EE \, \varphi\left(\sum_{i=1}^p Y_i\right)\right| \leq C \|\varphi\|_L \left[\sum_{i=1}^p \left(\EE \, |X_i|^3 + \EE \, |Y_i|^3\right)\right]^{1/3}.
        \end{equation*}
    \end{lemma}
    \noindent
    Lemma~\ref{lemma:BE_lipschitz} yields:
    \begin{equation*}
        | \EE \, \varphi(X) - \EE \, \varphi(Y)| \leq C \|\varphi\|_L \left[B\frac{\Tr[|S|^3]}{d^{3/2}} \right]^{1/3},
    \end{equation*}
    with $B = \EE[|z^2 - 1|^3] + 2^{3/2} \EE |z|^3$ for $z \sim \mcN(0,1)$.
    Using the definition of $A_d$, we reach:
    \begin{equation*}
        \sup_{S \in A_d} | \EE \, \varphi(X) - \EE \, \varphi(Y)| \leq C \|\varphi\|_L \sup_{S \in A_d} \left[\frac{1}{\sqrt{d}} \frac{\Tr |S|^3}{d} \right]^{1/3} \leq C \|\varphi\|_L \, d^{-\eta/3} \to 0.
    \end{equation*}
    This ends the proof.
\end{proof}

\myskip
We can now state the main result of this section, a corollary of Theorem~\ref{thm:universality_matrix} and Lemma~\ref{lemma:1d_clt_ellipse}.
\begin{corollary}[Universality for ellipsoid fitting]\label{cor:universality_ellipse}
    \noindent  
    Let $n,d \geq 1$ and $n,d \to \infty$ with $\alpha_1 d^2 \leq n \leq \alpha_2 d^2$ for some $0 < \alpha_1 \leq \alpha_2$. 
    Let $P_0$ be a probability distribution such that $\supp(P_0) \subseteq B_\op(C_0)$ for some constant $C_0 > 0$.
    Let $\phi : \bbR \to \bbR_+$ a bounded differentiable function with bounded derivative.
    For $X_1, \cdots, X_n \in \mcS_d$ we define the free entropy:
    \begin{equation*}
        F_d(\{X_\mu\}) \coloneqq \frac{1}{d^2} \log \int P_0(\rd S) \exp\Big\{-\sum_{\mu=1}^n \phi\left(\Tr[X_\mu S]\right)\Big\}.
    \end{equation*}
    Then for any $\psi$ such that $\|\psi\|_\infty, \|\psi'\|_\infty, \|\psi'\|_L < \infty$ we have
    \begin{equation}\label{eq:equivalence_asymptotic_fe_ellipse}
        \lim_{d \to \infty} \left| \EE_{\{W_\mu\} \iid \Ell(d)} \psi[F_d(\{W_\mu\})] - \EE_{\{G_\mu\} \iid \mathrm{GOE}(d)} \psi[F_d(\{G_\mu\})] \right| = 0.
    \end{equation}
\end{corollary}
\noindent
We notice that the only requirement for Corollary~\ref{cor:universality_ellipse} to hold is $\supp(P_0) \subseteq B_F(C\sqrt{d}) \cap B_3(d^{3/2-\eta})$ for some $C > 0$ and $\eta > 0$, a weaker requirement than $\supp(P_0) \subseteq B_\op(C_0)$.

\begin{proof}[Proof of Corollary~\ref{cor:universality_ellipse}]
Since $B_\op(C_0) \subseteq B_2(C_0 \sqrt{d})$, hypothesis~\ref{hyp:P0} of Theorem~\ref{thm:universality_matrix} is satisfied.
Condition~\ref{hyp:phi} is satisfied by hypothesis.
Since $B_\op(C_0) \subseteq B_3(C_0 d^{1/3}) \subseteq B_3(d^{1/2 - \eta})$
for any $\eta \in (0, 1/6)$, condition~\ref{hyp:Ad} of Theorem~\ref{thm:universality_matrix} is satisfied with $A_d = B_3(d^{1/2 - \eta})$. 
Finally, Lemma~\ref{lemma:1d_clt_ellipse} verifies condition~\ref{hyp:rho} for this choice of $A_d$.
All in all we can apply Theorem~\ref{thm:universality_matrix}, from which the conclusion follows.
\end{proof}

\subsubsection{Proof of Proposition~\ref{prop:universality_gs}}
\label{subsubsec:proof_universality_gs}

\noindent 
We are now ready to prove Proposition~\ref{prop:universality_gs}, taking a ``small-temperature'' limit. 
Such arguments are classical in rigorous statistical mechanics, see e.g.\ Appendix~A of \cite{montanari2022universality}.
Notice that the restriction $B \subseteq B_\op(C_0)$ will be critical because we proved an upper bound on the Lipschitz constant of the energy 
for the operator norm, cf.\ Lemma~\ref{lemma:energy_change_small_ball}.
Recall the definition of the energy function:
\begin{equation*}
    E_{\{X_\mu\}}(S) \coloneqq \frac{1}{d^2} \sum_{\mu=1}^n \phi[\Tr(X_\mu S)].
\end{equation*}
We fix $\eta \in (0,1)$, and $\mcN_\eta \subseteq B$ a minimal $\eta$-net of $B$ 
for $\|\cdot\|_\op$. 
Since $B \subseteq B_\op(C_0)$ and 
$\dim(\mcS_d) = d(d+1)/2$, it follows by standard covering number upper bounds \cite{vershynin2018high,van2014probability} that 
\begin{equation}\label{eq:net_upper_bound}
  \log |\mcN_\eta| = \log \mcN(B, \|\cdot\|_\op, \eta) 
  \leq \log \mcN\left(B_\op(C_0), \|\cdot\|_\op, \frac{\eta}{2}\right)
  \leq d^2 \log \frac{K}{\eta} ,
\end{equation}
for some $K > 0$ depending on $C_0$.
Recall the definition of $\GS_d(\{X_\mu\})$ in eq.~\eqref{eq:def_gs}.
We define:
\begin{equation*}
    \GS_d(\eta,\{X_\mu\}) \coloneqq \inf_{S \in \mcN_\eta} E_{\{X_\mu\}}(S).
\end{equation*}
We will show the two lemmas: 
\begin{lemma}\label{lemma:universality_GS_net}
    \noindent
    For any $\eta > 0$ and any $\psi$ such that $\|\psi\|_\infty, \|\psi'\|_\infty, \|\psi'\|_L < \infty$: 
    \begin{equation*}
        \lim_{d \to \infty} \left| \EE_{\{W_\mu\} \iid \Ell(d)} \psi[\GS_d(\eta,\{W_\mu\})] - \EE_{\{G_\mu\} \iid \mathrm{GOE}(d)} \psi[\GS_d(\eta,\{G_\mu\})] \right| = 0.
    \end{equation*}
\end{lemma}
\begin{lemma}\label{lemma:GS_net}
    \noindent
    Let $X_1, \cdots, X_n \iid \rho$, with $\rho \in \{\GOE(d),\Ell(d)\}$. Then, with probability 
    at least $1 - 2 e^{-n}$:
    \begin{equation*}
        |\GS_d(\eta,\{X_\mu\}) - \GS_d(\{X_\mu\})| \leq C \|\phi'\|_\infty \cdot \eta.
    \end{equation*}
\end{lemma}
\noindent
These results are proven in the following, let us first see how
they end the proof of Proposition~\ref{prop:universality_gs}.
We fix $\eta \in (0,1)$.
    We have
    \begin{align*}
        \left| \EE \psi[\GS_d(\{W_\mu\})] - \EE \psi[\GS_d(\{G_\mu\})] \right| 
        \leq &\left| \EE \psi[\GS_d(\eta,\{W_\mu\})] - \EE \psi[\GS_d(\eta, \{G_\mu\})] \right| \\
        &+ \|\psi\|_L \sum_{X \in \{W, G\}} \EE |\GS_d(\eta,\{X_\mu\}) - \GS_d(\{X_\mu\})|.
    \end{align*}
    The first term goes to $0$ as $d \to \infty$ by Lemma~\ref{lemma:universality_GS_net}.
    Finally, using Lemma~\ref{lemma:GS_net} and the Cauchy-Schwarz inequality, we have: 
    \begin{align*}
        \EE |\GS_d(\eta,\{X_\mu\}) - \GS_d(\{X_\mu\})| 
        &\leq e^{-n/2} \! \sqrt{2\EE |\GS_d(\eta,\{X_\mu\}) - \GS_d(\{X_\mu\})|^2} + C \|\phi'\|_\infty \eta, \\ 
        &\aleq 4 \alpha \|\phi\|_\infty e^{-n/2} + C \|\phi'\|_\infty \eta,
    \end{align*}
    using in $(\rm a)$ that $|E(S)| \leq \alpha \|\phi\|_\infty$.
    Letting $d \to \infty$, we get
    \begin{equation*}
        \lim_{d \to \infty} \left| \EE \psi[\GS_d(\{W_\mu\})] - \EE \psi[\GS_d(\{G_\mu\})] \right| \leq C \|\phi'\|_\infty \|\psi\|_L \cdot \eta.
    \end{equation*}
    Taking the limit $\eta \to 0$ ends the proof of eq.~\eqref{eq:universality_gs}.
    The claim of eq.~\eqref{eq:universality_gs_probabilities} can be obtained easily by picking $\psi$ approximating an indicator function, 
    see e.g.\ Section~A.1.3 of \cite{montanari2022universality} for a detail of this argument.
    $\qed$

\begin{proof}[Proof of Lemma~\ref{lemma:universality_GS_net}]
We define, for $\beta > 0$: 
\begin{equation*}
    F_d(\eta, \beta, \{X_\mu\}) \coloneqq \frac{1}{d^2 \beta} \log \frac{1}{|\mcN_\eta|} \sum_{S \in \mcN_\eta} \exp\Big\{-\beta \sum_{\mu=1}^n \phi\left(\Tr[X_\mu S]\right)\Big\}.
\end{equation*}
Using Corollary~\ref{cor:universality_ellipse} with $P_0$ being the uniform distribution over $\mcN_\eta$, we have, for any $\beta > 0$:
\begin{equation}\label{eq:universality_fe_net}
    \lim_{d \to \infty} \left| \EE \psi[F_d(\eta, \beta, \{W_\mu\})] - \EE \psi[F_d(\eta, \beta, \{G_\mu\})] \right| = 0.
\end{equation}
Moreover, for any fixed $d, \eta$, we have $\GS_d(\eta, \{X_\mu\}) = \lim_{\beta \to \infty} F_d(\eta, \beta, \{X_\mu\})$.
Thus: 
\begin{equation}\label{eq:gs_fenergy_net}
    |\GS_d(\eta, \{X_\mu\}) - F_d(\eta, \beta, \{X_\mu\})| \leq \int_{\beta}^\infty \left|\frac{\partial F_d(\eta, s, \{X_\mu\})}{\partial s}\right| \rd s.
\end{equation}
Defining the ``Gibbs'' measure for $S \in \mcN_\eta$:
\begin{equation*}
    \bbP_{\beta}(S) \coloneqq \frac{\exp\left\{-\beta \sum_{\mu=1}^n \phi\left(\Tr[X_\mu S]\right)\right\}}{\sum_{S' \in \mcN_\eta} \exp\left\{-\beta \sum_{\mu=1}^n \phi\left(\Tr[X_\mu S']\right)\right\}},
\end{equation*}
it is easy to check that
\begin{align*}
    \left|\frac{\partial F_d(\eta, s, \{X_\mu\})}{\partial s}\right|
     &= \frac{1}{s^2 d^2} \left|\sum_{S \in \mcN_\eta} \bbP_{s}(S) \log \bbP_{s}(S) + \log |\mcN_\eta| \right|, \\ 
    &\aleq \frac{1}{s^2 d^2} \log |\mcN_\eta|, \\ 
    &\bleq \frac{1}{s^2} \log \frac{K}{\eta},
\end{align*}
where $(\rm a)$ follows from the fact that, the uniform distribution over $\mcN_\eta$
maximizes the entropy, and $(\rm b)$ is eq.~\eqref{eq:net_upper_bound}.
Plugging the result back in eq.~\eqref{eq:gs_fenergy_net} we get:
\begin{align}\label{eq:gs_fenergy_net_2}
    \nonumber
    |\GS_d(\eta, \{X_\mu\}) - F_d(\eta, \beta, \{X_\mu\})| 
    &\leq \log \left(\frac{K}{\eta}\right)  \int_{\beta}^\infty \frac{\rd s}{s^2}, \\ 
    &\leq\frac{1}{\beta} \log \left(\frac{K}{\eta}\right).
\end{align}
Combining eqs.~\eqref{eq:universality_fe_net} and \eqref{eq:gs_fenergy_net_2} we get, for any $\beta > 0$:
\begin{align*}
    &\limsup_{d \to \infty} \left| \EE \psi[\GS_d(\eta, \{W_\mu\})] - \EE \psi[\GS_d(\eta, \{G_\mu\})] \right| \\
    & \leq \|\psi\|_L \limsup_{d \to \infty} \sum_{X \in \{W, G\}} \EE |\GS_d(\eta, \{X_\mu\}) - F_d(\eta, \beta, \{X_\mu\})|, \\ 
    &\leq \frac{2 \|\psi\|_L}{\beta} \log \left(\frac{K}{\eta}\right).
\end{align*}
Taking the limit $\beta \to \infty$ ends the proof of Lemma~\ref{lemma:universality_GS_net}.
\end{proof}

\begin{proof}[Proof of Lemma~\ref{lemma:GS_net}]
Note that $\GS_d(\eta, \{X_\mu\}) \geq \GS_d(\{X_\mu\}) $ since $\mcN_\eta \subseteq B$.
The other side of this inequality is a direct consequence of Lemma~\ref{lemma:energy_change_small_ball}. Indeed, assuming that $E(S)$ is $C \|\phi'\|_\infty$-Lipschitz with respect to the operator norm, 
let us fix $S^\star \in B$ such that $E(S^\star) = \GS_d(\{X_\mu\})$ (since $B$ is closed and bounded it is compact, therefore this minimizer exists).
Letting $S \in \mcN_\eta$ such that $\|S^\star - S\|_\op \leq \eta$, we have 
\begin{align*}
    \GS_d(\{X_\mu\}) &= E(S^\star), \\ 
    &\geq E(S) - |E(S^\star) - E(S)|, \\ 
    &\geq  \GS_d(\eta, \{X_\mu\}) - C \|\phi'\|_\infty \cdot \eta,
\end{align*}
which ends the proof.
\end{proof}

\section*{Acknowledgements}

\noindent
The authors are grateful to March Boedihardjo, Tim Kunisky, Petar Nizi\'c-Nikolac, and Joel Tropp for insightful discussions and suggestions.

\bibliography{refs}

\newpage
\appendix 
\addtocontents{toc}{\protect\setcounter{tocdepth}{1}} 

\section{A classical concentration inequality}\label{sec_app:classical}
\noindent
We will make use of the following elementary concentration inequality, a generalization of
Bernstein's inequality for $\psi_q$ tails \cite{talagrand1994supremum,hitczenko1997moment,adamczak2011restricted}. 
\begin{lemma}[Tail of sum of independent sub-Weibull random variables \cite{adamczak2011restricted} --]
    \label{lemma:tail_sum_sub_Weibull}
    \noindent 
    Let $q \in (0,2]$, and $W_1, \cdots, W_n$ be i.i.d.\ centered random variables satisfying $\bbP[|W_1| \geq t] \leq 2 e^{-C t^q}$.
    \begin{itemize}
        \item[$(i)$] If $q \in [1,2]$, then for all $a \in \bbR^n$ and all $t \geq 0$: 
        \begin{equation*}
            \bbP \left[\left|\sum_{\mu=1}^n a_\mu W_\mu\right| \geq t\right] \leq 2 \exp\Big\{-c \min\left(\frac{t^2}{\|a\|_2^2}, \frac{t^q}{\|a\|_{q^\star}^q}\right)\Big\},
        \end{equation*}
        where $q^\star \in [2, + \infty]$ with $1/q + 1/q^\star = 1$.
        \item[$(ii)$]
        If $q \in [1/2, 1]$, then for all $t > 0$
        \begin{equation*}
            \bbP \left[\left|\frac{1}{n}\sum_{\mu=1}^n W_\mu\right| \geq t\right] \leq 2 \exp\left\{-c_q \min(n t^2, (nt)^q)\right\}.
        \end{equation*}
    \end{itemize}
\end{lemma}
\noindent
This lemma is stated in \cite{adamczak2011restricted}, see Lemmas~3.6 and 3.7 -- and eq.~(3.7) -- and is a consequence 
of the same result for symmetric Weibull random variables \cite{hitczenko1997moment}.

\section{Fitting error of the sphere}\label{sec_app:error_identity}
\noindent
We show here eq.~\eqref{eq:error_identity}.
By Bernstein's inequality \cite{vershynin2018high}, we have for all $\mu \in [n]$ and $u \geq 0$: 
\begin{equation*}
    \bbP\left[\left|\frac{\|x_\mu\|^2}{d} - 1\right|\geq u\right] \leq 2 \exp\left(-C d \min(u,u^2)\right).
\end{equation*}
As a consequence, if $X_\mu \coloneqq \sqrt{d}(\|x_\mu\|^2/d - 1)$, then\footnote{Where for $q \in [1, \infty)$, 
we defined the Orlicz norm of a random variable $X$ as $\|X\|_{\psi_q} \coloneqq \inf\{t > 0 \, : \, \EE \exp(|X|^q / t^q) \leq 2\}$. 
In particular if $\|X\|_{\psi_1} < \infty$ then $X$ is said to be a \emph{sub-exponential} random variable. We 
refer to \cite{vershynin2018high} for more details on these classical definitions.
} $\|X_\mu\|_{\psi_1} \leq C$.
Let $Y_\mu \coloneqq |X_\mu|^r - \EE[|X_\mu^r|]$.

\myskip
By the central limit theorem, $X_\mu \dto \mcN(0, 2)$ as $d \to \infty$.
One shows easily that (e.g.\ for $\eps = 2$):
\begin{equation*}
  \sup_{d \geq 1} \EE[|X_\mu|^{2+\eps}] < \infty,  
\end{equation*}
and thus (since $r \leq 2$) $|X_\mu|^r$ is uniformly integrable as $d \to \infty$.
This implies (cf.\ Theorem~3.5 of \cite{billingsley2013convergence}) that $\EE|X_\mu|^r \to \EE|Z|^r$ with $Z \sim \mcN(0,2)$. 
Notice that $\EE[|Z|^r] = 2^r \Gamma([r+1]/2)/\sqrt{\pi}$.

\myskip 
Let $q \coloneqq 1/r \in [1/2,1]$.
Since $\|X_\mu\|_{\psi_1} \leq C$, we have $\| |X_\mu|^r \|_{\psi_q} \leq C$, 
and thus $\|Y_\mu\|_{\psi_q} \leq C'$.
We can then use Lemma~\ref{lemma:tail_sum_sub_Weibull}, and we get: 
\begin{equation*}
    \bbP\left[\left|\frac{1}{n} \sum_{\mu=1}^n (|X_\mu|^r - \EE[|X_\mu|^r])\right| \geq t\right]
    \leq 2 \exp\left\{-c_r \min(nt^2, (nt)^{1/r})\right\}.
\end{equation*}
Combining the above, we get that for any $\eps > 0$, we have with probability $1 - \smallO_d(1)$:
\begin{equation*}
 \EE[|Z|^r] - \eps \leq \frac{1}{n} \sum_{\mu=1}^n \left|\sqrt{d} \left[\frac{\|x_\mu\|^2}{d} - 1\right]\right|^r \leq \EE[|Z|^r] + \eps.
\end{equation*}

\section{Towards exact ellipsoid fitting}\label{sec_app:geometric_approx_exact}
\noindent
We show here the following proposition.
\begin{proposition}[From approximate to exact ellipsoid fitting]\label{prop:bound_to_exact_conjecture}
    \noindent
    Let $n, d \to \infty$ with $n / d^2 \to \alpha \in (0, 1/2)$.
    Let $\{X_\mu\}_{\mu=1}^n$ be symmetric random matrices, and $H_{\mu \nu} \coloneqq \Tr[X_\mu X_\nu]$. 
    Assume that:
    \begin{itemize}
        \item[$(i)$] With probability $1 - \smallO_d(1)$, 
        $\|H^{-1}\|_\op \leq C n^{-1/2}$ for some $C = C(\alpha) > 0$.
        \item[$(ii)$] There exists $\lambda_- > 0$ (depending only on $\alpha$) such that
        \begin{equation}\label{eq:necessary_bound}
            \plim_{n \to \infty}
            \min_{S \succeq \lambda_- \Id_d} \frac{1}{\sqrt{n}}\sum_{\mu=1}^n |\Tr(X_\mu S) - 1 |^2 = 0.
        \end{equation}
    \end{itemize}
    Then, with probability $1 - \smallO_d(1)$ there exists $S \succeq 0$ such that $\Tr[X_\mu S] = 1$ for all $\mu \in [n]$.
\end{proposition}
\noindent
Let us emphasize that given our current proof of Theorem~\ref{thm:main_positive_side} (cf.\ Section~\ref{subsec:proof_thm_main_positive_side}), 
if the assumptions of Proposition~\ref{prop:bound_to_exact_conjecture} hold for $X_\mu \sim \Ell(d)$ and $\alpha < 1/4$, then the first part of Conjecture~\ref{conj:ellipsoid_fitting} will hold.

\myskip
However, the proof of Proposition~\ref{prop:bound_to_exact_conjecture} is rather naive, as it crudely bounds the operator norm distance 
of the minimizer of eq.~\eqref{eq:necessary_bound} to a subspace $V$ (the affine subspace of solutions to the linear constraints $\Tr[X_\mu S] = 1$)
by its distance in Frobenius norm.
For these reasons, the assumptions of Proposition~\ref{prop:bound_to_exact_conjecture} may be far from being optimal.

\myskip
\textbf{Remark I --}
Note that the condition $(i)$ is clearly satisfied if $X_\mu \iid \GOE(d)$ and $\alpha \in (0, 1/2)$, 
since $H$ is then distributed as a Wishart matrix.
On the other hand, while we expect it to hold as well for $X_\mu \iid \Ell(d)$, this condition is (to the best of our knowledge) not known unless $\alpha$ is small enough: 
interestingly, this was one of the limitations in the recent works \cite{bandeira2023fitting,tulsiani2023ellipsoid,hsieh2023ellipsoid} that proved that ellipsoid fitting is feasible 
for $\alpha$ sufficiently small.

\myskip
\textbf{Remark II --}
Note that it is sufficient for eq.~\eqref{eq:necessary_bound} 
to hold that there exists $r \in [1,2]$ such that:
\begin{equation*}
    \plim_{n \to \infty}
    \min_{S \succeq \lambda_- \Id_d} \frac{1}{n^{r/4}}\sum_{\mu=1}^n |\Tr(X_\mu S) - 1 |^r = 0.
\end{equation*}
At the moment, our proof of Theorem~\ref{thm:main_positive_side} (cf.\ Lemma~\ref{lemma:positive_side_strong_phi}) only implies (for $r < 4/3$) a similar statement with a prefactor $1/n$ rather than the required $1/n^{r/4}$.
It would thus need to be improved to show that eq.~\eqref{eq:necessary_bound} holds for the ellipsoid fitting setting.

\begin{proof}[Proof of Proposition~\ref{prop:bound_to_exact_conjecture}]
    Let $V \coloneqq \{S \in \mcS_d \, : \, \Tr[X_\mu S] = 1, \, \forall \mu \in [n]\}$ be the affine space of solutions to
     the constraints.
    Let $\eps > 0$, and $\hS \succeq \lambda_- \Id_d$ such that 
    \begin{equation*}
        \frac{1}{\sqrt{n}}\sum_{\mu=1}^n |\Tr(X_\mu \hS) - 1 |^2 \leq \eps.
    \end{equation*}
    Note that for all $M \in \mcS_d$, 
    if $\|M\|_\op \leq \lambda_-$, then $\hS + M \succeq 0$.
    In particular, $\|M\|_\rF \leq \lambda_- \Rightarrow \hS + M \succeq 0$.
    In order to conclude it thus suffices to show that 
     $d_\rF(\hS, V) \leq \lambda_-$.
     The following lemma is an elementary geometrical result: 
     \begin{lemma}[Euclidean distance to an affine subspace --]\label{lemma:distance_subspace}
       \noindent 
       Let $d \geq 1$ and $1 \leq r \leq d$ two integers.
       Let $(a_k, b_k)_{k=1}^r \in (\bbR^d \times \bbR)^r$, with $(a_k)_{k=1}^r$ linearly independent.
       We define $G \coloneqq \{x \in \bbR^d \, : \, a_k^\intercal x + b_k = 0, \ \forall k \in [r]\}$.
       Then, for any $y \in \bbR^d$:
       \begin{equation*}
        d_2(y, G)^2 = v^\intercal H^{-1} v,
       \end{equation*}
       in which $v_k \coloneqq a_k^\intercal y + b_k$, and $H_{k k'} = \langle a_k , a_{k'} \rangle$ is the Gram matrix of the $\{a_k\}$.
     \end{lemma}
    \noindent
     We now condition on the event of condition $(i)$.
     For any $\eps > 0$,
     applying Lemma~\ref{lemma:distance_subspace} yields (with probability $1 - \smallO_d(1)$):
     \begin{equation*}
        d_\rF(\hS, V)^2 \leq C(\alpha) \eps,
     \end{equation*}
     so that picking $\eps \leq \lambda_-^2 / C(\alpha)$ implies the result.
\end{proof}

\section{Additional proofs for Proposition~\ref{prop:universality_gs}}\label{sec_app:technical_universality}
\subsection{Proof of Lemma~\ref{lemma:emp_proc_gaussian}}
\label{subsec_app:proof_lemma_emp_proc_gaussian}
    \noindent
    The lemma is a direct corollary of the following elementary result.
    \begin{proposition}[Bounding a Gaussian process on the sphere]\label{prop:bound_G_process}
        \noindent
        Let $n, p \geq 1$. Let $G \in \bbR^{n \times p}$ with $G_{ij} \iid \mcN(0,1)$. 
        There is $A > 0$ such that for all $\delta > 0$ and $r \geq 1$:
        \begin{equation*}
            \bbP\Big[\max_{\|x\|_2 = 1} \|G x\|_r \geq \left[A(\sqrt{n} + \sqrt{p}) + \delta \sqrt{n}\right] n^{\max(1/r - 1/2,0)}\Big] \leq \exp\Big\{-\frac{n\delta^2}{2}\Big\}.
        \end{equation*}
    \end{proposition}
\begin{proof}[Proof of Proposition~\ref{prop:bound_G_process}]
    Notice that $G \to \max_{\|x \|_2 = 1} \|G x \|_r$ is Lipschitz: 
    \begin{align*}
        \left|\max_{\|x\|_2 = 1} \| G_1 x \|_r - \max_{\|x\|_2 = 1} \| G_2 x \|_r\right| 
        &\leq \max_{\|x\|_2 = 1} \| (G_1  - G_2) x \|_r, \\ 
        &\aleq n^{\max(\frac{1}{r} - \frac{1}{2},0)} \max_{\|x\|_2 = 1} \| (G_1  - G_2) x \|_2, \\
        &\leq n^{\max(\frac{1}{r} - \frac{1}{2},0)} \|G_1 - G_2 \|_\rF,
    \end{align*}
where we used H\"older's inequality in the form $\| y \|_r \leq n^{\frac{1}{r} - \frac{1}{2}} \|y \|_2$ for $r \leq 2$, 
and for $r \geq 2$ the fact that $\|y \|_r \leq \|y \|_2$, alongside with $\|G\|_\op \leq \|G\|_F$.
By Theorem~\ref{thm:gaussian_conc_lipschitz} we reach:
\begin{equation*}
    \bbP\Big[\max_{\|x\|_2 = 1} \| G x \|_r \geq \EE \max_{\|x\|_2 = 1} \| G x \|_r + \delta n^{\max(1/r,1/2)}] 
    \leq \exp\Big\{- \frac{n \delta^2}{2}\Big\}.
\end{equation*}
The proof is complete if we can show that $\EE \max_{\|x\|_2 = 1} \| G x \|_r = \mcO(n^{\max(1/r-1/2,0)}) \cdot (\sqrt{n} + \sqrt{p})$.
This follows by the same inequality as above: 
\begin{equation*}
    \EE \max_{\|x\|_2 = 1} \| G x \|_r \leq n^{\max(\frac{1}{r} - \frac{1}{2},0)} \EE \max_{\|x\|_2 = 1} \| G x \|_2.
\end{equation*}
The bound $\EE \max_{\|x\|_2 = 1} \| G x \|_2 = \mcO(\sqrt{n} + \sqrt{p})$ is well-known, see e.g.\ \cite{vershynin2018high}.
\end{proof}

\subsection{Proof of Theorem~\ref{thm:universality_matrix}}\label{subsec_app:proof_universality_matrix}

\noindent
As we have seen, it suffices to prove Lemmas~\ref{lemma:domination} and \ref{lemma:pointwise_limit}.
We introduce some additional notations. 
\begin{itemize}
    \item For any $\mu \in [n]$, we denote
    \begin{equation}\label{eq:def_Gibbs_mu}
        \langle \cdot \rangle_\mu \coloneqq \frac{\int P_0(\rd S) \, e^{-\sum_{\nu(\neq \mu)} \phi[\Tr(U_\nu(t) S)]} \big(\cdot\big)}{\int P_0(\rd S) \, e^{-\sum_{\nu(\neq \mu)} \phi[\Tr(U_\nu(t) S)]}}. 
    \end{equation}
    We do not write explicitly the $t$ dependency of this average as it will be clear from context.
    \item For any $\mu$, we denote $\EE_{(\mu)}$ the expectation conditioned on $\{G_\mu, W_\mu\}$, i.e.\ the expectation over 
    $\{G_\nu, W_\nu\}_{\nu (\neq \mu)}$. Note that this notation is different from \cite{montanari2022universality}, 
    for which $\EE_{(\mu)}$ was the expectation with respect to $(G_\mu, W_\mu)$.
    On the other hand, we denote without parenthesis the expectation with respect to these variables: e.g.\ $\EE_{(1)}$ is conditioned on $\{G_1, W_1\}$, 
    but $\EE_{G_1, W_1}$ is the expectation w.r.t.\ $G_1, W_1$.
\end{itemize}

\subsubsection{Proof of Lemma~\ref{lemma:domination}}\label{subsubsec:proof_lemma_domination}

We fix $t \in (0, \pi/2)$, and we start from eq.~\eqref{eq:ee_derivative_psi}, which we can rewrite using eq.~\eqref{eq:def_Gibbs_mu} as:
\begin{equation*}
    \EE \, \frac{\partial \psi[F_d(U(t))]}{\partial t} = - \frac{n}{d^2} \EE\Bigg[\psi'[F_d(U(t))] \frac{\Big \langle e^{-\phi(\Tr[U_1(t) S])} \Big(\Tr[S \tU_1(t)] \, \phi'(\Tr[U_1(t) S])\Big) \Big \rangle_1}{\Big \langle e^{-\phi(\Tr[U_1(t) S])} \Big \rangle_1}\Bigg].
\end{equation*}
Since $\|\psi'\|_\infty < \infty$ and $n/d^2 \leq \alpha_2$, using the triangular inequality the proof of Lemma~\ref{lemma:domination} is complete if one can show the following bound, which will also be useful afterwards:
\begin{lemma}\label{lemma:boundedness_single_matrix}
    \noindent
    There exists a universal constant $C > 0$ such that: 
    \begin{equation*}
       \sup_{d \geq 1} \sup_{t \in (0, \pi/2)} \sup_{\{W_\mu, G_\mu\}_{\mu=2}^n} \EE_{W_1, G_1} \frac{\Big \langle e^{-\phi(\Tr[U_1 S])} \Big|\Tr[S \tU_1] \, \phi'(\Tr[U_1 S]) \Big|\Big \rangle_1}{\Big \langle e^{-\phi(\Tr[U_1 S])} \Big \rangle_1}
          \leq C.
    \end{equation*}
\end{lemma}
\begin{proof}[Proof of Lemma~\ref{lemma:boundedness_single_matrix}]
Note that
\begin{align}\label{eq:domination_dpsi_dt}
    \nonumber
    &\EE_{W_1, G_1} \frac{\Big \langle e^{-\phi(\Tr[U_1 S])} \Big| \Tr[S \tU_1] \, \phi'(\Tr[U_1 S])\Big|\Big \rangle_1}{\Big \langle e^{-\phi(\Tr[U_1 S])} \Big \rangle_1} \\
    \nonumber
    &\leq  \|\phi'\|_\infty \EE_{W_1, G_1} \Bigg[\frac{\Big\langle e^{-\phi(\Tr[U_1 S])} \Big|\Tr[S \tU_1(t)]\Big|\Big\rangle_1 }{\Big\langle e^{-\phi[\Tr(U_1 S)]} \Big \rangle_1}\Bigg], \\ 
    &\leq e^{\|\phi\|_\infty} \|\phi'\|_\infty \EE_{W_1, G_1} \Big[\Big\langle \Big|\Tr[S \tU_1(t)]\Big|\Big\rangle_1\Big],
\end{align}
in which we used the positivity and boundedness of $\phi$ in the last inequality.
To control the last term in eq.~\eqref{eq:domination_dpsi_dt} we write:
\begin{align}\label{eq:bound_first_moment_1}
    \nonumber
    \EE_{G_1, W_1} \Big[\Big\langle \Big|\Tr[S \tU_1(t)]\Big|\Big\rangle_1\Big] 
    &\aeq \Big\langle \EE_{G_1, W_1}\Big|\Tr[S \tU_1(t)]\Big| \Big\rangle_1, \\
    \nonumber
    &\bleq \Big\langle \Big\{\EE_{G_1, W_1}\Big(\Tr[S \tU_1(t)]^2\Big) \Big\}^{1/2} \Big\rangle_1, \\
    &\cleq \sqrt{2} \Big\langle \Big\{ \frac{1}{d}\Tr[S^2] \Big\}^{1/2} \Big\rangle_1, \\
    &\dleq \sqrt{2} C_0,
\end{align}
where $(\rm a)$ uses that $\langle \cdot \rangle_1$ is independent of $\{W_1,G_1\}$, $(\rm b)$ is from the Cauchy-Schwarz inequality, in $(\rm c)$ we use the hypothesis on the first two moments of $\rho$ matching the ones of $\mathrm{GOE}(d)$, 
and in $(\rm d)$ that $\supp(P_0) \subseteq B_2(C \sqrt{d})$.
\end{proof}

\subsubsection{Proof of Lemma~\ref{lemma:pointwise_limit}}

We fix $t \in (0, \pi/2)$ for the rest of the proof, and we write $U, \tU$ for $U(t), \tU(t)$.
We follow the ideas of Appendix~A.3 of \cite{montanari2022universality}, and start again from eq.~\eqref{eq:ee_derivative_psi}:
\begin{align}
    \label{eq:def_I1_I2}
    &\Bigg|\EE \, \frac{\partial \psi[F_d(U(t))]}{\partial t}\Bigg| \leq \alpha_2(I_1 + I_2),
\end{align}
with:
\begin{align*}
    I_1 &\!= \Bigg|\EE\Bigg[\left\{\psi'[F_d(U)] - \psi'[F_d(U^{(1)})]\right\} \frac{\int P_0(\rd S) \, e^{-\sum_\nu \phi(\Tr[U_\nu S])} \Big(\Tr[S \tU_1] \, \phi'(\Tr[U_1 S])\Big)}{\int P_0(\rd S) \, e^{-\sum_\nu \phi(\Tr[U_\nu S])}}\Bigg] \Bigg|, \\ 
    I_2 &\!= \Bigg|\EE\Bigg[\psi'[F_d(U^{(1)})] \frac{\int P_0(\rd S) \, e^{-\sum_\nu \phi(\Tr[U_\nu S])} \Big(\Tr[S \tU_1] \, \phi'(\Tr[U_1 S])\Big)}{\int P_0(\rd S) \, e^{-\sum_\nu \phi(\Tr[U_\nu S])}}\Bigg] \Bigg|,
\end{align*}
with $U^{(\mu)}$ obtained from $U$ by setting $U_\mu = 0$.
We show successively $I_1 \to 0$ and $I_2 \to 0$.
Since $\psi'$ is assumed to be Lipschitz, we have: 
\begin{align*}
    |\psi'[F_d(U)] - \psi'[F_d(U^{(1)})]|
    &\leq \frac{\|\psi'\|_\Lip}{d^2} \Bigg|\log \frac{\int P_0(\rd S) e^{-\sum_{\nu=1}^n \phi(\Tr[U_\nu S])}}{\int P_0(\rd S) e^{-\sum_{\nu=2}^n \phi(\Tr[U_\nu S])}}  \Bigg|, \\ 
    &\leq \frac{\|\psi'\|_\Lip}{d^2} \Big|\log \Big\langle e^{-\phi(\Tr[U_1 S])}\Big\rangle_1  \Big|, \\ 
    &\leq -\frac{\|\psi'\|_\Lip}{d^2} \log \Big\langle e^{-\phi(\Tr[U_1 S])}\Big\rangle_1, \\ 
    &\leq \frac{\|\psi'\|_\Lip \|\phi\|_\infty}{d^2}.
\end{align*}
Therefore,
\begin{equation}\label{eq:bound_I1}
    I_1 \leq \mcO_d\Bigg(\frac{1}{d^2}\Bigg) \times \EE \Bigg|\frac{\int P_0(\rd S) \, e^{-\sum_\nu \phi(\Tr[U_\nu S])} \Big(\Tr[S \tU_1] \, \phi'(\Tr[U_1 S])\Big)}{\int P_0(\rd S) \, e^{-\sum_\nu \phi(\Tr[U_\nu S])}} \Bigg|.
\end{equation}
The second term in eq.~\eqref{eq:bound_I1} is bounded by Lemma~\ref{lemma:boundedness_single_matrix}, uniformly in $t$.
Therefore, we reach that $I_1 \to 0$ as $d \to \infty$.

\myskip 
We now tackle $I_2$.
Note that since $U^{(1)}$ is independent of $U_1$, we can rewrite it as:
\begin{align}\label{eq:bound_I2_1}
    \nonumber
    I_2 &= \Bigg|\EE_{(1)} \Bigg[\psi'[F_d(U^{(1)})]  \EE_{G_1, W_1}\Bigg[\frac{\Big\langle e^{-\phi(\Tr[U_1 S])} \Big(\Tr[S \tU_1] \, \phi'(\Tr[U_1 S])\Big)\Big\rangle_1}{\Big\langle e^{-\phi(\Tr[U_1 S])} \Big\rangle_1} \Bigg]\Bigg] \Bigg|, \\ 
    &\leq \|\psi'\|_\infty \EE_{(1)}  \Bigg| \Bigg\langle\EE_{G_1, W_1}\Bigg[\frac{e^{-\phi(\Tr[U_1 S])} \Big(\Tr[S \tU_1] \, \phi'(\Tr[U_1 S])\Big)}{\Big\langle e^{-\phi(\Tr[U_1 S])} \Big\rangle_1} \Bigg] \Bigg\rangle_1\Bigg|.
\end{align}
We focus on bounding the right-hand side of eq.~\eqref{eq:bound_I2_1}. 
We will show the following lemma:
\begin{lemma}\label{lemma:bound_term_I2}
   \noindent Uniformly over all $t \in [0, \pi / 2]$ and $\{W_\mu, G_\mu\}_{\mu=2}^n$, and under the hypotheses of Theorem~\ref{thm:universality_matrix}:
   \begin{equation}\label{eq:lemma_bound_term_I2}
        \lim_{d \to \infty} \Bigg\langle \EE_{G_1, W_1}\Bigg[\frac{e^{-\phi(\Tr[U_1 S])} \Big(\Tr[S \tU_1] \, \phi'(\Tr[U_1 S])\Big)}{\Big\langle e^{-\phi(\Tr[U_1 S])} \Big\rangle_1} \Bigg] \Bigg\rangle_1
        = 0.
   \end{equation}
\end{lemma}
\noindent
One directly concludes that $I_2 \to 0$ from using the dominated convergence theorem in eq.~\eqref{eq:bound_I2_1} (the pointwise limit is given by Lemma~\ref{lemma:bound_term_I2} and the domination hypothesis
by Lemma~\ref{lemma:boundedness_single_matrix}). 
This ends the proof of Lemma~\ref{lemma:pointwise_limit}.

\myskip
We thus focus on the proof of Lemma~\ref{lemma:bound_term_I2}.
Following \cite{montanari2022universality}, the sketch of the proof is the following:
\begin{itemize}
    \item[$(i)$] Show that the denominator appearing in eq.~\eqref{eq:lemma_bound_term_I2} can be 
    moved to the numerator by using the expansion of $1/x$ in power series around $1$. This transforms the quantity to control 
    to a sum of terms of the type $\langle \EE_{G_1, W_1} [f(S_1, \cdots, S_k)] \rangle_1$, 
    with $S_1, \cdots, S_k$ independent samples under $\langle \cdot \rangle_1$.
    \item[$(ii)$] Extend the one-dimensional CLT of Definition~\ref{def:one_dimensional_CLT} to $k$-dimensional projections 
    of $G$ and $W$ (with $k = \mcO_d(1)$), and to square-integrable locally-Lipschitz functions. 
    This allows to apply it to the terms appearing in $(i)$, 
    and write (uniformly in $S_1, \cdots, S_k$) that $\EE_{G_1, W_1} [f(S_1, \cdots, S_k)] \simeq \EE_{G_1, \widetilde{G}_1} [f(S_1, \cdots, S_k)]$, with $\widetilde{G}_1$ an independent $\mathrm{GOE}(d)$ matrix. 
    \item[$(iii)$] For the case of Gaussian matrices, as explained above we have $\tU_1$ independent of $U_1$. Using the form of the function $f$ 
    that appears in eq.~\eqref{eq:lemma_bound_term_I2} this implies that $\EE_{G_1, \widetilde{G}_1} [f(S_1, \cdots, S_k)] = 0$ and concludes the proof.
\end{itemize}
Let us perform this strategy in detail. The following lemma is proven in Section~\ref{subsec_app:additional_proofs}.
\begin{lemma}[Polynomial approximation to the fraction --]\label{lemma:poly_approx}
    \noindent 
    For all $\delta > 0$, there exists a real polynomial $Q$ (depending only on $\delta$) such that for all $d \geq 1$, all $t \in (0, \pi/2)$ and 
    all $\{W_\mu, G_\mu\}_{\mu=2}^n$: 
    \begin{align*}
       &\Bigg| \Bigg\langle \EE_{G_1, W_1}\Bigg[\frac{e^{-\phi(\Tr[U_1 S])} \Big(\Tr[S \tU_1] \, \phi'(\Tr[U_1 S])\Big)}{\Big\langle e^{-\phi(\Tr[U_1 S])} \Big\rangle_1} \Bigg] \Bigg\rangle_1 \Bigg| \\
       &\leq 
       \Big| \EE_{G_1, W_1}\Big\{\Big\langle e^{-\phi(\Tr[U_1 S])} \Big(\Tr[S \tU_1] \, \phi'(\Tr[U_1 S])\Big) \Big\rangle_1Q\Big(\Big\langle e^{-\phi(\Tr[U_1 S])} \Big\rangle_1\Big) \Big\} \Big| 
        + \delta.
    \end{align*}
\end{lemma}
\noindent
We fix $\delta > 0$, and denote $Q(X) = \sum_{k=0}^K a_k X^k$ the polynomial of Lemma~\ref{lemma:poly_approx}. 
Therefore, uniformly in $d$, $t$, and $\{W_\mu, G_\mu\}$:
\begin{align}\label{eq:poly_expansion_application}
    \nonumber
    &\Bigg| \Bigg\langle \EE_{G_1, W_1}\Bigg[\frac{e^{-\phi(\Tr[U_1 S])} \Big(\Tr[S \tU_1] \, \phi'(\Tr[U_1 S])\Big)}{\Big\langle e^{-\phi(\Tr[U_1 S])} \Big\rangle_1} \Bigg] \Bigg\rangle_1 \Bigg| \\
    &\leq \sum_{k=0}^K |a_k| 
       \Big| \Big \langle \EE_{G_1, W_1}\Big\{e^{-\sum_{a=0}^k\phi(\Tr[U_1 S_a])} \Big(\Tr[S_0 \tU_1] \, \phi'(\Tr[U_1 S_0])\Big) \Big\} \Big\rangle_1 \Big|
    + \delta,
\end{align}
with $\{S_a\}_{a=0}^k$ i.i.d.\ samples from $\langle \cdot \rangle_1$.
We then extend the one-dimensional CLT of Definition~\ref{def:one_dimensional_CLT} to finite-dimensional projections, similarly to Lemmas~29 and 30 of \cite{montanari2022universality}.
The proof of this lemma is deferred to Section~\ref{subsec_app:additional_proofs}. 
\begin{lemma}[Extension of the CLT to finite-dimensional projections --]\label{lemma:finite_dim_clt}
    \noindent 
    Let $R \geq 1$ an integer, and $\tG \sim \mathrm{GOE}(d)$, independent of everything else.
    Let $\varphi : \bbR^{2R} \to \bbR$ a locally Lipschitz function such that for both $X \in \{W, \tG\}$:
    \begin{equation}\label{eq:square_integrable_phi}
        \sup_{d \geq 1} \sup_{\{W_\mu, G_\mu\}_{\mu=2}^n} \sup_{t \in (0, \pi/2)} \Big\langle\EE_{X, G}\Big[\varphi\Big(\{\Tr(X S_a)\}_{a=1}^R, \{\Tr(G S_a)\}_{a=1}^R\Big)^2\Big] \Big\rangle_1 < \infty.
    \end{equation}
    It is understood there that $\{S_a\} \iid \langle \cdot \rangle_1$.
    Then 
    \begin{align*}
        \lim_{d \to \infty} \sup_{\{W_\mu, G_\mu\}_{\mu=2}^n} \sup_{t \in (0, \pi/2)} \Big\langle\Big| &\EE_{W,G} \, \varphi(\{\Tr(W S_a)\}, \{\Tr(G S_a)\}) \\ 
        & - \EE_{\tG,G} \, \varphi(\{\Tr(\tG S_a)\}, \{\Tr(G S_a)\})\Big| \Big \rangle_1
        = 0.
    \end{align*}
\end{lemma}
\noindent
We wish to apply Lemma~\ref{lemma:finite_dim_clt} to eq.~\eqref{eq:poly_expansion_application}, i.e.\ 
to
\begin{equation*}
    \varphi(\{\Tr(W_1 S_a)\}, \{\Tr(G_1 S_a)\}) \coloneqq e^{-\sum_{a=0}^k\phi(\Tr[U_1 S_a])} \Big(\Tr[S_0 \tU_1] \, \phi'(\Tr[U_1 S_0])\Big).
\end{equation*}
$\varphi$ is locally Lipschitz by our hypotheses on $\phi$.
Moreover, note that for $X \in \{W, \tG\}$, and $\{S^a\} \in \supp(P_0)$ (using that $W$ has the same two first moments as $\tG$):
\begin{equation*}
    \EE_{X, G}\Big[\varphi\Big(\{\Tr(X S_a)\}, \{\Tr(G S_a)\}\Big)^2\Big] 
    \leq 2 \|\phi'\|_\infty^2 \Tr[S_0^2]/d \leq 2 C_0^2 \|\phi'\|_\infty^2 < \infty.
\end{equation*}
This allows to apply Lemma~\ref{lemma:finite_dim_clt} in eq.~\eqref{eq:poly_expansion_application}, 
and to reach that, uniformly in $\{W_\mu, G_\mu\}_{\mu=2}^n$ and $t \in (0, \pi/2)$ we have:
\begin{align*}
    &\limsup_{d \to \infty} \Bigg| \Bigg\langle \EE_{G_1, W_1}\Bigg[\frac{e^{-\phi(\Tr[U_1 S])} \Big(\Tr[S \tU_1] \, \phi'(\Tr[U_1 S])\Big)}{\Big\langle e^{-\phi(\Tr[U_1 S])} \Big\rangle_1} \Bigg] \Bigg\rangle_1 \Bigg| \\
    & \leq \delta + \sum_{k=0}^K |a_k| 
    \limsup_{d \to \infty} \Big\langle \Big| \EE_{G_1,W_1} \, \varphi(\{\Tr(W_1 S_a)\}, \{\Tr(G_1 S_a)\})\Big| \Big \rangle_1, \\
    & \leq \delta 
    + \sum_{k=0}^K |a_k| \limsup_{d \to \infty} \Big\langle \Big| \EE_{G_1,\tG_1} \, \varphi(\{\Tr(\tG_1 S_a)\}, \{\Tr(G_1 S_a)\})\Big| \Big \rangle_1
    \\ 
    &+ \sum_{k=0}^K |a_k| \limsup_{d \to \infty} \Big\langle \Big|\EE_{G_1,W_1} \, \varphi(\{\Tr(W_1 S_a)\}, \{\Tr(G_1 S_a)\}) \\ 
    & \hspace{4cm}- \EE_{G_1,\tG_1} \, \varphi(\{\Tr(\tG_1 S_a)\}, \{\Tr(G_1 S_a)\})\Big| \Big \rangle_1, \\
    &\aleq
    \delta + \sum_{k=0}^K |a_k|
    \limsup_{d \to \infty} 
       \Big| \Big \langle \EE_{G_1, \tG_1}\Big\{e^{-\sum_{a=0}^k\phi(\Tr[V_1 S_a])} \Big(\Tr[S_0 \tV_1] \, \phi'(\Tr[V_1 S_0])\Big) \Big\} \Big\rangle_1 \Big|
,
\end{align*}
where we used Lemma~\ref{lemma:finite_dim_clt} in $(\rm a)$, and
with $V_1 = \cos(t) \tG_1 + \sin(t) G_1$, and $\tV_1 = -\sin(t) \tG_1 + \cos(t) G_1$.
Since $G_1, \tG_1$ are gaussians, so are $V_1$ and $\tV_1$, and one verifies easily that they are independent since their covariance is zero.
Therefore, we have:
\begin{align*}
    &\EE_{G_1, \tG_1}\Big\{e^{-\sum_{a=0}^k\phi(\Tr[V_1 S_a])} \Big(\Tr[S_0 \tV_1] \, \phi'(\Tr[V_1 S_0])\Big) \Big\} \\
    &= 
    \EE_{V_1}\Big\{e^{-\sum_{a=0}^k\phi(\Tr[V_1 S_a])} \Big(\underbrace{\Big[\EE_{\tV_1}\Tr[S_0 \tV_1]\Big]}_{=0} \, \phi'(\Tr[V_1 S_0])\Big) \Big\} = 0.
\end{align*}
Thus, we reach, for any $\delta > 0$:
\begin{equation*}
    \limsup_{d \to \infty} \sup_{t \in (0, \pi/2)} \sup_{\{W_\mu, G_\mu\}_{\mu=2}^n} \Bigg| \Bigg\langle \EE_{G_1, W_1}\Bigg[\frac{e^{-\phi(\Tr[U_1 S])} \Big(\Tr[S \tU_1] \, \phi'(\Tr[U_1 S])\Big)}{\Big\langle e^{-\phi(\Tr[U_1 S])} \Big\rangle_1} \Bigg] \Bigg\rangle_1 \Bigg| 
    \leq \delta.
\end{equation*}
Letting $\delta \to 0$ finishes the proof of Lemma~\ref{lemma:bound_term_I2}. $\qed$

\subsection{Additional proofs}\label{subsec_app:additional_proofs}

\subsubsection{Proof of Lemma~\ref{lemma:poly_approx}}

We use the power expansion of $x \mapsto 1/x$ around $x = 1$, defining, for $M \geq 1$:
\begin{equation*}
    Q_M(x) \coloneqq \sum_{k=0}^M (1-x)^k, \hspace{1cm} R_M(x) \coloneqq \frac{1}{x} - Q_M(x).
\end{equation*}
We make use of Lemma~27 of \cite{montanari2022universality}:
\begin{lemma}[\cite{montanari2022universality}]\label{lemma:27_montanari}
    \noindent
    For any integer $M \geq 1$ we have 
    \begin{itemize}
        \item For all $x \neq 0$, $R_M(x) = (1-x)^{M+1}/x$.
        \item $x \mapsto R_M(x)^2$ is convex on $(0, 1]$.
        \item For any $s \in (0,1)$ and $\eta > 0$, there exists $M \geq 1$ such that 
        $\sup_{t \in [s, 1]} |R_M(t)| \leq \eta$.
    \end{itemize} 
\end{lemma}
\noindent
Since $e^{-\|\phi\|_\infty} \leq e^{-\phi} \leq 1$, we have that for all $\eta > 0$ there exists $M_\eta \geq 1$ 
such that for all $M \geq M_\eta$, all matrices $\{W_\mu, G_\mu\}_{\mu=1}^n$, all $t \in (0, \pi/2)$: 
\begin{equation}\label{eq:bound_RM}
    \Big| R_M\Big(\langle e^{-\phi(\Tr[S U_1])} \rangle_1\Big) \Big| \leq \eta.
\end{equation}
Thus, since $1/x = Q_M(x) + R_M(x)$:
\begin{align*}
&\Bigg| \Bigg\langle \EE_{G_1, W_1}\Bigg[\frac{e^{-\phi(\Tr[U_1 S])} \Big(\Tr[S \tU_1] \, \phi'(\Tr[U_1 S])\Big)}{\Big\langle e^{-\phi(\Tr[U_1 S])} \Big\rangle_1} \Bigg] \Bigg\rangle_1 \Bigg| \\
&\leq \Big| \Big\langle \EE_{G_1, W_1}\Big[e^{-\phi(\Tr[U_1 S])} \Big(\Tr[S \tU_1] \, \phi'(\Tr[U_1 S])\Big) Q_M\Big(\Big\langle e^{-\phi(\Tr[U_1 S])} \Big\rangle_1\Big) \Big] \Big\rangle_1 \Big| \\
&\qquad + \Big| \Big\langle \EE_{G_1, W_1}\Big[e^{-\phi(\Tr[U_1 S])} \Big(\Tr[S \tU_1] \, \phi'(\Tr[U_1 S])\Big) R_M\Big(\Big\langle e^{-\phi(\Tr[U_1 S])} \Big\rangle_1\Big) \Big] \Big\rangle_1 \Big|, \\
&\aleq \Big| \Big\langle \EE_{G_1, W_1}\Big[e^{-\phi(\Tr[U_1 S])} \Big(\Tr[S \tU_1] \, \phi'(\Tr[U_1 S])\Big) Q_M\Big(\Big\langle e^{-\phi(\Tr[U_1 S])} \Big\rangle_1\Big) \Big] \Big\rangle_1 \Big| \\
&\qquad+ \eta \Big\langle \EE_{G_1, W_1}\Big[e^{-\phi(\Tr[U_1 S])} \Big| \Big(\Tr[S \tU_1] \, \phi'(\Tr[U_1 S])\Big)  \Big|\Big] \Big\rangle_1, \\
&\leq \Big| \Big\langle \EE_{G_1, W_1}\Big[e^{-\phi(\Tr[U_1 S])} \Big(\Tr[S \tU_1] \, \phi'(\Tr[U_1 S])\Big) Q_M\Big(\Big\langle e^{-\phi(\Tr[U_1 S])} \Big\rangle_1\Big) \Big] \Big\rangle_1 \Big| \\ 
& \hspace{2cm}+ \eta \|\phi'\|_\infty \langle \EE_{G_1, W_1} |\Tr[S \tU_1]| \rangle_1,\\
&\bleq \Big| \Big\langle \EE_{G_1, W_1}\Big[e^{-\phi(\Tr[U_1 S])} \Big(\Tr[S \tU_1] \, \phi'(\Tr[U_1 S])\Big) Q_M\Big(\Big\langle e^{-\phi(\Tr[U_1 S])} \Big\rangle_1\Big) \Big] \Big\rangle_1 \Big| + C \eta \|\phi'\|_\infty,
\end{align*}
using eq.~\eqref{eq:bound_RM} in $(\rm a)$, and the Cauchy-Schwarz inequality (cf.\ the proof of Lemma~\ref{lemma:boundedness_single_matrix}) in $(\rm b)$. 
We emphasize that this bound is uniform in $t$ and $\{W_\mu, G_\mu\}$.
Choosing $\eta = \delta / (C \|\phi'\|_\infty)$ ends the proof of Lemma~\ref{lemma:poly_approx}.

\subsection{Proof of Lemma~\ref{lemma:finite_dim_clt}}\label{subsec_app:proof_finite_dim_clt}

\noindent
We first prove that the conclusion of Lemma~\ref{lemma:finite_dim_clt} holds 
uniformly over all $S_1, \cdots, S_R \in A_d$ when the function $\varphi$ is assumed to be continuous with compact support: 
\begin{lemma}\label{lemma:finite_dim_clt_compact_support}
    \noindent
    Let $R \in \bbN^\star$, and $\tG \sim \mathrm{GOE}(d)$, independent of everything else.
    Let $\varphi : \bbR^{2R} \to \bbR$ be Lipschitz and compactly supported. 
    Recall that $W \sim \rho$, a measure who is assumed to satisfy a one-dimensional CLT with respect to a set $A_d \subseteq \mcS_d$, 
    see Definition~\ref{def:one_dimensional_CLT}. 
    We assume that $A_d$ is convex and symmetric.
    Then:
    \begin{align*}
        \lim_{d \to \infty} \sup_{S_1, \cdots, S_R \in A_d} \Big| &\EE_{W,G} \, \varphi\Big(\{\Tr(W S_a)\}, \{\Tr(G S_a)\}\Big) \\ 
        &\hspace{1cm}- \EE_{\tG,G} \, \varphi\Big(\{\Tr(\tG S_a)\}, \{\Tr(G S_a)\}\Big)\Big| = 0.
    \end{align*}
\end{lemma}
\begin{proof}[Proof of Lemma~\ref{lemma:finite_dim_clt_compact_support}]
Recall that we use the matrix flattening function of eq.~\eqref{eq:def_flattening}.
Let us denote:
\begin{equation}\label{eq:notations_H_h_v}
    \begin{dcases} 
    H &\coloneqq 
    \begin{pmatrix}
        \flatt(S^1) & 0 & \flatt(S_2) & 0 & \cdots & \flatt(S^R) & 0 \\ 
        0 & \flatt(S^1) & 0 & \flatt(S_2) & \cdots & 0 & \flatt(S^R)
    \end{pmatrix} \in \bbR^{d(d+1) \times 2 R}, \\
    v &\coloneqq (\flatt(W)^\T, \flatt(G)^\T)^\T \in \bbR^{d(d+1)}, \\
    h &\coloneqq (\flatt(\tG)^\T, \flatt(G)^\T)^\T \in \bbR^{d(d+1)}.
    \end{dcases}
\end{equation}
Using these notations, we have: 
\begin{equation}\label{eq:notations_H_h_v_2}
    \begin{dcases} 
    \Big(\{\Tr(W S_a)\}_{a=1}^R, \{\Tr(G S_a)\}_{a=1}^R\Big) &= H^\T v \in \bbR^{2R}, \\ 
    \Big(\{\Tr(\tG S_a)\}_{a=1}^R, \{\Tr(G S_a)\}_{a=1}^R\Big) &= H^\T h \in \bbR^{2R}.
    \end{dcases}
\end{equation}
We add a small Gaussian noise to help us deal with characteristic functions later on. 
Let $\delta > 0$, and $Z \sim \mcN(0, \delta^2 \Id_{2R})$.
For all $S_1, \cdots, S_R$ we have:
\begin{align}
    \label{eq:adding_small_gauss_noise}
    \nonumber
    |\EE \, \varphi(H^\T v) - \EE \, \varphi(H^\T h)|
    &\leq 
    |\EE \, \varphi(H^\T v) - \EE \, \varphi(H^\T v + Z)| +
    |\EE \, \varphi(H^\T h) - \EE \, \varphi(H^\T h + Z)| \\ 
    \nonumber
    & \hspace{4.15cm}+
    |\EE \, \varphi(H^\T v + Z) - \EE \, \varphi(H^\T h + Z)|, \\ 
    \nonumber
    &\leq 2 \|\varphi\|_L \EE \|Z\|_2 + 
    |\EE \, \varphi(H^\T v + Z) - \EE \, \varphi(H^\T h + Z)|, \\ 
    &\leq C R^{1/2} \|\varphi\|_L \delta + 
    |\EE \, \varphi(H^\T v + Z) - \EE \, \varphi(H^\T h + Z)|.
\end{align}
We now control the last term of eq.~\eqref{eq:adding_small_gauss_noise}.
For $X \in \bbR^{2R}$ a random variable, we define its characteristic function as 
$\phi_X(u) \coloneqq \EE \, e^{-i u^\T X}$.
We have then
\begin{align*}
\frac{1}{(2\pi)^{2R}}\int_{\bbR^{2R}} \varphi(t) \int_{\bbR^{2R}}  e^{i t^\T u - \frac{\delta^2}{2} \|u\|^2} \phi_X(u) \rd u \, \rd t
&= \frac{1}{(2\pi \delta^2)^{R}} \EE_X \int_{\bbR^{2R}} \varphi(t) e^{-\frac{\|t-X\|^2}{2 \delta^2}} \, \rd t, \\ 
&= \EE \, \varphi(X + Z).
\end{align*}
Coming back to eq.~\eqref{eq:adding_small_gauss_noise} we get:
\begin{align}
    \label{eq:bound_with_gauss_noise}
    \nonumber
    &|\EE \, \varphi(H^\T v + Z) - \EE \, \varphi(H^\T h + Z)|  \\
    \nonumber
    &\leq 
    \frac{1}{(2\pi)^{2R}}\int_{\bbR^{2R}} |\varphi(t)| \Bigg|\int_{\bbR^{2R}}  e^{i t^\T u - \frac{\delta^2}{2} \|u\|^2} [\phi_{H^\T v}(u) - \phi_{H^\T h}(u)] \rd u \, \Bigg| \, \rd t, \\
    &\leq 
    \frac{\|\varphi\|_{L_1}}{(2\pi)^{2R}}\int_{\bbR^{2R}}  e^{- \frac{\delta^2}{2} \|u\|^2} |\phi_{H^\T v}(u) - \phi_{H^\T h}(u)| \rd u.
\end{align}
Here $\|\varphi\|_{L_1} \coloneqq \int |\varphi(t)| \rd t$.
We will show that for any $u \in \bbR^{2R}$: 
\begin{equation}\label{eq:pointwise_cv_charac}
   \lim_{d \to \infty} \sup_{S_1, \cdots, S_R \in A_d}  |\phi_{H^\T v}(u) - \phi_{H^\T h}(u)| = 0.
\end{equation}
Combining eq.~\eqref{eq:pointwise_cv_charac} with the dominated convergence theorem applied in eq.~\eqref{eq:bound_with_gauss_noise}, 
we get: 
\begin{equation*}
    \lim_{d \to \infty} \sup_{S_1, \cdots, S_R \in A_d}|\EE \, \varphi(H^\T v + Z) - \EE \, \varphi(H^\T h + Z)| = 0. 
\end{equation*}
Plugging this back into eq.~\eqref{eq:adding_small_gauss_noise}, we get that for any $\delta > 0$: 
\begin{equation*}
    \limsup_{d \to \infty} \sup_{S_1, \cdots, S_R \in A_d}|\EE \, \varphi(H^\T v) - \EE \, \varphi(H^\T h)| \leq C R^{1/2} \|\varphi\|_L \delta.
\end{equation*}
Letting $\delta \to 0$ ends the proof of Lemma~\ref{lemma:finite_dim_clt_compact_support}. 
There remains to prove eq.~\eqref{eq:pointwise_cv_charac}. 
Let $u = (u^{(1)}, u^{(2)}) \in \bbR^{2R}$, with $u^{(1)}, u^{(2)} \in \bbR^R$, and let us fix $S_1, \cdots, S_R \in A_d$.
We have 
\begin{align}\label{eq:diff_charac_1}
    \nonumber
    &|\phi_{H^\T v}(u) - \phi_{H^\T h}(u)| \\ 
    \nonumber
    &= \Big|\EE \, e^{-i \sum_{a=1}^r u^{(1)}_a \Tr[W S_a] -i \sum_{a=1}^r u^{(2)}_a \Tr[G S_a]} - \EE \, e^{-i \sum_{a=1}^r u^{(1)}_a \Tr[\tG S_a] -i \sum_{a=1}^r u^{(2)}_a \Tr[G S_a]}\Big|, \\
    &\aeq \Big|\EE \, \exp\Big\{-i \, \Tr\Big[W \sum_{a=1}^r u^{(1)}_a S_a \Big]\Big\} - \EE \, \exp\Big\{-i \, \Tr\Big[\tG \sum_{a=1}^r u^{(1)}_a S_a \Big]\Big\}\Big|,
\end{align}
using in $(\rm a)$ that $G$ is independent of $W, \tG$.
We can assume that $u^{(1)} \neq 0$, otherwise the result of eq.~\eqref{eq:pointwise_cv_charac} is clear.
Since $A_d$ is symmetric and convex we have 
\begin{equation*}
    \hS \coloneqq \frac{1}{\|u^{(1)}\|_1}\sum_{a=1}^r u^{(1)}_a S_a \in A_d.
\end{equation*}
Therefore, letting $\varphi_u(x) \coloneqq e^{-i \|u\|_1 x}$, 
we have by eq.~\eqref{eq:diff_charac_1}:
\begin{equation}\label{eq:diff_charac_2}
    |\phi_{H^\T v}(u) - \phi_{H^\T h}(u)| 
    = |\EE \, \varphi_u(\Tr[\hS W]) - \EE \, \varphi_u(\Tr[\hS \tG])|
\end{equation}
Moreover, 
\begin{equation*}
    |\varphi_u(x) - \varphi_u(y)| = \|u\|_1 \Bigg|\int_0^{y-x} e^{i \|u\|_1 t} \rd t\Bigg| \leq \|u\|_1 |y-x|,
\end{equation*}
so that $\|\varphi_u\|_L \leq \|u\|_1$. We can then therefore apply the one-dimensional CLT of Definition~\ref{def:one_dimensional_CLT} 
in eq.~\eqref{eq:diff_charac_2}, and we get that 
\begin{equation*}
    \sup_{S_1, \cdots, S_R \in A_d}|\phi_{H^\T v}(u) - \phi_{H^\T h}(u)| 
    \leq \sup_{S \in A_d}  |\EE \, \varphi_u(\Tr[S W]) - \EE \, \varphi_u(\Tr[S \tG])| \to_{d\to\infty} 0.
\end{equation*}
This ends the proof of eq.~\eqref{eq:pointwise_cv_charac}.
\end{proof}

\myskip
We then deduce the full statement of Lemma~\ref{lemma:finite_dim_clt} by a truncation argument. 

\begin{proof}[End of the proof of Lemma~\ref{lemma:finite_dim_clt}]
    $\varphi$ is now only assumed to be locally Lipschitz and square integrable, in the sense of eq.~\eqref{eq:square_integrable_phi}.
    We take the same notations as in the proof of Lemma~\ref{lemma:finite_dim_clt_compact_support}, 
    see in particular eq.~\eqref{eq:notations_H_h_v}, so that
    \begin{align*}
         &\EE_{W,G} \, \varphi\Big(\{\Tr(W S_a)\}, \{\Tr(G S_a)\}\Big) - \EE_{\tG,G} \, \varphi\Big(\{\Tr(\tG S_a)\}, \{\Tr(G S_a)\}\Big) \\ 
         &= \EE \, \varphi(H^\T v) - \EE \, \varphi(H^\T h).
    \end{align*}
    Let $B > 0$, and let us denote 
    $u_B : \bbR_+ \to [0,1]$ a $\mcC^\infty$ function such that $u_B(x) = 1$ if $x \leq B$ and $u_B(x) = 0$ if $x \geq B+1$. 
    We denote $\varphi_B(z) \coloneqq \varphi(z) u_B(\|z\|)$.
    One can then check easily that $\varphi_B$ is Lipschitz (because $\varphi$ is locally-Lipschitz), 
    and compactly supported.
    Moreover, we have:
    \begin{align}\label{eq:loc_lipschitz_1}
        \nonumber
        &\langle |\EE \, \varphi(H^\T v) - \EE \, \varphi(H^\T h)| \rangle_1 \\
        &\leq 
        \langle |\EE \, \varphi_B(H^\T v) - \EE \, \varphi_B(H^\T h)| \rangle_1 + 
        \sum_{z \in \{h, v\}} \langle \EE \, |\varphi(H^\T z)| (1 - u_B(\|H^\T z\|)) \rangle_1.
    \end{align}
    We now control the different terms in eq.~\eqref{eq:loc_lipschitz_1} successively. 
    Notice that 
    \begin{equation*}
        \langle |\EE \, \varphi_B(H^\T v) - \EE \, \varphi_B(H^\T h)| \rangle_1 \leq \sup_{S_1, \cdots, S_R \in A_d} |\EE \, \varphi_B(H^\T v) - \EE \, \varphi_B(H^\T h)|,
    \end{equation*}
    so that by Lemma~\ref{lemma:finite_dim_clt_compact_support}:
    \begin{equation}\label{eq:loc_lipschitz_2}
        \lim_{d \to \infty} \sup_{\{W_\mu, G_\mu\}_{\mu=2}^n} \sup_{t \in (0, \pi/2)}
        \langle |\EE \, \varphi_B(H^\T v) - \EE \, \varphi_B(H^\T h)| \rangle_1 = 0.
    \end{equation}
    We now tackle the remaining terms in eq.~\eqref{eq:loc_lipschitz_1}.
    Let $z \in \{h, v\}$.
    Using the Cauchy-Schwarz inequality twice we get:
    \begin{align}\label{eq:loc_lipschitz_3_1}
        \nonumber
        \langle \EE \, |\varphi(H^\T z)| (1 - u_B(\|H^\T z\|)) \rangle_1 
        &\leq \langle \EE \, |\varphi(H^\T z)| \indi\{\|H^\T z\|_2 \geq B\} \rangle_1, \\
        \nonumber
        &\leq \langle (\EE_z [\varphi(H^\T z)^2])^{1/2} \bbP_z\{\|H^\T z\|_2 \geq B\}^{1/2} \rangle_1 , \\
        &\leq \langle (\EE_z [\varphi(H^\T z)^2])\rangle_1^{1/2} \, \cdot \, \langle \bbP_z\{\|H^\T z\|_2 \geq B\} \rangle_1^{1/2}. 
    \end{align}
    The first term in eq.~\eqref{eq:loc_lipschitz_3_1} is bounded by the square integrability assumption, cf.\ eq.~\eqref{eq:square_integrable_phi}. 
    To bound the second term, we use Markov's inequality:
    \begin{equation*}
        \langle \bbP_z\{\|H^\T z\|_2 \geq B\} \rangle_1 
        \leq \frac{1}{B^2} \langle \EE_z\, \|H^\T z\|_2^2 \rangle_1.
    \end{equation*}
    From eq.~\eqref{eq:notations_H_h_v_2} and the matching of the first two moments of $\rho$ with $\GOE(d)$, we have:
    \begin{equation*}
        \EE_z\, \|H^\T z\|_2^2 = \frac{4}{d} \sum_{a=1}^R \Tr[S_a^2]. 
    \end{equation*}
    Thus, we get:
    \begin{equation*}
        \langle \bbP_z\{\|H^\T z\|_2 \geq B\} \rangle_1 
        \leq \frac{4R}{B^2} \Bigg\langle \frac{\Tr S^2}{d} \Bigg\rangle_1 \leq \frac{4RC_0}{B^2}.
    \end{equation*}
    All in all, we get:
    \begin{equation}\label{eq:loc_lipschitz_3}
        \sup_{d \geq 1} \sup_{\{W_\mu, G_\mu\}_{\mu=2}^n} \sup_{t \in (0, \pi/2)} \langle \EE \, |\varphi(H^\T z)| (1 - u_B(\|H^\T z\|)) \rangle_1  
        \leq \frac{C(R, \varphi)}{B}.
    \end{equation}
    Combining eqs.~\eqref{eq:loc_lipschitz_2} and eq.~\eqref{eq:loc_lipschitz_3} 
    into eq.~\eqref{eq:loc_lipschitz_1}, and taking $B \to \infty$ after $d \to \infty$, we conclude the proof of Lemma~\ref{lemma:finite_dim_clt}.
\end{proof}

\end{document}